\documentclass[10pt]{article}
\usepackage{graphicx} 
\usepackage[a4paper, left=2.9cm, right=2.9cm, top=2.8cm, bottom=3.3cm]{geometry}
\usepackage{amsfonts}
\usepackage{amsmath}
\usepackage{amsthm}
\usepackage{setspace}
\usepackage{amssymb}
\usepackage{pst-node}
\usepackage{mathrsfs}
\usepackage{tikz}
\usepackage{pgfplots}
\pgfplotsset{compat=1.18}
\usetikzlibrary{arrows.meta, positioning, decorations.pathmorphing}
\usepackage{hyperref}
\colorlet{darkblue}{blue!50!black}
\hypersetup{
	colorlinks,%
	citecolor=darkblue,%
	filecolor=red,%
	linkcolor=darkblue,%
	urlcolor=darkblue,%
	pdfnewwindow=true,%
	pdfstartview={FitH}
}
\usepackage{comment}
\usepackage{tabularx}
\usepackage{pifont}
\usepackage{microtype}
\usepackage{bbm}
\usepackage{appendix}
\usepackage{fancyhdr}
\usetikzlibrary{automata,arrows,positioning,calc}
\usepackage{accents}
\usepackage{enumitem}

\usepackage[framemethod=TikZ]{mdframed}
\usepackage{thmtools}
\usepackage{xcolor}
\usepackage[framemethod=TikZ]{mdframed}
\usepackage{cleveref}
\usepackage{booktabs}
\usepackage{makecell}
\usepackage{titlesec}
\usepackage{graphicx}
\usepackage{stmaryrd}

\titleformat{\section}
  {\normalfont\Large\bfseries}  
  {\thesection}{1em}{}          

\titleformat{\subsection}
  {\normalfont\large\bfseries}
  {\thesubsection}{1em}{}

\colorlet{myblue}{blue!50!white}
\colorlet{myblue2}{blue!30!white}

\makeatletter
\newcommand{\hathat}[1]{%
\begingroup%
  \let\macc@kerna\z@%
  \let\macc@kernb\z@%
  \let\macc@nucleus\@empty%
  \hat{\mathchoice%
    {\raisebox{.2ex}{\vphantom{\ensuremath{\displaystyle #1}}}}%
    {\raisebox{.2ex}{\vphantom{\ensuremath{\textstyle #1}}}}%
    {\raisebox{.16ex}{\vphantom{\ensuremath{\scriptstyle #1}}}}%
    {\raisebox{.14ex}{\vphantom{\ensuremath{\scriptscriptstyle #1}}}}%
    \smash{\hat{#1}}}%
\endgroup%
}
\makeatother

\declaretheoremstyle[
  headfont=\bfseries\color{black},   
  bodyfont=\normalfont,
  headpunct={.},                      
  spaceabove=6pt, spacebelow=6pt,
  mdframed={
    backgroundcolor=myblue!8,        
    linecolor=myblue,                 
    linewidth=2pt,
    topline=false, bottomline=false,
    leftline=false, rightline=false
  }
]{mytheoremstyle}

\declaretheoremstyle[
  headfont=\bfseries\color{black},   
  bodyfont=\normalfont,
  headpunct={.},                      
  spaceabove=6pt, spacebelow=6pt,
  mdframed={
    backgroundcolor=myblue2!20,        
    linecolor=myblue,                 
    linewidth=2pt,
    topline=false, bottomline=false,
    leftline=false, rightline=false
  }
]{mytheoremstyle2}

\declaretheoremstyle[
  headfont=\bfseries\color{black},   
  bodyfont=\normalfont,
  headpunct={.},                     
  spaceabove=0pt, spacebelow=4pt,
  mdframed={
    backgroundcolor=none,            
    linecolor=myblue2,                 
    linewidth=1pt,                   
    topline=false, bottomline=false, 
    rightline=false, leftline=true,  
    innerleftmargin=8pt,             
    skipabove=4pt, skipbelow=4pt     
  }
]{mytheoremstyle3}

\theoremstyle{plain}
\newtheorem{theorem}{Theorem}[section]
\newtheorem{lemma}[theorem]{Lemma}
\newtheorem{proposition}[theorem]{Proposition}

\newtheorem{corollary}[theorem]{Corollary}

\theoremstyle{definition}
\newtheorem{definition}[theorem]{Definition}

\newtheorem{remark}[theorem]{Remark}

\declaretheorem[name=Hypothesis]{hyp}

\numberwithin{equation}{section}

\newcommand{\NN}{\mathbb{N}}
\newcommand{\ZZ}{\mathbb{Z}}
\newcommand{\QQ}{\mathbb{Q}}
\newcommand{\RR}{\mathbb{R}}

\newcommand{\PP}{\mathbb{P}}
\newcommand{\EE}{\mathbb{E}}

\newcommand{\eqdef}{ \overset{\textup{\tiny def.}}{=}}

\newcommand{\dbl}{\llbracket}
\newcommand{\dbr}{\rrbracket}

\newcommand{\notation}[4]{#1 & #2 & #3 & #4 \\}

\title{\Large Large Deviation Principle for the Empirical Measures of Simple Random Walks on $\overline{\ZZ}$}
\author{
Jan-Luka Fatras\thanks{
Univ Toulouse, INSA Toulouse, CNRS, IMT, Toulouse, France. Email: \texttt{jan-luka.fatras@math.univ-toulouse.fr}.}
}

\date{\small \today}

\begin{document}

\maketitle

\vspace{-1em}

\begin{abstract}
\noindent In this article we establish a large deviation principle for the empirical measures of a simple spatially inhomogeneous random walk on $\overline{\ZZ}$, the two-point compactification of $\ZZ$. The classical Donsker--Varadhan framework does not apply, since the random-walk kernel and the topology of $\overline{\ZZ}$ fall outside its standard assumptions. In certain regimes, the resulting rate function is non-convex on its effective domain. We also derive a large deviation principle for empirical means of observables $f:\ZZ \to \RR^d$ admitting limits at $\pm\infty$. This result is optimal in the sense that in general, no large deviation principle holds for the larger class of bounded continuous functions on $\ZZ$.
\end{abstract}


\medskip


\tableofcontents

\newpage

\section{Introduction}\label{section_intro}

The seminal work of Donsker and Varadhan \cite{Donsker_varadhan} initiated the systematic study of large deviation principles (LDPs) for the empirical measures of Markov chains. For the symmetric random walk on $\ZZ$, the classical theory yields a weak LDP, see \cite{Donsker_varadhan_3}, \cite[Lemma 5]{Bryc_Dembo}. However, Hypothesis~H$^*$ of \cite[p.~415]{Donsker_varadhan_3}, which implies positive recurrence, is violated in this setting and it was later shown in \cite[pp.~922-923]{Baxter_Jain_Varadhan} that a full LDP fails for the empirical measures of the symmetric random walk. 

In this article, we revisit this obstruction by enlarging the state space, along a line already suggested in \cite[p.~366]{Varadhan_role_topology}. In Theorem~\ref{theorem_existence_and_identification_of_rate_function}, we establish a full LDP for the empirical measures of nearest-neighbour random walks on $\overline{\ZZ}$, the two-point compactification of $\ZZ$. The introduction of the points at infinity captures the mass that escapes to infinity in the original model, and this phenomenon manifests itself through non-trivial additional contributions to the large deviation rate function. Since $\pm \infty$ are accumulation points of $\overline{\ZZ}$, the resulting process falls outside the classical framework of Markov chains on purely discrete state spaces, as studied in \cite[Lemma 5]{Bryc_Dembo}, \cite{Fayolle}, \cite{Jian_Wu} and \cite{Daures}. It therefore requires a refined large deviation analysis. The establishment of a full LDP on $\overline{\ZZ}$ enables the use of the contraction principle to obtain large deviation results for a certain class of observables defined on $\ZZ$, a feature that is unavailable under the weak LDP on $\ZZ$.

\noindent Throughout, we will use the notation $\NN=\{1,2,\dots\}$ and for $m \leq n$, we let $\dbl m,n \dbr=\{m,m+1,\dots,n\}$. We also use the convention that $0 \times \infty=0$.

\subsection{Main results}\label{subsection_results}

Let $\overline{\ZZ} \eqdef \ZZ \cup \{-\infty,+\infty\}$ be the two-point compactification of $\ZZ$, where the topology is generated by the sets $\dbl -\infty,n\dbr$ and $\dbl n,+\infty\dbr$, with $n \in \ZZ$. With this topology, $\overline{\ZZ}$ is compact and metrisable. Consider the canonical probability space $(\widehat{\Omega},\widehat{\mathcal{F}})=\big(\overline{\ZZ}^{\NN},\mathcal{B}(\overline{\ZZ})^{\otimes \NN} \big)$, where $\mathcal{B}(\overline{\ZZ})$ denotes the Borel $\sigma$-algebra on $\overline{\ZZ}$, and let $(S_n)_{n \in \NN}$ be the coordinate process. For each $m \in \ZZ$, let $\PP_m$ be a probability measure on $(\widehat{\Omega},\widehat{\mathcal{F}})$ such that $\PP_m(S_1=m)=1$, and under which $(S_n)_{n \in \NN}$ is a Markov chain with transition probabilities given by
\begin{equation*}
\PP_m(S_{n+1}=+\infty \, \vert \, S_n=+\infty)=1, \qquad \PP_m(S_{n+1}=-\infty \, \vert \, S_n=-\infty)=1 \,,
\end{equation*}
and for all $k \in \ZZ$,
\begin{equation*}
\PP_m(S_{n+1}=k+1 \, \vert \, S_n=k)=p(k) , \qquad \PP_m(S_{n+1}=k-1 \, \vert \, S_n=k)=1-p(k) \,,
\end{equation*}
where $p : \ZZ \longrightarrow [0,1]$ is called the \textit{transition probability function}. The process $(S_n)_{n \in \NN}$ is a spatially inhomogeneous nearest-neighbour random walk on $\overline{\ZZ}$, with absorbing states $\pm \infty$. Throughout the article we will make the following hypotheses on $p$.

\begin{hyp}\label{hypothesis_on_transition_function_p}
\begin{enumerate}
\item $0<p(k)<1$ for all $k \in \ZZ$,
\item there exists $p_{-},p_{+} \in (0,1)$ such that 
\begin{equation*}\label{eq_def_p_pm}
p_{-} \eqdef \lim \limits_{k \to -\infty} p(k) \qquad \text{and} \qquad p_+ \eqdef \lim \limits_{k \to +\infty} p(k).
\end{equation*}
\end{enumerate}
\end{hyp} 
\noindent Note that the second condition in Hypothesis~\ref{hypothesis_on_transition_function_p} is equivalent to the requirement that $p$ admit a continuous extension from $\ZZ$ to $\overline{\ZZ}$.

Let $(\ell_n)_{n \in \NN}$ be the sequence of empirical measures associated to the random walk,
\begin{equation*}\label{eq_def_empirical_measure}
\ell_n \eqdef \frac{1}{n} \sum_{j=1}^n \delta_{S_j} , \quad n \in \NN.
\end{equation*}
The sequence $(\ell_n)_{n \in \NN}$ is a sequence of random variables in the space of probability measures on $\overline{\ZZ}$, denoted by $\mathcal{P}(\overline{\ZZ})$. We will be interested in showing a LDP for the sequence $(\ell_n)_{n \in \NN}$ in the space $\mathcal{P}(\overline{\ZZ})$ endowed with the topology of weak convergence. We recall that a sequence of random variables $(X_n)_{n \in \NN}$ on a topological space $\mathcal{X}$, endowed with its Borel $\sigma$-algebra, satisfies a (full) LDP with rate function $I:\mathcal{X} \to [0,+\infty]$, where $I$ is lower semicontinuous, if 
\begin{align}
-\inf \limits_{x \in U} I(x) \leq \liminf \limits_{n \to \infty} \frac{1}{n} \log \PP \left( X_n \in U \right) & \qquad \text{for all open sets } U \subseteq \mathcal{X} , \label{eq_def_LDP_lower_bound} \\
-\inf \limits_{x \in F} I(x) \geq \limsup \limits_{n \to \infty} \frac{1}{n} \log \PP \left( X_n \in F \right)   & \qquad  \text{for all closed sets } F \subseteq \mathcal{X}. \label{eq_def_LDP_upper_bound}
\end{align}
If the lower bound~(\ref{eq_def_LDP_lower_bound}) holds for all open sets and the upper bound~(\ref{eq_def_LDP_upper_bound}) holds only for compact sets, rather than for all closed sets, we say that the LDP is \textit{weak} rather than full. The rate function $I$ is said to be \textit{good} if for all $\alpha \geq 0$, the level set $\{ I \leq \alpha\}$ is compact. 

Any measure $\mu \in \mathcal{P}(\overline{\ZZ})$ admits a decomposition of the form
\begin{equation}\label{eq_decomposition_measures_on_Z_bar}
\mu=\sum_{\sigma \in \{-,0,+\}} \alpha_{\sigma} \mu_{\sigma},
\end{equation}
where $\alpha_{\sigma} \geq 0$, $\sum_{\sigma} \alpha_{\sigma}=1$ and 
\begin{align*}
\mu_0 \in \mathcal{P}(\ZZ) \text{,} \qquad \mu_{-}=\delta_{- \infty}  \qquad \text{and} \qquad \mu_{+}=\delta_{+ \infty}.
\end{align*}
If $\mu(\ZZ)>0$, this decomposition is unique. If $\mu(\ZZ)=0$, then $\alpha_0=0$, and the choice of $\mu_0 \in \mathcal{P}(\ZZ)$ is arbitrary. We shall refer to $\mu_0$ as the \textit{central} part of $\mu$. When writing $\mu_0 \in \mathcal{P}(\ZZ)$, we implicitly identify $\mathcal{P}(\ZZ)$ with its natural embedding into $\mathcal{P}(\overline{\ZZ})$, and adopt this convention throughout the article. 

We define the function $I : \mathcal{P}(\overline{\ZZ}) \longrightarrow [0,+\infty]$ by 
\begin{equation}\label{eq_def_rate_fuction}
I(\mu) \eqdef  \alpha_0 I_{\textup{DV}}(\mu_0) + \min \left\{
\alpha_{-}I_{\textup{Cr}}^{p_{-}}(0) + \alpha_+ \inf \limits_{x \in [0,1]} I_{\textup{Cr}}^{p_{+}}(x)  \, ; \,
\alpha_{+}I_{\textup{Cr}}^{p_{+}}(0) + \alpha_{-} \inf \limits_{x \in [-1,0]} I_{\textup{Cr}}^{p_{-}}(x)
\right\},
\end{equation}
where $\mu$ is decomposed as in \eqref{eq_decomposition_measures_on_Z_bar}. The function $I$ is well defined, since when $\mu(\ZZ)=0$, the decomposition of $\mu$ is not unique, but the value of $I(\mu)$ does not depend on the choice of $\mu_0$. Here, $I_{\textup{DV}}$ and $I_{\textup{Cr}}^{p_{\sigma}}$ are classical rate functions defined as follows. The functional $I_{\textup{DV}}$ is the usual Donsker--Varadhan rate function \cite{Donsker_varadhan}, given by 
\begin{align}\label{Donsker_Varadhan_our_case}
\begin{array}{cccc}
I_{\textup{DV}}: & \mathcal{P}(\ZZ) & \longrightarrow & [0,+\infty] \\
& \mu & \longmapsto & \sup \limits_{\substack{(u_k)_{k \in \ZZ} \, : \\ u_k \geq 1 \, , \, \forall k \in \ZZ }} \displaystyle \sum \limits_{k \in \ZZ} \mu(k) \log \left( \frac{u_k}{p(k)u_{k+1}+(1-p(k))u_{k-1}} \right) 
\end{array}.
\end{align}
For $p \in (0,1)$, let $X$ be a biased Rademacher random variable with parameter $p$, that is,
\begin{align*}
\PP(X=1)=p \qquad \text{and} \qquad \PP(X=-1)=1-p,
\end{align*} 
and let $\Lambda_p(\lambda)=\log \EE[e^{\lambda X}]$ denote its cumulant generating function. The corresponding Cramér rate function is then given by
\begin{equation}\label{eq_I_Cr_Fenchel_Legendre}
I_{\textup{Cr}}^{p} \eqdef (\Lambda_p)^*,
\end{equation}
where $f^*$ denotes the Fenchel--Legendre transform of $f$. Using a variational formula for the relative entropy
\cite[Proposition~1.4.2]{Dupuis_Ellis} and the Fenchel--Legendre duality for convex functions \cite[Lemma~4.5.8]{DZ_LDP}, one can show that
\begin{equation*}
I_{\textup{Cr}}^{p}(x)= D_{\textup{KL}}\left(\textup{Rad}\left(\frac{1+x}{2}\right) \, \Big\| \, \textup{Rad}(p)\right), \quad x \in [-1,1].
\end{equation*}
where $D_{\textup{KL}}(\cdot \| \cdot)$ is the Kullback--Leibler divergence and $\textup{Rad}(p)$ is a Rademacher distribution of parameter $p$. In particular, an explicit expression of $I_{\textup{Cr}}^{p}$ is given by
\begin{equation}\label{eq_def_I_Cr_p}
I_{\textup{Cr}}^{p}(x) = \begin{cases}
\frac{1-x}{2} \log \left( \frac{1-x}{2(1-p)} \right) + \frac{1+x}{2} \log \left( \frac{1+x}{2p} \right), & \text{if } x \in [-1,1], \\
+ \infty, & \text{otherwise.}
\end{cases}
\end{equation}
The following theorem is the main result of this article.
\begin{theorem}\label{theorem_existence_and_identification_of_rate_function}
For every $m \in \ZZ$, the sequence of empirical measures $(\ell_n)_{n \in \NN}$ satisfies, under $\PP_m$, a large deviation principle with good rate function $I : \mathcal{P}(\overline{\ZZ}) \longrightarrow [0,+\infty]$ given by Equation~(\ref{eq_def_rate_fuction}).
\end{theorem}

We now briefly discuss the interpretation of the rate function $I$. The term $\alpha_0 I_{\textup{DV}}(\mu_0)$ arises from the analysis of the Markov chain on the discrete state space $\ZZ$, see \cite{Donsker_varadhan_3}, \cite[Lemma~5]{Bryc_Dembo} and \cite[Theorem~1.4]{Daures}. It represents the large deviation cost for the random walk to approximate the central measure $\mu_0$ during a fraction $\alpha_0$ of its time. Turning to the terms involving $I_{\textup{Cr}}^{p_{\sigma}}$, observe that if the empirical measure of a trajectory is close to a measure $\mu$ assigning mass $\alpha_{+}$ to $+\infty$, then the trajectory must spend a proportion $\alpha_{+}$ of its time close to $+ \infty$. Since we are considering a nearest-neighbour random walk, if a trajectory ends close to $-\infty$, then all visits to a neighbourhood of $+\infty$ necessarily occur as \textit{excursions}. By contrast, if the trajectory ends near $+\infty$, the final visit to that neighbourhood is not required to return to its starting point: it is a \textit{meander} rather than necessarily an excursion. The term $\alpha_{+} I_{\textup{Cr}}^{p_+}(0)$ corresponds to the large deviation cost associated with spending a fraction $\alpha_{+}$ of the time in a neighbourhood of $+\infty$ in the form of excursions, whereas the term $\alpha_{+} \inf_{x \in [0,1]} I_{\textup{Cr}}^{p_+}(x)$ represents the cost of spending the same proportion of time in a neighbourhood of $+\infty$, but without imposing the excursion constraint. The variational formulation of $I$ reflects the fact that the large deviation cost is minimised over the possible neighbourhoods in which the trajectory may terminate.

Another expression of the rate function $I$ is, for all $\mu=\sum_{\sigma \in \{-,0,+\}} \alpha_{\sigma} \mu_{\sigma} \in \mathcal{P}(\overline{\ZZ})$,
\begin{equation}\label{eq_def_rate_fuction_other_form_2}
I(\mu) =  \inf_{(x_-,x_+) \in \mathcal{C}} \Big\{ \alpha_0 I_{\textup{DV}}(\mu_0) + 
\alpha_{-}I_{\textup{Cr}}^{p_{-}}(x_-) + \alpha_{+} I_{\textup{Cr}}^{p_{+}}(x_+) \Big\},
\end{equation}
where $\mathcal{C} \eqdef \big\{(x_-,x_+) \in [-1,0] \times [0,1] \, : \, x_-=0 \; \text{or} \; x_+=0 \big\}$. The condition that, for $(x_-,x_+) \in \mathcal{C}$, either $x_-=0$ or $x_+=0$ reflects the constraint that visits to at least one of the two neighbourhoods of infinity must occur exclusively in the form of excursions. Accordingly, $\mathcal{C}$ captures the geometric structure inherent to the trajectories of a nearest-neighbour random walk.

\medskip

For Markov chains on compact state spaces, the large deviation principle for empirical measures is typically governed by the Donsker--Varadhan rate function. However, the standing assumption of strict positivity, stated in \cite[p.~3]{Donsker_varadhan}, is violated by the Markov chain $(S_n)_{n \in \NN}$ considered here. Indeed, starting from $0$, the chain never reaches $+\infty$, and hence is not irreducible. Extensions of the results in \cite{Donsker_varadhan} to non-irreducible and countable Markov chains have been obtained in \cite{Fayolle}, \cite{Jian_Wu}, \cite[Corollary 13.6]{Rassoul} and \cite[Theorem~1.4]{Daures}. In those settings, the rate function coincides with the Donsker--Varadhan functional on certain subsimplices of $\mathcal{P}(\mathcal{X})$, referred to as the set of admissible measures, and is infinite elsewhere. These results, however, do not apply in the present framework, since we do not endow $\overline{\ZZ}$ with the discrete topology: the two points at infinity are accumulation points of the discrete core. As shown by \eqref{eq_def_rate_fuction}, the presence of these two accumulation points alters the structure of the rate function in a non-trivial way.

\begin{remark}\label{remark_other_expression_of_rate_function}
By evaluating the expression of $I_{\textup{Cr}}^{p}$ given in (\ref{eq_def_I_Cr_p}) at 0, we obtain
\begin{equation*}
I_{\textup{Cr}}^{p}(0)=\frac{1}{2} \log \left( \frac{1}{4p(1-p)} \right).
\end{equation*}
In addition, for all $p \in (0,1) $, the rate function $I_{\textup{Cr}}^{p}$ is strictly convex on $[-1,1]$ with a unique minimum reached at $\EE[X]=2p-1$. Hence, 
\begin{align*}
\inf_{x \in [0,1]} I_{\textup{Cr}}^{p}(x)&=
\begin{cases}
\frac{1}{2} \log \left( \frac{1}{4p(1-p)} \right)=I_{\textup{Cr}}^{p}(0), & \text{if } p \leq \frac{1}{2}, \\
0, & \text{otherwise,}
\end{cases} \\
\text{and} \qquad \inf_{x \in [-1,0]} I_{\textup{Cr}}^{p}(x)&=
\begin{cases}
\frac{1}{2} \log \left( \frac{1}{4p(1-p)} \right)=I_{\textup{Cr}}^{p}(0), & \text{if } p \geq \frac{1}{2}, \\
0, & \text{otherwise}.
\end{cases}
\end{align*}
Thus, we obtain the following alternative expression of the rate function $I$,
\begin{align}\label{eq_rate_function_other_form}
I(\mu)=
\begin{cases}
\alpha_0 I_{\textup{DV}}(\mu_0) + \alpha_{+} I_{\textup{Cr}}^{p_+}(0), & \text{if } p_+<\frac{1}{2} \; \text{and} \; p_-<\frac{1}{2}, \\
\alpha_0 I_{\textup{DV}}(\mu_0) + \alpha_{-} I_{\textup{Cr}}^{p_-}(0), & \text{if } p_+ \geq \frac{1}{2} \; \text{and} \; p_- \geq \frac{1}{2}, \\
\alpha_0 I_{\textup{DV}}(\mu_0) + \alpha_{-} I_{\textup{Cr}}^{p_-}(0) + \alpha_{+} I_{\textup{Cr}}^{p_+}(0), & \text{if } p_+ < \frac{1}{2} \; \text{and} \; p_- \geq \frac{1}{2}, \\
\alpha_0 I_{\textup{DV}}(\mu_0) + \min \left\{ \alpha_{+} I_{\textup{Cr}}^{p_+}(0) ,  \alpha_{-} I_{\textup{Cr}}^{p_-}(0) \right\}, & \text{if } p_+ \geq \frac{1}{2} \; \text{and} \; p_- < \frac{1}{2}.
\end{cases}
\end{align}
\end{remark}

\begin{remark}\label{remark_generalisation_to_step_distribution}
We expect the result of Theorem~\ref{theorem_existence_and_identification_of_rate_function} to extend to random walks with compactly supported step distributions. More precisely, let $(\nu_k)_{k \in \ZZ}$ be a family of distributions in $\mathcal{P}(\ZZ)$ such that $\text{Supp}(\nu_k) \subseteq \dbl-M,M \dbr$ for some $M \in \NN$. We assume that $\nu_k \xrightarrow[|k| \to \infty]{\mathcal{L}} \nu_{\pm}$ and that a condition analogous to Hypothesis~\ref{hypothesis_on_transition_function_p} holds. Under these conditions, the spatially inhomogeneous random walk with step distributions $(\nu_k)_{k \in \ZZ}$ is expected to satisfy a LDP at the level of empirical measures with rate function
\begin{align*}
I(\mu) & = \inf_{(x_-,x_+) \in \mathcal{C}} \big\{ \alpha_0 I_{\textup{DV}}(\mu_0)+\alpha_{-}I_{\nu_-}(x_-)+\alpha_{+} I_{\nu_+}(x_+)  \big\} \\
& = \alpha_0 I_{\textup{DV}}(\mu_0)+ \min \left\{\alpha_{-}I_{\nu_-}(0)+ \alpha_{+} \inf_{x \in [0,1]} I_{\nu_+}(x) \, ; \, \alpha_{-} \inf_{x \in [-1,0]} I_{\nu_-}(x) + \alpha_{+} I_{\nu_+}(0) \right\},
\end{align*}
where $I_{\nu}$ denotes the Cramér rate function for a random variable distributed according to $\nu$. Indeed, the local estimates in Section \ref{section_regional_estimates} can be extended to this setting, while the constructions of Sections \ref{section_lower_bound} and \ref{section_upper_bound} can be adapted accordingly, subject to minor technical adjustments. We restrict attention to the simple random walk for clarity of exposition.
\end{remark}

We briefly discuss Hypothesis~\ref{hypothesis_on_transition_function_p}.
The first condition guarantees the irreducibility of the Markov chain on $\ZZ$. Moreover, requiring $p_+,p_-$ to be in $(0,1)$ enforces the property commonly known as \textit{uniform ellipticity}, which is standard in the literature on random walks in random environments, see \cite[p.~258]{Zeitouni_RWRE} and \cite[p.~235]{Rassoul}. Whether uniform ellipticity is strictly necessary remains an open question. For example, it is not required when the measure $\mu_0$ is compactly supported. Extending the result to arbitrary measures, as in \cite[Proposition~2.7.4]{Daures}, appears to be out of reach in the present framework, since our arguments ultimately rely on escaping compact sets in order to approximate $\pm \infty$. This suggests a delicate interplay between the rate of decay of the tails of $\mu_0$ and the rate at which $p$ converges to $0$ or $1$. 

\begin{remark}\label{remark_no_full_LDP_no_hypo_H_no_tightness}
The following statement concerns a random walk with transition probability function $p$ satisfying Hypothesis~\ref{hypothesis_on_transition_function_p}, defined on $\ZZ$ equipped with the discrete topology rather than on $\overline{\ZZ}$. Since $(S_n)_{n \in \NN}$ is an irreducible discrete Markov chain, its associated empirical measure satisfies a weak LDP with the Donsker--Varadhan rate function, see  \cite{Donsker_varadhan_3}, \cite[Lemma 5]{Bryc_Dembo}, \cite{Fayolle}, \cite[Theorem~1.4]{Daures}. However, one can verify that Hypothesis~H$^*$ of \cite[p.~415]{Donsker_varadhan_3} fails in this setting, and that exponential tightness does not hold. Indeed, let $K \subseteq \mathcal{P}(\ZZ)$ be compact, and consider the event $S_k=k-1$ for all $k \in \dbl 1,n \dbr$. For $n$ sufficiently large, the corresponding empirical measure $\ell_n$ cannot belong to $K$. Otherwise, there would exist a subsequence $(\ell_{n_k})_{k \in \NN}$ converging weakly to some probability measure, which is not possible since the mass of $\ell_n$ drifts to infinity on the event $S_k=k-1$ for all $k \in \dbl 1,n \dbr$. Consequently,
\begin{align*}
\liminf_{n \to \infty} \frac{1}{n} \log \PP(\ell_n \notin K) \geq \liminf_{n \to \infty} \frac{1}{n} \log \PP(S_k=k-1 \,,\, \forall k \in \dbl 1,n \dbr) = \log p_+,
\end{align*}
which precludes exponential tightness. In fact, it is shown in \cite[pp.~922-923]{Baxter_Jain_Varadhan} that no full LDP holds for the symmetric random walk. Using the same closed set 
\begin{align*}
\mathcal{C}=\left\{ \frac{1}{2n+1} \sum_{k=0}^{2n} \delta_{k} \right\}_{n \in \NN \cup \{0\}} \subseteq \mathcal{P}(\ZZ),
\end{align*}
one can similarly show that no full LDP holds in $\mathcal{P}(\ZZ)$ as soon as Hypothesis~\ref{hypothesis_on_transition_function_p} is satisfied. Finally, note that Hypothesis~H$^*$ would require $p_+=0$ and $p_-=1$, which is excluded from Hypothesis~\ref{hypothesis_on_transition_function_p}.
\end{remark}

Finally, we present a corollary to Theorem~\ref{theorem_existence_and_identification_of_rate_function}. Consider an observable $f:\ZZ \to \RR^d$ and the associated empirical mean 
\begin{align*}
\ell_n (f) = \frac{1}{n} \sum_{j=1}^n f(S_j),
\end{align*}
where $(S_n)_{n \in \NN}$ is a simple random walk on $\ZZ$. As mentioned previously, it is shown in \cite[pp.~922-923]{Baxter_Jain_Varadhan}, that there is no full LDP at the level of empirical measures and hence one cannot use the contraction principle to obtain a LDP for all continuous and bounded observables on $\ZZ$. However, from Theorem~\ref{theorem_existence_and_identification_of_rate_function}, one obtains a LDP for a subclass of observables $f : \ZZ \to \RR^d$, which, to the best of our knowledge, has not appeared in the literature before.

\begin{corollary}\label{corollary_contraction_on_observables_with_limits}
Let $d \in \NN$ and let $f:\ZZ \to \RR^d$ be such that the limits
\begin{align*}
\lim_{k \to +\infty} f(k) \eqdef f^+ \qquad \text{and} \qquad \lim_{k \to -\infty} f(k) \eqdef f^-
\end{align*}
exist. Then the sequence $\big(\ell_n(f)\big)_{n \in \NN}$ satisfies a LDP with good rate function $I_f : \RR^d \to [0,\infty]$ given by
\begin{equation*}
I_f(x)=\inf \left\{ I(\mu) \; : \; \mu \in \mathcal{P}(\overline{\ZZ}) \; ,\; \alpha_{+} f^+ + \alpha_0 \sum_{k \in \ZZ} f(k) \mu_0(\{k\}) + \alpha_{-} f^- = x \right\}, \quad x \in \RR^d.
\end{equation*}
\end{corollary}

\begin{proof}
Every function $f:\ZZ \to \RR^d$ admitting limits at $\pm \infty$ extends to a continuous and bounded function $\tilde{f} \in \mathcal{C}_b(\overline{\ZZ},\RR^d)$. The result follows by applying the contraction principle \cite[Theorem~4.2.1]{DZ_LDP} to the sequence $(\ell_n)_{n \in \NN}$ with the map
\begin{align*}
\mathcal{P}(\overline{\ZZ}) \ni \mu \mapsto \int \tilde{f}(k) \mathrm{d}\mu(k) \in \RR^d,
\end{align*}
which is continuous on $\mathcal{P}(\overline{\ZZ})$ by definition of the weak topology.
\end{proof}

\noindent The following proposition shows that the class of observables admitting limits at $\pm \infty$ is optimal, in the sense that the result of Corollary~\ref{corollary_contraction_on_observables_with_limits} cannot be extended to all continuous and bounded observables.

\begin{proposition}\label{proposition_no_LDP_for_bounded_observables}
Let $(S_n)_{n \in \NN}$ be a simple random walk on $\ZZ$ with constant transition probability function $p \equiv \overline{p} \neq \frac{1}{2}$. Then there exists a bounded function $f : \ZZ \to \RR$ such that the sequence $\big(\ell_n(f)\big)_{n \in \NN}$ does not satisfy a LDP.
\end{proposition}

\noindent In Appendix~\ref{subsection_appendix_proof_prop_no_LDP_bdd_obser}, a proof of Proposition~\ref{proposition_no_LDP_for_bounded_observables} is given by constructing an observable $f:\ZZ \to \RR$ for which $(\ell_n(f))_{n \in \NN}$ does not satisfy a LDP. The observable $f$ oscillates between 0 and 1, the blocks on which $f=1$ and $f=0$ become progressively longer.

\begin{remark}\label{remark_LDP_for_different_compactification}
In this article, we study the simple random walk on the two-point compactification of $\ZZ$. One could, however, consider other compactifications of $\ZZ$. For instance, using the same methods as presented here, one can establish a LDP for the one-point compactification of $\ZZ$. An application of the contraction principle, as in Corollary~\ref{corollary_contraction_on_observables_with_limits}, then yields a LDP for observables that admit a common limit at $\pm \infty$. This result is weaker than the one obtained in Corollary~\ref{corollary_contraction_on_observables_with_limits}. At the opposite extreme, suppose that an LDP could be established on $\beta \ZZ$, the Stone-Čech compactification of $\ZZ$. In that case, one would obtain a LDP for \textit{all} bounded observables on $\ZZ$. Indeed, by definition \cite[Theorem~38.2]{Munkres}, every bounded function on $\ZZ$ extends uniquely to a continuous function on $\beta \ZZ$. Applying the contraction principle, as in Corollary~\ref{corollary_contraction_on_observables_with_limits}, would then yield the desired LDP for the corresponding empirical averages. Proposition~\ref{proposition_no_LDP_for_bounded_observables} shows that if the transition probability function $p$ is constant at $\overline{p} \neq \frac{1}{2}$, then there exists a bounded observable for which $(\ell_n(f))_{n \in \NN}$ does not satisfy a LDP. Therefore, the sequence of empirical measures cannot satisfy an LDP in $\mathcal{P}(\beta \ZZ)$ with such a transition probability function.
\end{remark}

\begin{remark}\label{remark_on_projective_limit}
It remains unclear whether the result of Proposition \ref{proposition_no_LDP_for_bounded_observables} extends to a general random walk with transition probability function $p$ satisfying Hypothesis \ref{hypothesis_on_transition_function_p}. In particular, we were unable to construct an observable $f$ for which, in the case of the symmetric random walk $(S_n)_{n \in \NN}$, the sequence $(\ell_n(f))_{n \in \NN}$ fails to satisfy a LDP. We attempted to address this question via a projective-limit argument. Suppose that an LDP were to hold for all $d \in \NN$ and every observable $f:\ZZ \to \RR^d$. By \cite[Theorem 4.6.9]{DZ_LDP}, one can lift these finite-dimensional LDPs to an LDP on the algebraic dual of the space of bounded functions on $\ZZ$. One could then hope to restrict this LDP to the subset of probability measures. If such a restriction were valid, it would yield a full LDP for the sequence of empirical measures of the symmetric random walk on $\ZZ$, contradicting the result of \cite[pp.~922-923]{Baxter_Jain_Varadhan}. However, we were unable to justify this restriction step. A related strategy appears in \cite[Lemma 2]{Bryc_Dembo}, where the Dawson-Gärtner theorem is used to lift LDPs for observables to the space of finitely additive nonnegative set functions. As observed in \cite[Remark 1]{Bryc_Dembo}, this argument yields only a weak LDP, which is insufficient for deriving the contradiction we seek.
\end{remark}

\subsection{Rate functions for representative examples}\label{subsection_typical_examples}

\noindent In this section, we present explicit formulas for the rate function $I$ in selected cases and identify its minimisers. Throughout, we fix a probability measure $\mu=\sum_{\sigma \in \{-,0,+\}} \alpha_{\sigma} \mu_{\sigma} \in \mathcal{P}(\overline{\ZZ})$.

\textit{1.} The case where $(S_n)_{n \in \NN}$ is a simple symmetric random walk corresponds to the constant transition probability function $p \equiv 1 \slash 2$. Hypothesis~\ref{hypothesis_on_transition_function_p} is then satisfied with $p_-=1 \slash 2=p_+$. A direct computation shows that $I_{\textup{Cr}}^{1 \slash 2}(0)=0$. Therefore, the rate function $I$ reduces to
\begin{align*}
I(\mu)=\alpha_0 I_{\textup{DV}}(\mu_0).
\end{align*}
By \cite[Lemma 2.5]{Donsker_varadhan}, if there existed a probability measure $\mu_0 \in \mathcal{P}(\ZZ)$ such that $I_{\textup{DV}}(\mu_0)=0$, then $\mu_0$ would be an invariant measure for the symmetric random walk. Since no such probability measure exists, we deduce that $I(\mu_0)>0$ for all $\mu_0 \in \mathcal{P}(\ZZ)$. Thus, the minimisers of the rate function $I$ are
\begin{align*}
\big\{ \mu \in \mathcal{P}(\overline{\ZZ}) \, : \, I(\mu)=0 \big\} = \big\{ \alpha \delta_{-\infty} + (1-\alpha)\delta_{+\infty} \in \mathcal{P}(\overline{\ZZ}) \, : \, \alpha \in [0,1] \big\}.
\end{align*}

\textit{2.} A simple random walk with constant drift corresponds to the case where $p \equiv \overline{p} \in (0,1) \backslash \{1 \slash 2\}$. In this setting, the rate function $I$ takes the form
\begin{align*}
I(\mu)=
\begin{cases}
\alpha_0 I_{\textup{DV}}(\mu_0) + \frac{\alpha_+}{2} \log \left( \frac{1}{4 \overline{p}(1-\overline{p})} \right), & \text{if } \overline{p} < \frac{1}{2}, \\
\alpha_0 I_{\textup{DV}}(\mu_0) + \frac{\alpha_{-}}{2} \log \left( \frac{1}{4 \overline{p}(1-\overline{p})} \right), & \text{if } \overline{p} > \frac{1}{2}.
\end{cases}
\end{align*}

When the drift is towards $+\infty$, corresponding to $\overline{p}>\frac{1}{2}$, the empirical measure incurs an increasing cost in assigning mass to a neighbourhood of $-\infty$. This is reflected in the presence of the term $\frac{\alpha_{-}}{2} \log \left( \frac{1}{4 \overline{p}(1-\overline{p})} \right)$, which grows both as $\mu$ allocates more weight to $-\infty$ and as the drift parameter $\overline{p}$ increases to 1. Since a simple random walk with drift is transient, it admits no invariant probability measure, and therefore $I_{\textup{DV}}(\mu_0)>0$ for all $\mu_0 \in \mathcal{P}(\ZZ)$. Thus, if $\overline{p}<\frac{1}{2}$, then $\delta_{-\infty}$ is the unique minimiser of the rate function $I$ and if $\overline{p}>\frac{1}{2}$, then $\delta_{+\infty}$ is the unique minimiser of $I$.

\textit{3.} Suppose that the transition probability function $p$ is given by
\begin{align*}
p(k)=
\begin{cases}
p_+, & \text{if } k \geq 0, \\
p_-, & \text{if } k <0,
\end{cases}
\end{align*}
with $p_- < 1 \slash 2<p_+$. This is a particular instance of an oscillating random walk in the sense of \cite{Kemperman_RW}. In this regime, the walk experiences a drift towards $+\infty$ when it is on the right of 0, and a drift towards $-\infty$ when it is on its left. The rate function then reduces to
\begin{align*}
I(\mu)=\alpha_0 I_{\textup{DV}}(\mu_0) + \min \left\{ \alpha_{-} I_{\textup{Cr}}^{p_-}(0) \, , \, \alpha_{+} I_{\textup{Cr}}^{p_+}(0) \right\}.
\end{align*}
In particular, on the segment $[\delta_{-\infty},\delta_{+\infty}]$, we have for all $\alpha \in [0,1]$,
\begin{align*}
I\big((1-\alpha) \delta_{-\infty}+ \alpha \delta_{+\infty}\big)&=\min \left\{ (1-\alpha) I_{\textup{Cr}}^{p_-}(0) \, , \, \alpha I_{\textup{Cr}}^{p_+}(0)  \right\}.
\end{align*}
As the minimum of two affine functions, this is concave in $\alpha$ and satisfies $I(\delta_{-\infty})=0$ and $I(\delta_{+\infty})=0$. Figure~\ref{fig:concave_functions} illustrates the concavity of the rate function on the segment $[\delta_{-\infty},\delta_{+\infty}]$.
\begin{figure}[ht]
\centering
\begin{tikzpicture}
\begin{axis}[
    domain=0:1,
    samples=200,
    xlabel={$\mu \in [\delta_{-\infty},\delta_{+\infty}]$},
    ylabel={$I(\mu)$},
    ylabel style={rotate=-90},  
    width=12cm,
    height=5.6cm,
    legend style={font=\small, fill=white, fill opacity=0.8, draw opacity=1, at={(0.95,0.95)}, anchor=north east},
    xtick={0,1},
    xticklabels={$\delta_{-\infty}$,$\delta_{+\infty}$},
    ytick=\empty,
    axis x line=bottom,
    axis y line=left,
    grid=major,
    major grid style={gray!30}, 
    xmajorgrids=false,          
]

\addplot[
    blue,
    thick,
] 
({x}, {0.5*min(x*ln(1/(4*0.62*(1-0.62))), (1-x)*ln(1/(4*0.31*(1-0.31))))});
\addlegendentry{$p_+=0.62,\, p_-=0.31$}

\addplot[
    red,
    thick,
] 
({x}, {0.5*min(x*ln(1/(4*0.8*(1-0.8))), (1-x)*ln(1/(4*0.35*(1-0.35))))});
\addlegendentry{$p_+=0.8,\, p_-=0.35$}

\end{axis}
\end{tikzpicture}
\caption{The rate function $I$ on the segment $[\delta_{-\infty},\delta_{+\infty}]$ for two different sets of parameters.}
\label{fig:concave_functions}
\end{figure}
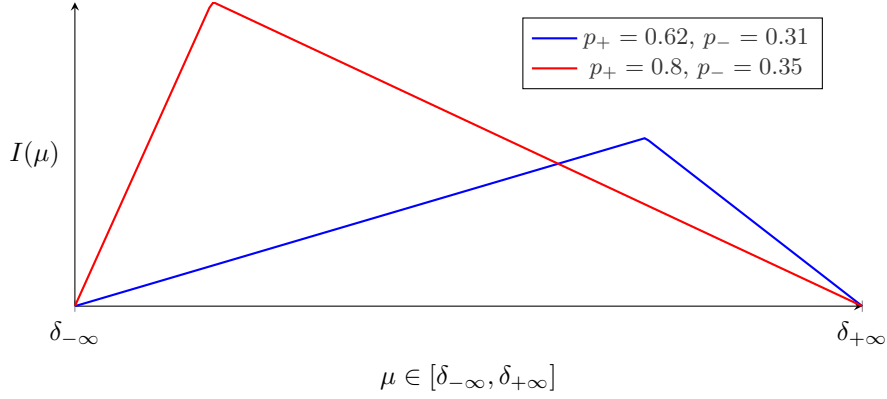
With this transition probability function, the process typically escapes either to $+\infty$ or to $-\infty$. Thus, the empirical measure can concentrate near $\delta_{+\infty}$ or $\delta_{-\infty}$ without incurring any large-deviation cost. In contrast, for the empirical measure to approximate a non-trivial convex combination of these two Dirac masses, the random walk must spend significant time in neighbourhoods of both infinities. This requires moving against the prevailing drift for part of its trajectory, which generates a large-deviation cost. The variational expression thus reflects the direction in which such upstream motion is least costly. The set of minimisers of the rate function $I$ is given by
\begin{align*}
\big\{ \mu \in \mathcal{P}(\overline{\ZZ}) \, : \, I(\mu)=0 \big\} = \big\{ \delta_{-\infty},\delta_{+\infty} \big\}.
\end{align*}
For empirical measures of non-irreducible discrete Markov chains, the rate function $I$ may also fail to be convex, see \cite[Example 2.1 and 2.2]{Dinwoodie} and \cite{Daures}. Indeed, if two target measures $\mu$ and $\nu$ are supported on distinct, non-communicating classes, then $I$ can be finite at $\mu$ and $\nu$, yet infinite at any non-trivial convex combination of the two, since the chain cannot allocate mass to both classes simultaneously. In contrast, in our setting the rate function is non-convex on its effective domain. Although transitions between neighbourhoods of the two infinities are possible, they carry an additional large-deviation cost.

\textit{4.} Suppose that $p$ is as in \textit{3.}, but now with $p_+<1 \slash 2<p_-$. In this regime, the simple random walk is biased towards 0. The rate function admits the affine decomposition
\begin{align*}
I(\mu)=\alpha_0 I_{\textup{DV}}(\mu_0) + \alpha_{-}I_{\textup{Cr}}^{p_-}(0) + \alpha_{+}I_{\textup{Cr}}^{p_+}(0).
\end{align*}
If we consider the random walk on $\ZZ$, the process admits as unique invariant probability measure $\mu_{\infty}$, given by $\mu_{\infty}(k)=C\left( \frac{p_+}{1-p_+} \right)^k$ if $k \geq 0$ and $\mu_{\infty}(k)=C \frac{1-p_+}{p_-} \left( \frac{1-p_-}{p_-} \right)^{-k-1}$ if $k \leq -1$, for some $C>0$. Thus, the Donsker--Varadhan rate function for the process on $\ZZ$ vanishes at $\mu$ if and only if $\mu=\mu_{\infty}$. Therefore, since $\log\left( \frac{1}{4p_{\pm}(1-p_{\pm})} \right)>0$, the rate function $I$ admits a unique minimiser, the invariant measure $\mu_{\infty}$.

\subsection{Sketch of proof}\label{subsection_sketch_proof}

Let us present a brief sketch of the proof of Theorem~\ref{theorem_existence_and_identification_of_rate_function}. Since the state space $\mathcal{P}(\overline{\ZZ})$ is compact for the topology of weak convergence, exponential tightness automatically holds. Hence, it suffices to establish a weak LDP. By \cite[Theorem~4.1.11]{DZ_LDP}, this reduces to proving local large deviation bounds on open balls: for every $\mu \in \mathcal{P}(\overline{\ZZ})$,
\begin{align*}
-I(\mu) \leq \inf_{\varepsilon>0} \liminf_{n \to \infty} \frac{1}{n} \log \PP_m \big( \ell_n \in B(\mu,\varepsilon)\big) \quad \text{and} \quad \inf_{\varepsilon>0} \limsup_{n \to \infty} \frac{1}{n} \log \PP_m \big( \ell_n \in B(\mu,\varepsilon) \big) \leq -I(\mu),
\end{align*}
where $B(\mu,\varepsilon)$ denotes an open ball for a metric that metrises the topology of weak convergence on $\mathcal{P}(\overline{\ZZ})$. Thus, the proof is split into a lower bound and an upper bound on $\PP_m\left( \ell_n \in B(\mu,\varepsilon) \right)$, addressed in Sections~\ref{section_lower_bound} and \ref{section_upper_bound} respectively. 

The proof relies on a partition of $\overline{\ZZ}$ into three disjoint regions: a finite central region $A^0$ and neighbourhoods $A^{\pm}$ of $\pm \infty$. With this decomposition in mind, we present in Section~\ref{section_regional_estimates} a series of exponential rates of decay for events localised in each region.

Using this decomposition, we partition the space of trajectories into three \textit{classes} according to the region in which they end. For the lower bound, we construct, for each target measure  $\mu=\sum_{\sigma \in \{-,0,+\}} \alpha_{\sigma} \mu_{\sigma}$ and each class, a set of \textit{typical trajectories} whose empirical measures approximate $\mu$ and whose visits to the regions occur in a prescribed sequential order. For instance, within the class of trajectories that end in the neighbourhood of $-\infty$, the set of typical trajectories consists of trajectories that first approximate the central measure $\mu_0$ for $\alpha_0 n$ steps, then make a single excursion of length $\alpha_+ n$ into $A^+$ and finally enter $A^{-}$, where they stay for the remaining $\alpha_{-}n$ steps in the form of a meander. For each $\sigma \in \{-,0,+\}$, the corresponding set of typical trajectories is, by construction, a subset of the event $\{\ell_n \in B(\mu,\varepsilon),S_n \in A^{\sigma}\}$ and therefore yields a lower bound. The Markov property at successive exit times of regions allows the probability of such trajectories to factorise into contributions from the three regions. Using the exponential rates of decay obtained in Section~\ref{section_regional_estimates}, we obtain for each $\sigma \in \{-,0,+\}$, a lower bound on the exponential rate of decay of $\PP_m\big(\ell_n \in B(\mu,\varepsilon) \, , \, S_n \in A^{\sigma}\big)$. 

For the upper bound, we remove the sequentiality constraint and allow arbitrary trajectories with multiple returns between regions. In Section~\ref{section_upper_bound}, we show that this additional freedom produces a combinatorial factor which does not alter the exponential rate. Hence, the upper and lower bounds obtained in Sections~\ref{section_lower_bound} and \ref{section_upper_bound} match.

Finally, since the number of classes is fixed and finite, we have
\begin{align*}
\max_{\sigma \in \{-,0,+\}} 
\liminf \limits_{n \to \infty} \frac{1}{n}\log \PP_m\big(\ell_n \in B(\mu,\varepsilon) \, , \, S_n \in A^{\sigma}\big) \leq \liminf \limits_{n \to \infty} \frac{1}{n}\log \PP_m\big(\ell_n \in B(\mu,\varepsilon)\big),
\end{align*}
and the corresponding statement holds with $\liminf$ replaced by $\limsup$, in which case the inequality becomes an equality. In other words, each class yields a candidate exponential rate and the overall upper and lower bounds are obtained by taking the most likely class. This leads to the variational expression of $I$ given in Equation~(\ref{eq_def_rate_fuction}).

\section{Notations and preliminary results}\label{subsection_preliminary_results}

In this section we collect several notational conventions and auxiliary results that will be used throughout the article. A summary of all the useful notations is given in Appendix~\ref{section_appendix_list_of_notations}.

\subsection{Topological facts about the space of probability measures}\label{subsubsection_topological_facts_proba_measures}

The topological space $\overline{\ZZ}$ can be metrised by the distance $d_{\overline{\ZZ}}$ defined by setting for all $h,k \in \overline{\ZZ}$,
\begin{equation*}\label{eq_def_distance_Z_bar}
d_{\overline{\ZZ}}(h,k) \eqdef |\varphi(h)-\varphi(k)|,
\end{equation*}   
where $\varphi:\overline{\ZZ} \longrightarrow [-1,1]$ is the map defined by
\begin{align*}
\varphi(+\infty)=1 \quad , \quad \varphi(-\infty)=-1 \quad \text{and} \quad \varphi(k)=
\begin{cases}
1-2^{-k}, & \text{if } k \geq 0, \\
-1+2^{-|k|}, & \text{if } k<0,
\end{cases} \quad k \in \ZZ.
\end{align*} 
The exact metric chosen to metrise the topology on $\overline{\ZZ}$ is not important but in the rest of this article we will be using the one defined above for the sake of concreteness. 

Since $(\overline{\ZZ},d_{\overline{\ZZ}})$ is a compact metric space, the space of probability measures $\mathcal{P}(\overline{\ZZ})$ endowed with the topology of weak convergence is itself compact and metrisable \cite[Theorem~3.2.2]{Bogachev}. Consider the Kantorovich-Rubinstein (KR) norm on the space of signed measures $\mathcal{M}(\overline{\ZZ})$, defined by
\begin{equation}\label{eq_def_BL_norm}
\| \mu \| \eqdef \sup \left\{ \int_{\overline{\ZZ}} f(x) \, \mathrm{d} \mu(x) \; : \; f \in \text{Lip}_1  \; , \; \sup_{k \in \overline{\ZZ}}|f(k)| \leq 1 \right\},
\end{equation}
where $\text{Lip}_1$ is the space of functions whose Lipschitz constant is bounded by one. The restriction of the distance induced by the KR norm to the convex subset $\mathcal{P}(\overline{\ZZ}) \subseteq \mathcal{M}(\overline{\ZZ})$ metrises the topology of weak convergence on $\mathcal{P}(\overline{\ZZ})$ \cite[Theorem~3.2.2]{Bogachev}. In the rest of the article, when a norm is considered on the space of measures it will always be the KR norm and $B(\mu,\varepsilon)$ stands for the open ball in $\mathcal{P}(\overline{\ZZ})$ centred at $\mu$ and of radius $\varepsilon$ for the distance the KR norm induces on $\mathcal{P}(\overline{\ZZ})$.  We recall the following two properties of the KR norm which we will use later on. Their verification is straightforward. For all $h,k \in \overline{\ZZ}$,
\begin{equation}\label{eq_lemma_properties_KR_norm_distance_dirac}
\| \delta_h - \delta_k \| \leq d_{\overline{\ZZ}}(h,k).
\end{equation}
For all $\mu \in \mathcal{P}(\overline{\ZZ})$,
\begin{equation}\label{eq_lemma_properties_KR_norm_of_proba}
\| \mu \| = 1.
\end{equation}

\subsection{Reduction of the problem}\label{subsubsection_reduction_of_the_problem}

Since the open balls of the distance induced by the KR norm on $\mathcal{P}(\overline{\ZZ})$ form a basis for the weak topology, we obtain the following proposition.

\begin{proposition}\label{proposition_RL_functions}
Let $m \in \ZZ$ and suppose that for all $\mu \in \mathcal{P}(\overline{\ZZ})$,
\begin{align}\label{eq_prop_lower_bound_RL}
-I(\mu)\leq &\lim \limits_{\varepsilon \to 0^+} \liminf \limits_{n \to \infty} \frac{1}{n} \log \PP_m \big( \ell_n \in B(\mu,\varepsilon) \big), \qquad \qquad \qquad \qquad \qquad \qquad \\
\text{and} \qquad \qquad \qquad \qquad \quad  &\lim \limits_{\varepsilon \to 0^+} \limsup \limits_{n \to \infty} \frac{1}{n} \log \PP_m \big( \ell_n \in B(\mu,\varepsilon) \big) \leq -I(\mu).\label{eq_prop_upper_bound_RL}
\end{align}
Then, the sequence $(\ell_n)_{n \in \NN}$ satisfies under $\PP_m$ a full LDP with rate function $I$.
\end{proposition}

\begin{proof}
By \cite[Theorem~4.1.11]{DZ_LDP}, the local bounds in Equations~\eqref{eq_prop_lower_bound_RL} and \eqref{eq_prop_upper_bound_RL} imply that the sequence $(\ell_n)_{n \in \NN}$ satisfies a weak LDP under $\PP_m$ with rate function $I$, and furthermore, that $I$ is lower semicontinuous. Since $\mathcal{P}(\overline{\ZZ})$ is compact for the topology of weak convergence, the weak LDP upgrades to a full LDP. 
\end{proof}

Proposition~\ref{proposition_RL_functions} shows that in order to prove Theorem~\ref{theorem_existence_and_identification_of_rate_function} it is enough to obtain Equations~(\ref{eq_prop_lower_bound_RL}) and (\ref{eq_prop_upper_bound_RL}) for all $\mu \in \mathcal{P}(\overline{\ZZ})$. We further reduce the problem by showing that it is enough to obtain the result for random walks started at the origin.

\begin{lemma}\label{lemma_restriction_to_0}
Suppose Equations (\ref{eq_prop_lower_bound_RL}) and (\ref{eq_prop_upper_bound_RL}) are satisfied for $m=0$, then the result holds for all $m \in \ZZ$.
\end{lemma}

\noindent We also remark that the case of a finitely supported initial distribution is a straightforward consequence of Theorem~\ref{theorem_existence_and_identification_of_rate_function}. The proof of Lemma~\ref{lemma_restriction_to_0} is given in Appendix~\ref{subsection_appendix_restriction_starting_point_origin}. Given that it is enough to show the result for $m=0$, we will simply write $\PP$ instead of $\PP_0$.

\subsection{The space of trajectories}\label{subsubsection_space_of_trajectories}

While the results are stated in a probabilistic setting, our analysis will primarily be carried out at the level of individual trajectories. For all $x,y \in \ZZ$, we write $x \sim y$ if $|x-y|=1$, in other words $x \sim y$ if and only if they are adjacent in $\ZZ$. Note that since $(S_n)_{n \in \NN}$ is a nearest-neighbour random walk, we have $S_{n} \sim S_{n+1}$ for all $n \in \NN$. We now introduce several sets of trajectories that will be used in the sequel. For all $n \in \NN$, let
\begin{equation}\label{def_Omega_n}
\Omega_n \eqdef \left\{ (x_1, \dots , x_{n}) \in \ZZ^n : \, x_i \sim x_{i+1} , \, \forall i =1,\dots,n-1 \right\} \;\, \text{ and } \;\, \Omega_n^{(0)} \eqdef \left\{ w \in \Omega_n  : \, w_1=0 \right\}. 
\end{equation}
We also define the set of finite trajectories
\begin{align*}
\Omega_{\textup{fin}} \eqdef \bigcup_{n \in \NN} \Omega_n.
\end{align*}

\noindent In what follows, we will often refer to elements $w \in \Omega_n$ as \textit{words}. A word $w \in \Omega_{\textup{fin}}$ is said to be of length $n$ if $w \in \Omega_n$, and we denote this by $|w|=n$. The $k$-th element of $w$, denoted $w_k$, will be called a \textit{letter} of the word $w$. We write $w=w_1w_2 \cdots w_n$ for the decomposition of $w$ into its letters. The last letter of $w$ will be denoted by $w_{-1}$. For $w \in \Omega_n$ and integers $p,q \in \NN$ satisfying $1 \leq p \leq q \leq n$, we define $w_{[p:q]}=w_pw_{p+1} \cdots w_q$, and refer to $w_{[p:q]}$ as a \textit{subword} of $w$. Given two words $w^1,w^2 \in \Omega_{\textup{fin}}$ such that $w^1_{-1} \sim w^2_1$, we can define the \textit{concatenation} of $w^1$ and $w^2$ as 
\begin{align*}
w^1 \cdot w^2 \eqdef w^1_1 w^1_2 \cdots w^1_{-1} w^2_1 w^2_2 \cdots w^2_{-1}.
\end{align*}
Note that $w^1 \cdot w^2 \in \Omega_{\textup{fin}}$ and both $w^1$ and $w^2$ appear as subwords of their concatenation $w^1 \cdot w^2$. Moreover, the length of the concatenation of two words satisfies $|w^1 \cdot w^2|=|w^1|+|w^2|$. 

For all $x,y \in \ZZ$, let
\begin{align*}
P(x,y) \eqdef 
\begin{cases}
p(x),  & \text{if } y=x+1, \\
1-p(x), & \text{if } y=x-1, \\
0, & \text{otherwise,}
\end{cases}
\end{align*}
be the probability of transitioning from $x$ to $y$. For a word $w=w_1 \cdots w_n \in \Omega_n$, we write 
\begin{equation}\label{eq_def_probability_of_word}
\mathfrak{p}(w) \eqdef \PP_{w_1} \left(S_j=w_j \, ,\, \forall j \in \dbl 1,n \dbr \right).
\end{equation}
By the Markov property we have 
\begin{align}\label{eq_link_p_w_with_probability}
\mathfrak{p}(w)=\prod_{j=1}^{n-1} P(w_j,w_{j+1}).
\end{align}
For a given word $w \in \Omega_{n}$, we define its empirical measure by setting
\begin{align*}
\ell(w)=\ell_{n}(w) \eqdef \frac{1}{n}\sum_{j=1}^{n} \delta_{w_j}.
\end{align*}

\subsection{Decomposition of the state space}\label{subsubsection_decomposition_of_the_state_space}

In what follows, we decompose the state space $\overline{\ZZ}$ into three regions, indexed by $\sigma \in \{-,0,+\}$. Fix an integer $R \in \NN$ and define for all $\sigma \in \{-,0,+\}$ the subsets 
\begin{equation*}\label{eq_def_A_sigma_R}
A^{-}_R \eqdef \dbl - \infty , -R \dbr \qquad , \qquad A^{0}_R \eqdef \dbl -R+1 , R-1 \dbr \qquad \text{and} \qquad A^{+}_R \eqdef \dbl R , +\infty \dbr.
\end{equation*}
When clear from context, we omit the dependence on $R$. Note that $A^{+}_R$ and $A^{-}_R$ are respectively open neighbourhoods of $+\infty$ and $-\infty$. Next, we let for all $\varepsilon>0$,
\begin{equation}\label{eq_def_R_Z_varepsilon}
R_{\overline{\ZZ}}(\varepsilon) \eqdef \left \lceil \log_2 \left( \frac{1}{\varepsilon} \right) \right \rceil+1.
\end{equation}
From direct computation and the definition of $\varphi$, we have for all $R \geq R_{\overline{\ZZ}}(\varepsilon)$,
\begin{equation}\label{eq_lemma_proximity_to_infty}
d_{\overline{\ZZ}}(k, -\infty)<\varepsilon , \quad k \in A^{-}_R \qquad \text{and} \qquad d_{\overline{\ZZ}}(k, +\infty)<\varepsilon , \quad k \in A^{+}_R.
\end{equation}

We will be interested in three particular sets of trajectories, each localised in one of the regions $A^{-}$, $A^{0}$, $A^{+}$ of $\overline{\ZZ}$. First, given a measure $\mu \in \mathcal{P}(\overline{\ZZ})$ and $\varepsilon>0$, for all $n \in \NN$, let
\begin{equation}\label{eq_def_trajectories_close_to_mu}
\Omega_n^{(0)}(\mu,\varepsilon) \eqdef \left\{ w \in \Omega_n^{(0)} \, : \, \ell(w) \in B(\mu,\varepsilon) \right\}
\end{equation} 
denote the set of trajectories of length $n$ which start at $0$ and whose empirical measures approximate $\mu$. Then, for all $n \in \NN$ and $\sigma \in \{-,+\}$, let
\begin{align}
\Omega_{n,\textup{exc}}^{\sigma} & \eqdef \big\{ w \in \Omega_n \, : \, w_1= \sigma R=w_n \quad \text{and} \quad w_j \in A^{\sigma}_R \, , \, \forall j \in \dbl 1, n \dbr \big\}, \label{eq_def_trajectories_excursions_above_R} \\ 
\Omega_{n,\textup{mea}}^{\sigma} & \eqdef \big\{ w \in \Omega_n \, : \, w_1=\sigma R \quad \text{and} \quad w_j \in A^{\sigma}_R \, , \, \forall j \in \dbl 1, n \dbr \big\}.\label{eq_def_trajectories_end_above_R}
\end{align}
Thus, $\Omega_{n,\textup{exc}}^{\sigma}$ is the set of excursions of length $n$ in $A_R^\sigma$ starting from $\sigma R$, while $\Omega_{n,\textup{mea}}^{\sigma}$ is the corresponding set of meanders. Finally, for all $\mu \in \mathcal{P}(\overline{\ZZ})$, $\varepsilon>0$, $n \in \NN$ and $\sigma \in \{-,0,+\}$, let 
\begin{equation}\label{eq_def_class}
\mathscr{C}_n^{\sigma} \eqdef \left\{ w \in \Omega_n^{(0)} \, : \, w_n \in A^{\sigma} \right\} \quad \text{and} \quad \mathscr{C}_n^{\sigma}(\mu,\varepsilon) \eqdef \left\{ w \in \mathscr{C}_n^{\sigma} \, : \, \ell(w) \in B(\mu,\varepsilon) \right\}.
\end{equation}
We refer to $\mathscr{C}_n^{-}$, $\mathscr{C}_n^{0}$ and $\mathscr{C}_n^{+}$ as the three \textit{classes} of trajectories. For instance, $\mathscr{C}_n^+$ consists of all trajectories of length $n$ that terminate above $R$, whereas $\mathscr{C}_n^0(\mu,\varepsilon)$ contains those trajectories of length $n$ that end in the interval $\dbl -R+1,R-1 \dbr$ and whose empirical measures approximate $\mu$. 

For any $U_n \subseteq \Omega_n$, we adopt the shorthand notation $\PP(U_n)$ to denote $\PP(S_{[1:n]} \in U_n)$. In particular, we will frequently apply this convention when $U_n$ is one of the sets defined in \eqref{eq_def_trajectories_close_to_mu}, (\ref{eq_def_trajectories_excursions_above_R}) or (\ref{eq_def_trajectories_end_above_R}).

\subsection{Estimates for the time spent in different regions}\label{subsection_occupation_time_estimates}

Fix a measure $\mu =\sum_{\sigma \in \{-,0,+\}} \alpha_{\sigma} \mu_{\sigma} \in \mathcal{P}(\overline{\ZZ})$, with $\mu_0 \in \mathcal{P}(\ZZ)$, and suppose the empirical measure of a trajectory is close to the probability measure $\mu$. In this section, we give bounds on the time such trajectory spends in the different regions $A^{\sigma}$ based on the coefficients $\alpha_{-}$, $\alpha_{0}$ and $\alpha_{+}$. We further provide some consequences of these bounds.

For all $w \in \Omega_n$, $\sigma \in \{-,0,+\}$ and $m \in \ZZ$, we define the following occupation times
\begin{equation}\label{eq_def_C_sigma_and_C_k}
N^{\sigma}(w) \eqdef \# \left\{ j \in \dbl 1,n \dbr \, : \, w_j \in A^{\sigma} \right\} \qquad \text{and} \qquad N_m(w)  \eqdef \# \left\{ j \in \dbl 1,n \dbr  \, : \, w_j=m \right\}.
\end{equation}
We also define for all $\varepsilon>0$,
\begin{equation}\label{eq_def_R_mu_0_varepsilon}
R_{\mu_0}(\varepsilon) \eqdef \min\big\{R \in \NN \, : \, \mu_0(\dbl -R+1,R-1 \dbr) > 1-\varepsilon \big\}.
\end{equation}

\begin{lemma}\label{lemma_bound_on_last_point_if_close_to_mu_c_appendix}
If $R > R_{\mu_0}(\varepsilon)$, then for all $n \in \NN$ and $w \in \Omega_{n}^{(0)}(\mu,2^{-R}\varepsilon)$,
\begin{align}
& \left \vert N^{\sigma}(w)-\alpha_{\sigma} n \right \vert < 2 \varepsilon n, \qquad \sigma \in \{-,0,+\} \label{eq_lemma_bound_on_last_point_if_close_to_mu_c_appendix_1}\\
\text{and} \qquad & \; \, N_m(w)  < 3 \varepsilon n, \qquad \qquad \;\;\, m \in \{-R,-R+1,R-1,R\}. \label{eq_lemma_bound_on_last_point_if_close_to_mu_c_appendix_2}
\end{align}
\end{lemma}

\begin{proof}
Consider $R > R_{\mu_0}(\varepsilon)$, $n \in \NN$, $w \in \Omega_{n}^{(0)}(\mu,2^{-R}\varepsilon)$ and let
\begin{align*}
f^{\sigma} \eqdef 2^{-R} \mathbbm{1}_{A^{\sigma}}.
\end{align*}
From the definition of $d_{\overline{\ZZ}}$, one checks that $f^{\sigma} \in \textup{Lip}_1$. Moreover, $\sup_{k \in \overline{\ZZ}}|f^{\sigma}(k)|\leq 1$. Thus, since $\ell(w) \in B\big( \mu, 2^{-R}\varepsilon \big)$, the definition of the KR norm, given in (\ref{eq_def_BL_norm}), applied to $f^{\sigma}$ and $-f^{\sigma}$, yields
\begin{align}
2^{-R}\varepsilon & > \| \ell(w)-\mu \| \notag \\
& \geq \left \vert \int f^{\sigma} \, \mathrm{d} \ell(w) - \int f^{\sigma} \,\mathrm{d}\mu \right \vert \notag \\
& = 2^{-R} \Big|\frac{1}{n} N^{\sigma}(w) - \mu(A^{\sigma}) \Big|. \label{eq_lemma_technical_for_measures_1}
\end{align}
Since $R > R_{\mu_0}(\varepsilon)$, the definition of $R_{\mu_0}(\varepsilon)$ implies that $|\mu(A^{\sigma})-\alpha_{\sigma}| < \varepsilon$ for all $\sigma \in \{-,0,+\}$. Thus, multiplying inequality (\ref{eq_lemma_technical_for_measures_1}) by $2^{R}$, we obtain for each $\sigma \in \{-,0,+\}$,
\begin{align*}
\Big| \frac{1}{n} N^{\sigma}(w)-\alpha_{\sigma} \Big| \leq \Big| \frac{1}{n} N^{\sigma}(w)-\mu(A^{\sigma}) \Big| + \Big| \mu(A^{\sigma})-\alpha_{\sigma} \Big| < 2 \varepsilon. 
\end{align*}
By multiplying by $n$ on both sides, we obtain \eqref{eq_lemma_bound_on_last_point_if_close_to_mu_c_appendix_1}.

To show (\ref{eq_lemma_bound_on_last_point_if_close_to_mu_c_appendix_2}), we use a similar argument but this time by using, for $m \in \{-R,-R+1,R-1,R\}$, the map $f_m \eqdef 2^{-|m|-1} \mathbbm{1}_{\{m\}}$. We have
\begin{align*}
\int f_m \mathrm{d}\ell(w) = 2^{-|m|-1} \frac{1}{n} \sum_{j=1}^n \mathbbm{1}_{\{w_j=m\}} = 2^{-|m|-1} \frac{N_m(w)}{n} \qquad \text{and} \qquad \int f_m \mathrm{d}\mu = 2^{-|m|-1} \alpha_0 \mu_0(\{m\}).  
\end{align*}
Similarly as before, one can verify that $f_m$ is 1-Lipschitz and uniformly bounded above by 1. Using the definition of the KR norm, with the maps $f_{m}$ and $-f_{m}$, we get
\begin{align*}
2^{-R}\varepsilon & > \left| \int f_m \mathrm{d}\ell(w) - \int f_m \mathrm{d}\mu \right| \\
& \geq 2^{-|m|-1} \left| \frac{N_m(w)}{n}-\alpha_0 \mu_0(\{m\}) \right|.
\end{align*}
Thus, since $R > R_{\mu_0}(\varepsilon)$ and $R-1 \leq |m| \leq R$, we know that $\mu_0(\{m\}) < \varepsilon$ and thus 
\begin{align*}
N_m(w) & \leq \left( 2^{|m|+1-R}\varepsilon + \alpha_0 \mu_0(\{m\}) \right) n < 3 \varepsilon n. \qedhere
\end{align*}
\end{proof}

\noindent Let $\mu \in  \mathcal{P}(\overline{\ZZ})$ be a probability measure that assigns no mass to the points at infinity, corresponding to the case $\alpha_0=1$. Then the position of the last letter of any word whose empirical measure is close to $\mu$ can be bounded. This is the content of the following corollary.

\begin{corollary}\label{corollary_bound_on_last_point_if_close_to_mu_c_appendix}
If $\mu_0 \in \mathcal{P}(\ZZ)$ and $R > R_{\mu_0}(\varepsilon)$ then for all $n \in \NN$ and $w \in \Omega_{n}^{(0)}(\mu_0,2^{-R}\varepsilon)$,
\begin{equation*}\label{eq_corollary_bound_on_last_point_if_close_to_mu_c_appendix}
|w_n| \leq R+2 \varepsilon n.
\end{equation*}
\end{corollary}

\begin{proof}
Suppose for the sake of contradiction that $w_{n}>2\varepsilon n +R$. Since we are considering a simple random walk, we have $w_j \geq R$ for all $j \geq n-2\varepsilon n$. This would imply that
\begin{align*}
N^+(w) \geq \# \dbl n- \lfloor 2\varepsilon n \rfloor , n \dbr \geq 2\varepsilon n.
\end{align*}
This contradicts (\ref{eq_lemma_bound_on_last_point_if_close_to_mu_c_appendix_1})  since in our case we have $\alpha_{+}=0$. If we suppose that $w_n<-2 \varepsilon n -R$, we obtain in a similar way that $N^{-}(w) \geq 2\varepsilon n$ which again contradicts (\ref{eq_lemma_bound_on_last_point_if_close_to_mu_c_appendix_1})  since $\alpha_{-}=0$. We therefore obtain the desired inequality.
\end{proof}

Define the empirical measure restricted to $A^{0}$ by setting 
\begin{equation}\label{eq_def_restricted_empirical_measure}
\ell^{0}(w) \eqdef \frac{1}{N^{0}(w)} \sum_{ \substack{ 1 \leq j \leq n \, : \\ w_j \in A^{0}}} \delta_{w_j}.
\end{equation}
Note that this measure is well defined as long as $N^{0}(w)\neq 0$. The following lemma relates this restricted empirical measure to the original empirical measure $\ell(w)$.

\begin{lemma}\label{lemma_restriction_of_measure_good_approx}
Suppose $\alpha_{0} \neq 0$, and fix $\varepsilon < \frac{\alpha_{0}}{2}$. Let $R>\max\{R_{\overline{\ZZ}}(\varepsilon),R_{\mu_0}(\varepsilon)\}$, $n \in \NN$ and $w \in \Omega_n^{(0)}(\mu,2^{-R}\varepsilon)$. Then,
\begin{equation*}\label{eq_lemma_restriction_of_measure_good_approx_int}
\| \ell^{0}(w) - \mu_0 \| < \frac{6 \varepsilon}{\alpha_{0}}.
\end{equation*}
\end{lemma}

\begin{proof}
Since $\varepsilon < \frac{\alpha_{0}}{2}$, we obtain from Equation~(\ref{eq_lemma_bound_on_last_point_if_close_to_mu_c_appendix_1}) of Lemma~\ref{lemma_bound_on_last_point_if_close_to_mu_c_appendix}, that $N^{0}(w)>0$. Therefore, $\ell^{0}(w)$ is well defined. Define as in the proof of Lemma~\ref{lemma_bound_on_last_point_if_close_to_mu_c_appendix}, the function
\begin{align*}
f \eqdef 2^{-R}\mathbbm{1}_{A^0}.
\end{align*}
This function is 1-Lipschitz and bounded by one. Then, by definition of the KR norm, we know that there exists a function $g^* \in \text{Lip}_1$ such that $\|g^* \|_{\infty} \leq 1$ and 
\begin{align}\label{lemma_restriction_of_measure_good_approx_proof_1}
\| \ell^0(w) - \mu_0 \| - \varepsilon \leq \int g^* \mathrm{d}\ell^0(w) - \int g^* \mathrm{d}\mu_0.
\end{align}
The product $fg^*$ is 2-Lipschitz and bounded uniformly by one, and since $w \in \Omega_n^{(0)}(\mu,2^{-R}\varepsilon)$, we have
\begin{align}
2^{-R}\varepsilon & > \| \ell(w) - \mu \| \notag \\
& \geq \int \frac{fg^*}{2} \mathrm{d} \ell(w) - \int \frac{fg^*}{2} \mathrm{d} \mu \notag \\
& = 2^{-R-1} \left( \int g^* \mathbbm{1}_{A^0} \mathrm{d} \ell(w) - \int g^* \mathbbm{1}_{A^0} \mathrm{d} \mu \right). \label{lemma_restriction_of_measure_good_approx_proof_2}
\end{align}
Using the definition of $\ell^0(w)$, $\mu$ and $\mu_0$ we have
\begin{align*}
\int g^* \mathbbm{1}_{A^0} \mathrm{d} \ell(w) - \int g^* \mathbbm{1}_{A^0} \mathrm{d} \mu & = \frac{N^0(w)}{n} \int g^* \mathrm{d}\ell^0(w) - \alpha_0 \int g^* \mathrm{d} \mu_0 + \alpha_0 \sum_{|k| \geq R} \mu_0(\{k\}) g^*(k).
\end{align*}
Since $\|g^*\|_{\infty} \leq 1$ and $R > R_{\mu_0}(\varepsilon)$, the last term can be bounded below by $-\alpha_0 \varepsilon$, while $\int g^* \mathrm{d}\ell^0(w)$ is bounded below by $-1$. We thus obtain
\begin{align*}
\int g^* \mathbbm{1}_{A^0} \mathrm{d} \ell(w) - \int g^* \mathbbm{1}_{A^0} \mathrm{d} \mu & \geq \alpha_0 \left( \int g^* \mathrm{d}\ell^0(w) - \int g^* \mathrm{d} \mu_0 \right) - \left| \frac{N^0(w)}{n}-\alpha_0 \right| - \alpha_0 \varepsilon.
\end{align*}
Using (\ref{eq_lemma_bound_on_last_point_if_close_to_mu_c_appendix_1}) of Lemma~\ref{lemma_bound_on_last_point_if_close_to_mu_c_appendix}, we deduce that $\left| \frac{N^0(w)}{n}-\alpha_0 \right| < 2 \varepsilon$. Plugging this back into (\ref{lemma_restriction_of_measure_good_approx_proof_2}) and using (\ref{lemma_restriction_of_measure_good_approx_proof_1}) yields
\begin{align*}
2\varepsilon & \geq \alpha_0 \left( \int g^* \mathrm{d}\ell^0(w) -  \int g^* \mathrm{d} \mu_0 \right) - \left| \frac{N^0(w)}{n}-\alpha_0 \right| - \alpha_0 \varepsilon \\
& > \alpha_0 \left( \| \ell^0(w) - \mu_0 \| - \varepsilon \right) - 3 \varepsilon ,
\end{align*}
which gives us the desired inequality.
\end{proof}

\section{Regional estimates}\label{section_regional_estimates}

As mentioned in the sketch of the proof in Section~\ref{subsection_sketch_proof}, both lower and upper bounds rely on a decomposition of the probability into three parts, each localised in one of the regions defined in Section~\ref{subsubsection_decomposition_of_the_state_space}. In this section, we present large deviation estimates for trajectories defined in (\ref{eq_def_trajectories_close_to_mu}), (\ref{eq_def_trajectories_excursions_above_R}) and (\ref{eq_def_trajectories_end_above_R}). We fix for the rest of this section a measure $\mu=\sum_{\sigma \in \{-,0,+\}} \alpha_{\sigma} \mu_{\sigma} \in \mathcal{P}(\overline{\ZZ})$, $\varepsilon>0$ and an integer $R \in \NN$. 

\subsection{Estimates on the central part}\label{subsection_central_part}

For $\mu_0 \in \mathcal{P}(\ZZ)$, we are interested in the exponential rate of decay of
\begin{equation}\label{eq_probability_of_interest_central}
\PP\big( \ell_n \in B(\mu_0,\varepsilon)\big)=\PP\big(\Omega_n^{(0)}(\mu_0,\varepsilon)\big),
\end{equation}
where, as recalled above, $\PP$ stands for $\PP_0$. The reduction to the case $m=0$ is justified by Lemma~\ref{lemma_restriction_to_0}.  

Since the empirical measure may be viewed either as an element of $\mathcal{P}(\ZZ)$ or, after compactification, as an element of $\mathcal{P}(\overline{\ZZ})$, we first recall the general Donsker--Varadhan functional in order to compare the corresponding rate functions. Given a Markov chain $(X_n)_{n \in \NN}$ on a space $\mathcal{X}$ with transition kernel $\Pi$, Donsker and Varadhan introduced in \cite{Donsker_varadhan} the functional
\begin{equation}\label{eq_DV_general}
\begin{array}{cccc}
I_{\Pi} : & \mathcal{P}(\mathcal{X}) & \longrightarrow & [0,+\infty] \\
& \mu & \longmapsto & \sup \limits_{u \in \mathcal{U}_1(\mathcal{X})} \int_{\mathcal{X}} \log \left( \frac{u(x)}{\Pi u(x)} \right) \mathrm{d} \mu (x),
\end{array}
\end{equation} 
where $\mathcal{U}_1(\mathcal{X})$ denotes the set of bounded Borel functions on $\mathcal{X}$ that are bounded from below by 1 and $\Pi u(x)=\int u(y) \Pi(x,\mathrm{d}y)$. 

In this subsection, we denote by $B_{\overline{\ZZ}}(\mu,\varepsilon)$ an open ball in $\mathcal{P}(\overline{\ZZ})$. The subscript is used to emphasise that the ambient space is $\mathcal{P}(\overline{\ZZ})$, and to distinguish it from balls in $\mathcal{P}(\ZZ)$, which will not be used here. We also write $I_{\textup{DV},\overline{\ZZ}}$ for the Donsker--Varadhan rate function associated with $(\ell_n)_{n \in \NN}$ viewed as a process in $\mathcal{P}(\overline{\ZZ})$ and $I_{\textup{DV},\ZZ}$ for the corresponding rate function when $(\ell_n)_{n \in \NN}$ is viewed as a process in $\mathcal{P}(\ZZ)$. The explicit expression of $I_{\textup{DV},\ZZ}$ is recalled in \eqref{Donsker_Varadhan_our_case}. It is the specialisation of \eqref{eq_DV_general} to the simple random walk on $\ZZ$ with transition function $p$. We now show that the two functionals $I_{\textup{DV},\ZZ}$ and $I_{\textup{DV},\overline{\ZZ}}$ coincide on $\mathcal{P}(\ZZ)$, and that the exponential rate of decay in \eqref{eq_probability_of_interest_central} is given by their common value.

\begin{lemma}\label{lemma_I_DV_same_for_central_measures}
For all $\mu_0 \in \mathcal{P}(\ZZ)$,
\begin{equation}\label{eq_lemma_I_DV_same_for_central_measures}
I_{\textup{DV},\overline{\ZZ}}(\mu_0)=I_{\textup{DV},\ZZ}(\mu_0).
\end{equation}
\end{lemma}

\begin{proof}
Fix $\mu_0 \in \mathcal{P}(\ZZ)$ and denote by $\Pi_{\overline{\ZZ}}$ and $\Pi_{\ZZ}$ the kernels associated to the Markov chain $(S_n)_{n \in \NN}$ viewed as a process in $\overline{\ZZ}$ and $\ZZ$ respectively. Let $\overline{u} \in \mathcal{U}_1(\overline{\ZZ})$, and denote by $u$ its restriction to $\ZZ$. Since $\mu_0$ is supported on $\ZZ$ and $\Pi_{\overline{\ZZ}}\overline{u}(k)=\Pi_{\ZZ}u(k)$ for all $k \in \ZZ$, we have
\begin{equation}\label{eq_proof_lemma_I_DV_same_for_central_measures}
\int_{\overline{\ZZ}} \log\left(\frac{\overline{u}(k)}{\Pi_{\overline{\ZZ}} \overline{u}(k)}\right)\,\mathrm{d}\mu_0(k)
=
\int_{\ZZ} \log\left(\frac{u(k)}{\Pi_{\ZZ} u(k)}\right)\,\mathrm{d}\mu_0(k).
\end{equation}
Taking the supremum over all $u \in \mathcal{U}_1(\ZZ)$ yields $I_{\textup{DV},\overline{\ZZ}}(\mu_0)\leq I_{\textup{DV},\ZZ}(\mu_0)$. Conversely, let $u \in \mathcal{U}_1(\ZZ)$, and extend it to a function $\overline{u} \in \mathcal{U}_1(\overline{\ZZ})$ by setting $\overline{u}(+\infty)=\overline{u}(-\infty)=1$. Then \eqref{eq_proof_lemma_I_DV_same_for_central_measures} still holds. Taking the supremum over all $\overline{u} \in \mathcal{U}_1(\overline{\ZZ})$ gives $I_{\textup{DV},\overline{\ZZ}}(\mu_0)\geq I_{\textup{DV},\ZZ}(\mu_0)$. This proves Lemma~\ref{lemma_I_DV_same_for_central_measures}.
\end{proof}

As noted above, the sequence $(\ell_n)_{n \in \NN}$ satisfies a weak LDP in $\mathcal{P}(\ZZ)$. However, this weak LDP cannot be transferred directly to $\mathcal{P}(\overline{\ZZ})$ by a direct application of the contraction principle. Indeed, \cite[Theorem~4.2.1]{DZ_LDP} does not provide the upper bound in the weak LDP setting, since the preimage of a compact set under a continuous map need not be compact. Nevertheless, Lemma \ref{lemma_I_DV_lower_bound} provides the required general upper bound, while the contraction principle remains applicable for the LDP lower bound. This approach is carried out in the proof of Proposition~\ref{proposition_exponential_rate_decay_I_DV}.

\begin{lemma}\label{lemma_I_DV_lower_bound}
For all $\mu \in \mathcal{P}(\overline{\ZZ})$,
\begin{equation*}\label{eq_lemma_I_DV_lower_bound}
\lim \limits_{\varepsilon\to 0^+} \limsup_{n \to \infty} \frac{1}{n} \log \PP \big( \ell_n \in B_{\overline{\ZZ}}(\mu,\varepsilon)\big) \leq -I_{\textup{DV},\overline{\ZZ}}(\mu).
\end{equation*}
\end{lemma}

\begin{proof}
One can verify that the Markov chain $(S_n)_{n \in \NN}$ is Feller in $\overline{\ZZ}$. Thus, the proof of \cite[pp.~7-10]{Donsker_varadhan} can be replicated here to show that for any closed set $F \subseteq \mathcal{P}(\overline{\ZZ})$,
\begin{align*}
\limsup_{n \to \infty} \frac{1}{n} \log \PP( \ell_n \in F) \leq - \inf \limits_{\mu \in F} I_{\textup{DV},\overline{\ZZ}}(\mu).
\end{align*}
In particular, with $F$ as the closed ball $\overline{B}_{\overline{\ZZ}}(\mu,\varepsilon)$ and using that $B_{\overline{\ZZ}}(\mu,\varepsilon) \subseteq \overline{B}_{\overline{\ZZ}}(\mu,\varepsilon)$, we obtain
\begin{align*}
\limsup_{n \to \infty} \frac{1}{n} \log \PP\big( \ell_n \in B_{\overline{\ZZ}}(\mu,\varepsilon)\big) \leq - \inf \limits_{\nu \in \overline{B}_{\overline{\ZZ}}(\mu,\varepsilon)} I_{\textup{DV},\overline{\ZZ}}(\nu).
\end{align*}
In addition, since the level sets of lower semi-continuous functions are closed and $\overline{\ZZ}$ is compact, we deduce that $I_{\textup{DV},\overline{\ZZ}}$ is a good rate function. Thus, applying \cite[Lemma 4.1.6]{DZ_LDP} and letting $\varepsilon$ go to 0, we obtain
\begin{align*}
\lim \limits_{\varepsilon \to 0^+} \limsup_{n \to \infty} \frac{1}{n} \log \PP\big( \ell_n \in B_{\overline{\ZZ}}(\mu,\varepsilon)\big)
& \leq - \lim \limits_{\varepsilon \to 0^+} \inf \limits_{\nu \in \overline{B}_{\overline{\ZZ}}(\mu,\varepsilon)} I_{\textup{DV},\overline{\ZZ}}(\nu) \\
& =-I_{\textup{DV},\overline{\ZZ}}(\mu),
\end{align*}
which concludes the proof.
\end{proof}

\begin{proposition}[$I_{\textup{DV}}$ as rate function]\label{proposition_exponential_rate_decay_I_DV}
For all $\mu_0 \in \mathcal{P}(\ZZ) \subseteq \mathcal{P}(\overline{\ZZ})$,
\begin{align*}
\lim_{\varepsilon \to 0^+} \liminf_{n \to \infty} \frac{1}{n} \log \PP\left( \ell_n \in B_{\overline{\ZZ}}(\mu_0,\varepsilon) \right) & = \lim_{\varepsilon \to 0^+} \limsup_{n \to \infty} \frac{1}{n} \log \PP\left( \ell_n \in B_{\overline{\ZZ}}(\mu_0,\varepsilon) \right) \notag \\
& = -I_{\textup{DV},\ZZ}(\mu_0).
\end{align*}
\end{proposition}

\begin{proof}
As shown in \cite{Donsker_varadhan_3} and \cite[Lemma~5]{Bryc_Dembo}, the sequence of empirical measures associated with the simple random walk on $\ZZ$ satisfies a weak LDP in $\mathcal{P}(\ZZ)$ with rate function $I_{\textup{DV},\ZZ}$. Let $\iota : \ZZ \hookrightarrow \overline{\ZZ}$ be the canonical inclusion, and denote by $\iota_* : \mathcal{P}(\ZZ) \longrightarrow \mathcal{P}(\overline{\ZZ})$ the induced push-forward map on probability measures. By continuity of $\iota_*$, there exists an open set $\mathcal{U} \subseteq \mathcal{P}(\ZZ)$ such that $\mu_0 \in \mathcal{U}$ and $ \mathcal{U} \subseteq \iota_*^{-1} \left(B_{\overline{\ZZ}}(\mu_0,\varepsilon)\right)$. Thus, by the weak LDP lower bound in $\mathcal{P}(\ZZ)$,
\begin{align*}
-I_{\textup{DV},\ZZ}(\mu_0) & \leq \liminf_{n \to \infty} \frac{1}{n} \log \PP\left( \ell_n \in \mathcal{U} \right) \\
& \leq \liminf_{n \to \infty} \frac{1}{n} \log \PP\left( \ell_n \in B_{\overline{\ZZ}}(\mu_0,\varepsilon) \right).
\end{align*}
Taking $\varepsilon$ to 0, we obtain
\begin{align*}
-I_{\textup{DV},\ZZ}(\mu_0) \leq \lim \limits_{\varepsilon \to 0^+} \liminf_{n \to \infty} \frac{1}{n} \log \PP\left( \ell_n \in B_{\overline{\ZZ}}(\mu_0,\varepsilon) \right).
\end{align*}
By combining the results of Lemma~\ref{lemma_I_DV_same_for_central_measures} and Lemma~\ref{lemma_I_DV_lower_bound}, one obtains
\begin{align*}
\lim \limits_{\varepsilon \to 0^+} \limsup_{n \to \infty} \frac{1}{n} \log \PP \big( \ell_n \in B_{\overline{\ZZ}}(\mu_0,\varepsilon) \big) & \leq -I_{\textup{DV},\overline{\ZZ}}(\mu_0) \\
& = -I_{\textup{DV},\ZZ}(\mu_0).
\end{align*}
This gives us the other bound and thereby proves Proposition~\ref{proposition_exponential_rate_decay_I_DV}.
\end{proof}

\subsection{Estimates on excursions and meanders}

We are now interested in estimating, for $\sigma \in \{-,+\}$, the exponential rate of decay of the quantities
\begin{align*}
\PP_{\sigma R}\left( \Omega_{n,\textup{exc}}^{\sigma} \right) \qquad \text{and} \qquad \PP_{\sigma R}\left( \Omega_{n,\textup{mea}}^{\sigma} \right),
\end{align*}
where the sets $\Omega_{n,\textup{exc}}^{\sigma}$ and $\Omega_{n,\textup{mea}}^{\sigma}$ are defined in \eqref{eq_def_trajectories_excursions_above_R} and \eqref{eq_def_trajectories_end_above_R}. We begin by stating an analogous result for homogeneous random walks.

Let $p \in (0,1)$, and let $(X_n)_{n \in \NN}$ be a sequence of i.i.d.~biased Rademacher random variables with parameter $p$, defined on some probability space $\big( \Omega,\mathcal{A},\QQ\big)$. We define the associated random walk by $S_n^X \eqdef \sum_{k=1}^n X_k$ for all $n \in \NN$. Denote by $(\mathcal{F}_n)_{n \in \NN}$ the natural filtration generated by $(X_n)_{n \in \NN}$ and let $\mathcal{F}=\sigma\left( \bigcup_{n \in \NN} \mathcal{F}_n \right)$. When no ambiguity is possible, we also write $\QQ$ for its restriction to $\mathcal{F}$. 

\begin{lemma}[Homogeneous case]\label{lemma_exp_decay_rate_homogeneous_case}
For each $\sigma \in \{-,+\}$,
\begin{align}
\lim_{n \to \infty} \frac{1}{n} \log \QQ\left( \sigma S_k^X \geq 0 \, , \,\forall k \in \dbl 1,n \dbr \right)&=-\inf_{\sigma x \in [0,1]}I_{\textup{Cr}}^p (x), \label{eq_lemma_exp_decay_rate_homogeneous_case_end} \\
\lim_{n \to \infty} \frac{1}{2n} \log \QQ\big(\{\sigma S_k^X \geq 0 \, , \, \forall k \in \dbl 1,2n \dbr\} \cap&\{S_{2n}^X=0\}\big)=-I_{\textup{Cr}}^p(0). \label{eq_lemma_exp_decay_rate_homogeneous_case_exc}
\end{align}
\end{lemma}

\begin{proof}
The proof consists of matching upper and lower bounds. The upper bound is an immediate consequence of the general large deviation upper bound. The lower bound, in contrast, requires an adaptation of the exponential tilting argument from Cramér's theorem \cite[Theorem~2.2.3]{DZ_LDP} to accommodate the positivity constraint on the walk. The key observation is that, under the appropriately tilted measure, the probability of staying positive decays only polynomially. For clarity of exposition, we treat only the case $\sigma=+$. The case $\sigma=-$ follows by symmetry after a change of sign.

Let us start with the upper bound leading to \eqref{eq_lemma_exp_decay_rate_homogeneous_case_end}. We have
\begin{align*}
\QQ\big( S_k^X \geq 0 \,,\, \forall k \in \dbl 1,n \dbr \big) & \leq \QQ\left( \frac{S_n^X}{n} \geq 0\right) = \QQ\left( \frac{S_n^X}{n} \in [ 0, 1 ] \right),
\end{align*}
where the last equality comes from the fact that $|S_n^X| \leq n$. Taking the logarithm, dividing by $n$ and taking the $\limsup_{n \to \infty}$, we obtain thanks to the LDP upper bound on the sequence of empirical means of $(X_n)_{n \in \NN}$, obtained by Cramér's theorem \cite[Theorem~2.2.3]{DZ_LDP}, that
\begin{align*}
\limsup_{n \to \infty} \frac{1}{n} \log \QQ\big( S_k^X \geq 0 \,,\, \forall k \in \dbl 1,n \dbr \big) & \leq \limsup_{n \to \infty} \frac{1}{n} \log \QQ\left( \frac{S_n^X}{n} \in [ 0, 1 ] \right) \\
& \leq -\inf_{x \in [ 0 , 1]} \Lambda^*(x), 
\end{align*}
where $\Lambda^*$ is the Fenchel--Legendre transform of $\Lambda$, the cumulant generating function associated with the law of $X_1$, i.e. 
\begin{align*}
\Lambda(\lambda)\eqdef \log \EE[e^{\lambda X_1}].
\end{align*}
With the definition of $I_{\textup{Cr}}^p$, given in \eqref{eq_I_Cr_Fenchel_Legendre}, we obtain
\begin{equation}\label{eq_lemma_upper_bound_tilting_for_end}
\limsup_{n \to \infty} \frac{1}{n} \log \QQ\big( S_k^X \geq 0 \,,\, \forall k \in \dbl 1,n \dbr \big) \leq -\inf_{x \in [ 0 , 1]} I_{\textup{Cr}}^p(x).
\end{equation}

Let us now turn to the lower bound. We treat the case $x \in [0,1)$ first, and defer the case $x=1$ to the end. Since $X_1$ has a compact support, $\Lambda$ is well defined on $\RR$, $\lim_{|\lambda| \to \infty} \Lambda(\lambda)=+\infty$ and $\Lambda$ is differentiable everywhere \cite[Lemma~2.2.5]{DZ_LDP}. In addition, there exists $\theta_x \in \RR$ such that 
\begin{align*}
\Lambda'(\theta_x)=x \qquad \text{and} \qquad \Lambda(\theta_x)-\theta_x x=\inf_{\lambda \in \RR} \big\{ \Lambda(\lambda)-\lambda x \big\}.
\end{align*}
Define the tilted measure $\widetilde{\QQ}_{x}$ on $\mathcal{F}$ by requiring that, for every $n \in \NN$,
\begin{align*}
\frac{\mathrm{d}\widetilde{\QQ}_x\big|_{\mathcal{F}_n}}{\mathrm{d}\QQ\big|_{\mathcal{F}_n}}
\eqdef
\exp\big(\theta_x S_n^X - n\Lambda(\theta_x)\big).
\end{align*}
One can verify that under the tilted measure $\widetilde{\QQ}_{x}$, the sequence $(X_n)_{n \in \NN}$ is a sequence of i.i.d.~random variables following a biased Rademacher distribution. In addition, by choice of $\theta_x$, one has
\begin{align*}
x = \Lambda'(\theta_x) =\widetilde{\EE}_x \big[ X_1 \big],
\end{align*}
and thereby, under $\widetilde{\QQ}_x$, the sequence $(X_n-x)_{n \in \NN}$ is i.i.d. and centred. Fix $\delta>0$. By inclusion of events, we obtain that
\begin{align*}
& \QQ\big( \{S_k^X \geq 0 \, , \, \forall k \in \dbl 1, n \dbr\} \cap \{ S_n^X \in [xn,(x+\delta)n) \}  \big) \\
& \qquad \qquad \qquad \qquad \qquad \qquad \qquad \qquad \qquad \geq \QQ\big( \{ S_k^X- k x \geq 0 \, , \, \forall k \in \dbl 1, n \dbr\} \cap \{ S_n^X-nx \in [0,\delta n) \}  \big).
\end{align*}
Therefore, by using the tilted measure $\widetilde{\QQ}_{x}$, we obtain
\begin{align}\label{eq_def_tilting_lower_bound}
& \QQ\big( \{ S_k^X \geq 0 \,,\, \forall k \in \dbl 1, n \dbr\} \cap \{ S_n^X \in [xn,(x+\delta)n) \}  \big) \notag \\
& \qquad \qquad \qquad  \geq \widetilde{\EE}_{x}\big[ \mathbbm{1}_{\{ S_k^X- kx \geq 0 \, , \, \forall k \in \dbl 1, n \dbr\} \cap \{ S_n^X-nx \in [0,\delta n) \}} e^{-\theta_x S_n^X} \big] e^{n \Lambda(\theta_x)} \notag \\
& \qquad \qquad \qquad  \geq \widetilde{\QQ}_{x}\big(\{ S_k^X- kx \geq 0 \,,\, \forall k \in \dbl 1, n \dbr\} \cap \{ S_n^X-nx \in [0,\delta n) \}\big) e^{n \big( \Lambda(\theta_x)-\theta_x x - |\theta_x| \delta \big)},
\end{align} 
where in the last line, we have used that $-\theta_x S_n^X > -\theta_x x n - \vert \theta_x \vert \delta n$ on the event $\{S_n^X-nx \in [0,\delta n)\}$. To control the first term on the right-hand side, we invoke a local central limit theorem for random walks conditioned to remain positive, see \cite{Vatutin}, which yields a lower bound of order $n^{-3 \slash 2}$. In particular, this polynomial decay is negligible on the exponential scale relevant for the large deviation principle. For all $n \in \NN$, let
\begin{align*}
S_n^{X,x} \eqdef \sum_{k=1}^n (X_k-x)=S_n^X-nx.
\end{align*}
As previously noted, under $\widetilde{\QQ}_x$ the sequence $(X_i-x)_{i \in \NN}$ is i.i.d. and centred, so $(S_n^{X,x})_{n \in \NN}$ forms a centred random walk. In the terminology of \cite[p.~180]{Vatutin}, $(S_n^{X,x})_{n \in \NN}$ is $(2,1-x)$-lattice. Moreover, by the central limit theorem, the step distribution of $S_n^{X,x}$ under $\widetilde{\QQ}_x$ belongs to the domain of attraction of the standard Gaussian distribution. Let $\tau^-$ denote the first weak descending epoch
\begin{align*}
\tau^- \eqdef \min\{n \geq 1 \, : \, S_n^{X,x} \leq 0\}.
\end{align*}
Next, define for all $n \in \NN$,
\begin{align*}
k(n,x)\eqdef 2 \left \lfloor \frac{-(1-x)n}{2} \right \rfloor +2,
\end{align*}
so that $(1-x)n+k(n,x) \in (0,2]$. Fix $n$ sufficiently large so that $(1-x)n+k(n,x) \in [0,\delta  n)$. It then follows that
\begin{align}\label{eq_precise_event_tilting_lower_bound}
\widetilde{\QQ}_{x}\big(\{ S_k^{X,x} \geq 0 \,,\, \forall k \in \dbl 1, n \dbr\} \cap \{ S_n^{X,x} \in [0,\delta n) \}\big) \geq \widetilde{\QQ}_{x}\big( \tau^->n \, , \, S_n^{X,x}=(1-x)n+k(n,x) \big).
\end{align} 
By \cite[Theorem~6]{Vatutin}, and noting that $k(n,x) \in (-(1-x)n,-(1-x)n+2] \cap 2 \ZZ$, there exists a constant $C'>0$ such that
\begin{equation}\label{eq_equivalence_Vatutin}
\sqrt{n}\widetilde{\QQ}_{x}\big( S_n^{X,x}=(1-x)n+k(n,x) \big\vert \tau^->n \big) \underset{n \to \infty}{\sim} C' \frac{H\big((1-x)n+k(n,x)\big)}{n \widetilde{\QQ}_x(\tau^->n)},
\end{equation}
where $H$ is the renewal function defined in \cite[p.~179]{Vatutin}. Thus, for all $n$ sufficiently large, since $(1-x)n+k(n,x) \in (0,2]$, it follows from the definition of $H$ in \cite[(7), p.~179]{Vatutin} that
\begin{align*}
H\big((1-x)n+k(n,x)\big)\geq 1.
\end{align*}
Combining this with \eqref{eq_equivalence_Vatutin}, we deduce that there exists a constant $C>0$ such that
\begin{align*}
\widetilde{\QQ}_{x}\big( \tau^->n \, , \, S_n^{X,x}=(1-x)n+k(n,x) \big) \geq \frac{C}{n^{3/2}}.
\end{align*}
This provides a polynomial lower bound on the right-hand side of inequality~(\ref{eq_precise_event_tilting_lower_bound}). Substituting this into inequality~(\ref{eq_def_tilting_lower_bound}), then taking the logarithm, dividing by $n$ and finally taking $\liminf_{n \to \infty}$, we obtain for all $x \in [0,1)$ and all $\delta>0$,
\begin{align*}
& \liminf_{n \to \infty} \frac{1}{n} \log \QQ\big( S_k^X \geq 0 \, , \, \forall k \in \dbl 1, n \dbr  \big) \\
& \qquad \qquad \qquad \qquad \geq \liminf_{n \to \infty} \frac{1}{n} \log \QQ\big( \{S_k^X \geq 0 \, , \, \forall k \in \dbl 1, n \dbr\} \cap \{ S_n^X \in [xn,(x+\delta)n) \}  \big) \\
& \qquad \qquad \qquad \qquad \geq \Lambda(\theta_x)-\theta_x x - |\theta_x| \delta.
\end{align*}
By choice of $\theta_x$, we have
\begin{align*}
\Lambda(\theta_x)-\theta_x x=\inf_{\lambda \in \RR} \big\{ \Lambda(\lambda)-\lambda x \big\} = -\Lambda^*(x)=-I_{\textup{Cr}}^p(x).
\end{align*}
Thus, taking $\delta$ to 0, we deduce that for all $x \in [0,1)$ ,
\begin{equation}\label{eq_liminf_after_exponential_tilting}
\liminf_{n \to \infty} \frac{1}{n} \log \QQ\big( S^X_k \geq 0 \, , \, \forall k \in \dbl 1, n \dbr  \big) \geq -I_{\textup{Cr}}^p(x).
\end{equation}
To handle the case $x=1$, observe that 
\begin{align*}
\QQ\big( S_k^X \geq 0 \, , \, \forall k \in \dbl 1,n \dbr \big) & \geq \QQ \big( X_k=1 \, , \, \forall k \in \dbl 1,n \dbr \big) = p^n.
\end{align*}
Thus, using the definition of $I_{\textup{Cr}}^p$ given in \eqref{eq_def_I_Cr_p}, we deduce that
\begin{align*}
\liminf_{n \to \infty} \frac{1}{n} \log \QQ\big( S^X_k \geq 0 \, , \, \forall k \in \dbl 1, n \dbr  \big) \geq \log(p)=-I_{\textup{Cr}}^p(1).
\end{align*}

Leveraging this inequality and taking the supremum over $x \in [0,1]$ on the right-hand side of \eqref{eq_liminf_after_exponential_tilting} yields 
\begin{equation*}
\liminf_{n \to \infty} \frac{1}{n} \log \QQ\big( S_k^X \geq 0 \,,\, \forall k \in \dbl 1,n \dbr \big) \geq -\inf_{x \in [ 0 , 1]} I_{\textup{Cr}}^p(x),
\end{equation*}
which, together with the upper bound obtained in \eqref{eq_lemma_upper_bound_tilting_for_end} establishes \eqref{eq_lemma_exp_decay_rate_homogeneous_case_end}.

Equation~(\ref{eq_lemma_exp_decay_rate_homogeneous_case_exc}) is shown in a similar manner. We start with the upper bound. 
\begin{align*}
\QQ\big( \{S_k^X \geq 0 \, , \, \forall k \in \dbl 1, 2n \dbr\} \cap \{ S_{2n}^X=0\} \big) & \leq \QQ \left( \left \vert \frac{S_{2n}^X}{2n}  \right \vert \in \{0\} \right).
\end{align*}
Taking the logarithm, dividing by $2n$ and taking the $\limsup_{n \to \infty}$, we obtain thanks to the LDP upper bound on the sequence of empirical means of $(X_n)_{n \in \NN}$, obtained by Cramér's theorem \cite[Theorem~2.2.3]{DZ_LDP}, that
\begin{align*}
\limsup_{n \to \infty} \frac{1}{2n} \log \QQ\big( \{S_k^X \geq 0 \, , \, \forall k \in \dbl 1, 2n \dbr\} \cap \{ S_{2n}^X=0\} \big) \leq - I_{\textup{Cr}}^{p}(0).
\end{align*}
For the lower bound, we use that
\begin{align*}
\QQ\big( \{S_k^X \geq 0 \, , \, \forall k \in \dbl 1, 2n \dbr\} \cap \{ S_{2n}^X=0\} \big) \geq \QQ\big( \{S_k^X > 0 \, , \, \forall k \in \dbl 2, 2n-2 \dbr\} \cap \{ S_{2n-2}^X=2\} \big)(1-p)^2.
\end{align*}
By considering the tilted measure $\widetilde{\QQ}_0$ and the fact that, in the terminology of \cite[p.~180]{Vatutin}, $(S_n^X)_{n \in \NN}$ is $(2,1)$-lattice, we obtain using again \cite[Theorem~6]{Vatutin}, that
\begin{align*}
\liminf_{n \to \infty} \frac{1}{2n} \log \QQ \big( \{ S_k^X \geq 0 \,,\, \forall k \in \dbl 2,2n-2 \dbr\} \cap\{S_{2n-2}^X=2\}\big) \geq \Lambda(\theta_0).
\end{align*}
By choice of $\theta_0$, we have $\Lambda(\theta_0)=-I_{\textup{Cr}}^{p}(0)$ and Equation~(\ref{eq_lemma_exp_decay_rate_homogeneous_case_exc}) follows.
\end{proof}

\begin{remark}\label{remark_after_tilting_for_Catalan}
Lemma~\ref{lemma_exp_decay_rate_homogeneous_case} could alternatively be proved by a combinatorial argument, using Catalan numbers to enumerate excursions of a fixed length and Stirling's approximation to obtain the right exponential decay rate. Although this approach is simpler, the proof presented above is more robust, as it applies to random walks with an arbitrary step distribution of compact support. In particular, it readily extends to the generalisation discussed in Remark~\ref{remark_generalisation_to_step_distribution}. The only modification required is that, if the walk has period $d$ instead of 2, the term $2n$ in Equation~(\ref{eq_lemma_exp_decay_rate_homogeneous_case_exc}) must be replaced by $dn$.
\end{remark}

\begin{remark}\label{remark_other_tilting}
Note that in the proof of Lemma~\ref{lemma_exp_decay_rate_homogeneous_case}, we could have also defined the tilted measures $\widetilde{\QQ}_x$ by setting for $x \in [0,1)$,
\begin{align*}
\frac{\mathrm{d}\widetilde{\QQ}_{x}}{\mathrm{d}\QQ} \Big \vert_{\mathcal{F}_n} \eqdef \left( \frac{p_x}{p} \right)^{\frac{n+S_n^X}{2}} \left(\frac{1-p_x}{1-p} \right)^{\frac{n-S_n^X}{2}}, \quad \text{ where } p_x=\frac{1+x}{2}.
\end{align*}
This change of measure turns a simple random walk with step distribution of mean $2p-1$ into a simple random walk with step distribution of mean $2p_{x}-1=x$. Note that, upon taking the logarithm, dividing by $n$ and taking the limit $n \to \infty$ in the above expression involving $S_n^X$, one recovers the entropic expression of $I_{\textup{Cr}}^{p_{\sigma}}(x)$ given in (\ref{eq_def_I_Cr_p}).
\end{remark}

\noindent The following proposition adapts Lemma~\ref{lemma_exp_decay_rate_homogeneous_case} to our setting, in which the random walk is no longer spatially homogeneous. Recall that by Hypothesis~\ref{hypothesis_on_transition_function_p}, for each $\sigma \in \{-,+\}$, the limits $p_{\sigma} = \lim \limits_{k \to \sigma \infty} p(k)$ exist and $p_{\sigma} \in (0,1)$. We introduce the function $G:\RR_+^* \to \RR$, defined for all $\varepsilon>0$ by
\begin{equation*}\label{eq_def_G_varepsilon}
G(\varepsilon) \eqdef \max \limits_{\sigma \in \{-,+\}} \left\{ \frac{1}{\min\{p_{\sigma},1-p_{\sigma}\}} \sup_{\sigma k \geq R_{\overline{\ZZ}}(\varepsilon)} |p(k)-p_{\sigma}| \right\}.
\end{equation*}
We know that $\lim_{\varepsilon \to 0^+}R_{\overline{\ZZ}}(\varepsilon)=+\infty$. Since $\lim_{k \to \sigma \infty}p(k)=p_{\sigma}$, we deduce that 
\begin{equation}\label{eq_for_lemma_G_varepsilon_to_zero}
\lim \limits_{\varepsilon \to 0^+} G(\varepsilon)=0.
\end{equation}
In particular, there exists $\varepsilon^*>0$ such that for all $0<\varepsilon<\varepsilon^*$, we have $0 \leq G(\varepsilon)<1$. The following proposition gives the local estimates we are looking for. Since the length of an excursion is measured here by the number of positions in the trajectory, rather than by the number of steps, every excursion of a simple random walk has odd length. Therefore, $\Omega_{n,\textup{exc}}^{\sigma}=\emptyset$ for all $n \in 2 \NN$. This is why, in  the following proposition, we consider $2\left \lceil \frac{n}{2} \right \rceil+1$, which is odd, rather than $n$.

\begin{proposition}\label{proposition_exponential_rate_of_decay_exc_and_end}
There exists a function $g:\RR_+ \to \RR$ such that $\lim \limits_{\varepsilon \to 0^+}g(\varepsilon)=0$ and such that if $0<\varepsilon<\varepsilon^*$ and $R \geq R_{\overline{\ZZ}}(\varepsilon) $, then for all $\sigma \in \{-,+\}$,
\begin{align}
-I_{\textup{Cr}}^{p_{\sigma}}(0) - g(\varepsilon) & \leq \liminf \limits_{n \to \infty} \frac{1}{n} \log \PP_{\sigma R}\left(\Omega_{2 \left \lceil \frac{n}{2} \right \rceil+1,\textup{exc}}^{\sigma} \right), \label{eq_proposition_exponential_rate_of_decay_exc_lower} \\
-I_{\textup{Cr}}^{p_{\sigma}}(0) + g(\varepsilon) & \geq \limsup \limits_{n \to \infty} \frac{1}{n} \log \PP_{\sigma R}\left(\Omega_{2 \left \lceil \frac{n}{2} \right \rceil+1,\textup{exc}}^{\sigma} \right), \label{eq_proposition_exponential_rate_of_decay_exc_upper}\\
-\inf_{\sigma x \in [0,1]} I_{\textup{Cr}}^{p_{\sigma}}(x) - g(\varepsilon) & \leq \liminf \limits_{n \to \infty} \frac{1}{n} \log \PP_{\sigma R}\left( \Omega_{2 \left \lceil \frac{n}{2} \right \rceil +1,\textup{mea}}^{\sigma} \right), \label{eq_proposition_exponential_rate_of_decay_final_lower}\\
-\inf_{\sigma x \in [0,1]} I_{\textup{Cr}}^{p_{\sigma}}(x) + g(\varepsilon)& \geq \limsup \limits_{n \to \infty} \frac{1}{n} \log \PP_{\sigma R}\left( \Omega_{2 \left \lceil \frac{n}{2} \right \rceil + 1,\textup{mea}}^{\sigma} \right). \label{eq_proposition_exponential_rate_of_decay_final_upper}
\end{align}
\end{proposition}

\begin{proof}
We will restrict our attention to the case where $\sigma=+$, since the case where $\sigma=-$ follows by the same arguments. For a trajectory $w \in \Omega_{n}$, consider the number of upward and downward steps
\begin{align*}
N_{\uparrow}(w) \eqdef \# \big\{ j \in \dbl 1,n-1 \dbr : w_{j+1}=w_j+1 \big\} \quad \text{and} \quad N_{\downarrow}(w) \eqdef \# \big\{ j \in \dbl 1,n-1 \dbr : w_{j+1}=w_j-1 \big\}.
\end{align*} 
For all $k \geq R_{\overline{\ZZ}}(\varepsilon)$, we have by definition of $G(\varepsilon)$, that
\begin{align*}
\left \vert \frac{p(k)}{p_+}-1 \right \vert = \frac{1}{p_+}|p(k)-p_+| \leq G(\varepsilon) \qquad \text{and} \qquad \left \vert \frac{1-p(k)}{1-p_+} -1 \right \vert = \frac{1}{1-p_+}|p(k)-p_+| \leq G(\varepsilon).
\end{align*}
If $w \in \Omega_{n,\textup{mea}}^{+}$, then for all $j \in \dbl 1,n \dbr$, $w_j \geq R \geq R_{\overline{\ZZ}}(\varepsilon)$. Therefore, for all $j \in \dbl 1,n \dbr$,
\begin{equation}\label{eq_bounds_proba_transition_close_to_R}
1-G(\varepsilon) \leq \frac{p(w_j)}{p_+} \leq 1+G(\varepsilon) \qquad \text{and} \qquad 1-G(\varepsilon) \leq \frac{1-p(w_j)}{1-p_+} \leq 1+G(\varepsilon).
\end{equation}
We deduce that for all $n \in \NN$ and $w \in \Omega_{n,\textup{mea}}^{+}$, and hence also for all $w \in \Omega_{n,\textup{exc}}^+$, we have
\begin{equation}\label{eq_bound_proba_of_trajectory}
\big(1-G(\varepsilon)\big)^{n-1} p_+^{N_{\uparrow}(w)}(1-p_+)^{N_{\downarrow}(w)} \leq \mathfrak{p}(w) \leq \big(1+G(\varepsilon)\big)^{n-1} p_+^{N_{\uparrow}(w)}(1-p_+)^{N_{\downarrow}(w)}.
\end{equation}
In other words, up to the factors $(1-G(\varepsilon))^{n-1}$ and $(1+G(\varepsilon))^{n-1}$, we have reduced the problem to that of a simple random walk with constant transition probability function $p \equiv p_+$. We are now in a position to apply the results obtained in Lemma~\ref{lemma_exp_decay_rate_homogeneous_case}. Let $(X_n)_{n \in \NN}$ be a sequence of i.i.d. random variables with distribution $\QQ(X_1=1)=p_+$ and $\QQ(X_1=-1)=1-p_+$, and $S_n^X=\sum_{k=1}^n X_k$. By translation invariance of homogeneous random walks and by identifying excursions above $R$ with excursions above 0, we deduce that
\begin{align*}
\sum_{w \in \Omega_{2 \left \lceil \frac{n}{2} \right \rceil+1,\textup{exc}}^+}p_+^{N_{\uparrow}(w)}(1-p_+)^{N_{\downarrow}(w)} & = \QQ\left( \left\{S^X_{k} \geq 0 \,,\, \forall k \in \left\dbl 1,2 \left \lceil \frac{n}{2} \right \rceil \right\dbr \right\} \cap \left\{ S_{2 \left \lceil \frac{n}{2} \right \rceil}^X=0 \right\} \right) \\
\text{and} \qquad \sum_{w \in \Omega_{2 \left \lceil \frac{n}{2} \right \rceil+1, \textup{mea}}^+}p_+^{N_{\uparrow}(w)}(1-p_+)^{N_{\downarrow}(w)} & = \QQ\left( S^X_{k} \geq 0 \,,\, \forall k \in \left\dbl 1,2 \left \lceil \frac{n}{2} \right \rceil \right\dbr \right).
\end{align*} 
Therefore, summing over all $w \in \Omega_{2 \left \lceil \frac{n}{2} \right \rceil+1,\textup{exc}}^+$ in Equation~(\ref{eq_bound_proba_of_trajectory}), we obtain
\begin{align*}
& \big(1-G(\varepsilon)\big)^{2 \left \lceil \frac{n}{2} \right \rceil} \QQ\left( \left\{S^X_{k} \geq 0 \, , \, \forall k \in \left\dbl 1,2 \left \lceil \frac{n}{2} \right \rceil \right\dbr \right\} \cap \left\{ S_{2 \left \lceil \frac{n}{2} \right \rceil}^X=0\right\} \right) \leq \PP_{ R}\left(\Omega_{2 \left \lceil \frac{n}{2} \right \rceil+1,\textup{exc}}^{+} \right) \\
\text{and} \qquad & \PP_{ R}\left(\Omega_{2 \left \lceil \frac{n}{2} \right \rceil+1,\textup{exc}}^{+} \right) \leq \big(1+G(\varepsilon)\big)^{2 \left \lceil \frac{n}{2} \right \rceil } \QQ\left( \left\{S^X_{k} \geq 0 \, , \, \forall k \in \left\dbl 1,2 \left \lceil \frac{n}{2} \right \rceil \right\dbr \right\} \cap \left\{ S_{2 \left \lceil \frac{n}{2} \right \rceil}^X=0\right\} \right).
\end{align*}
By taking the logarithm, dividing by $n$ and then taking $\liminf_{n \to \infty}$ on both sides of the first inequality and then using Equation~(\ref{eq_lemma_exp_decay_rate_homogeneous_case_exc}) of Lemma~\ref{lemma_exp_decay_rate_homogeneous_case}, we obtain
\begin{align*}
-I_{\textup{Cr}}^{p_+}(0) + \log\big(1-G(\varepsilon)\big) & \leq \liminf \limits_{n \to \infty} \frac{1}{n} \log \PP_R\left( \Omega_{2\left \lceil \frac{n}{2} \right \rceil +1,\textup{exc}}^{+} \right).
\end{align*}
If instead, we take the logarithm, divide by $n$ and take $\limsup_{n \to \infty}$ in the second inequality, we obtain 
\begin{align*}
\limsup \limits_{n \to \infty} \frac{1}{n} \log \PP_R\left( \Omega_{2\left \lceil \frac{n}{2} \right \rceil +1,\textup{exc}}^{+} \right) & \leq -I_{\textup{Cr}}^{p_+}(0) + \log\big(1+G(\varepsilon)\big).
\end{align*}
Thus, inequalities (\ref{eq_proposition_exponential_rate_of_decay_exc_lower}) and (\ref{eq_proposition_exponential_rate_of_decay_exc_upper}) follow by setting for all $\varepsilon>0$,
\begin{align*}
g(\varepsilon) \eqdef \max\big\{ \big\vert \log\big(1-G(\varepsilon)\big) \big \vert \, ; \, \log\big(1+G(\varepsilon) \big)\big\}.
\end{align*}
By Equation~(\ref{eq_for_lemma_G_varepsilon_to_zero}), the function $g$ satisfies $\lim_{\varepsilon \to 0^+} g(\varepsilon)=0$. To obtain inequalities (\ref{eq_proposition_exponential_rate_of_decay_final_lower}) and (\ref{eq_proposition_exponential_rate_of_decay_final_upper}), we sum over all $w \in \Omega_{2 \left \lceil \frac{n}{2} \right \rceil+1,\textup{mea}}$ in Equation~(\ref{eq_bound_proba_of_trajectory}) and obtain
\begin{align*}
\big(1-G(\varepsilon)\big)^{2 \left \lceil \frac{n}{2} \right \rceil} \QQ\left(S^X_{k} \geq 0 \, , \, \forall k \in \left\dbl 1,2 \left \lceil \frac{n}{2} \right \rceil \right\dbr \right)
& \leq \PP_{ R}\left(\Omega_{2 \left \lceil \frac{n}{2} \right \rceil+1,\textup{mea}}^{+} \right) \\
& \leq \big(1+G(\varepsilon)\big)^{2 \left \lceil \frac{n}{2} \right \rceil} \QQ\left( S^X_{k} \geq 0 \, , \, \forall k \in \left\dbl 1,2 \left \lceil \frac{n}{2} \right \rceil \right\dbr \right).
\end{align*}
Taking $\liminf_{n \to \infty} \frac{1}{n} \log$ on both sides of the first inequality and $\limsup_{n \to \infty} \frac{1}{n} \log$ on both sides of the second inequality and using Equation~(\ref{eq_lemma_exp_decay_rate_homogeneous_case_end}) of Lemma~\ref{lemma_exp_decay_rate_homogeneous_case}, we obtain inequalities (\ref{eq_proposition_exponential_rate_of_decay_final_lower}) and (\ref{eq_proposition_exponential_rate_of_decay_final_upper}).
\end{proof}

\section{The lower bound}\label{section_lower_bound}

The aim of this section is to establish inequality (\ref{eq_prop_lower_bound_RL}) of Proposition~\ref{proposition_RL_functions}, which we restate below.

\begin{proposition}\label{prop_lower_bound}
For all $\mu \in \mathcal{P}(\overline{\ZZ})$, 
\begin{equation*}\label{eq_prop_lower_bound}
-I(\mu) \leq \lim_{\varepsilon \to 0^+} \liminf_{n \to \infty} \frac{1}{n} \log \PP\big( \ell_n \in B(\mu,\varepsilon) \big),
\end{equation*}
where the rate function $I$ is defined in (\ref{eq_def_rate_fuction}).
\end{proposition}

\noindent To obtain such a lower bound, we define a set of typical trajectories whose probability decays at the right exponential order and which is included in the event whose probability we aim at estimating. Throughout Sections~\ref{subsection_construction_of_typical_trajectories} and \ref{subsection_properties_of_typical_trajectories}, we fix in the following order
\begin{enumerate}
\item $\mu=\sum_{\sigma \in \{-,0,+\}} \alpha_{\sigma} \mu_{\sigma} \in \mathcal{P}(\overline{\ZZ})$ and  $0<\varepsilon<\min\left\{\frac{1}{9},\varepsilon^*\right\}$, where $\varepsilon^*$ is defined after \eqref{eq_for_lemma_G_varepsilon_to_zero},
\item $R>\max\{R_{\mu_0}(\varepsilon),R_{\overline{\ZZ}}(\varepsilon)\}$, where $R_{\overline{\ZZ}}(\varepsilon)$ is defined in (\ref{eq_def_R_Z_varepsilon}) and $R_{\mu_0}(\varepsilon)$ in \eqref{eq_def_R_mu_0_varepsilon},
\item $n > \frac{4R+9}{\varepsilon}$.
\end{enumerate}

\subsection{Construction of the typical trajectories}\label{subsection_construction_of_typical_trajectories}

For $\sigma \in \{-,0,+\}$, set
\begin{align}\label{eq_def_times_for_typical_traj}
t^{\sigma} \eqdef \max \left\{ 2 \left \lceil \frac{(\alpha_{\sigma}-3\varepsilon) n}{2} \right \rceil+1 \, , \, 1 \right\}.
\end{align}
The quantities $t^0$, $t^-$ and $t^+$ are times which will be used in the construction of the set of typical trajectories and are essentially the times spent by these trajectories in $A^0$, $A^-$ and $A^{+}$ respectively. Note that $t^{\sigma} \simeq \alpha_{\sigma} n$. We will make one construction for each class of trajectories $\mathscr{C}_{n}^{-}$ and $\mathscr{C}_n^{+}$. The class $\mathscr{C}_{n}^{0}$ will not play a role in our rate function and so we will not describe a construction for it. For the sake of readability, we introduce the following notations:
\begin{align*}
\Omega_{\textup{c}} = \left\{
\begin{array}{cl}
\Omega_{t^0}^{(0)}(\mu_0,2^{-R} \varepsilon) & \text{, if } \alpha_0>3 \varepsilon \\
\{(0)\} & \text{, if } \alpha_0 \leq 3 \varepsilon 
\end{array}
\right. \qquad , \qquad \Omega_{\textup{exc}}^{\sigma} = \Omega_{t^{\sigma},\textup{exc}}^{\sigma} \qquad \text{and} \qquad
\Omega_{\textup{mea}}^{\sigma} = \Omega_{t^{\sigma},\textup{mea}}^{\sigma}. 
\end{align*}
where we recall these sets are defined respectively in Equations~(\ref{eq_def_trajectories_close_to_mu}), (\ref{eq_def_trajectories_excursions_above_R}) and (\ref{eq_def_trajectories_end_above_R}). Note that when $\alpha_{\sigma} \leq 3 \varepsilon$, then $t^{\sigma}=1$. This is why, in this case and for $\sigma=0$, $\Omega_{\textup{c}}$ consists of a single trajectory: the one-letter word $(0)$. To give a more concrete description, the set $\Omega_{\textup{exc}}^{+}$ consists of trajectories which are excursions of length $t^{+} \simeq \alpha_{+} n$ above $R$ and the set $\Omega_{\textup{mea}}^{-}$ consists of meanders of length $t^{-} \simeq \alpha_{-}n$ below $-R$.

To construct the typical trajectories, we assemble words from the sets defined above, linking them with short intermediate subwords, which we define next. For each $\sigma \in \{-,+\}$, let
\begin{equation}\label{eq_def_xi}
\xi^{\sigma}(k) \eqdef 
\begin{cases}
(k+\sigma 1,k+\sigma 2, \cdots , \sigma(R-2),\sigma(R-1)), & \text{if } \sigma k \leq R, \\
(k-\sigma 1,k-\sigma 2, \cdots , \sigma(R+2),\sigma(R+1)), & \text{otherwise.}
\end{cases}
\end{equation}
with the convention that the word is empty when $k=\sigma(R-1)$ or $k=\sigma(R+1)$ and $\xi^{\sigma}(\sigma R)=(\sigma(R-1))$. We also define
\begin{equation}\label{eq_def_chi}
\chi^{\sigma} \eqdef \big(-\sigma (R-1),\cdots,-\sigma 1,0,\sigma 1, \cdots ,\sigma(R-1) \big) = \xi^{\sigma}(-\sigma R).
\end{equation}
The word $\xi^{\sigma}(k)$ is a word connecting the letter $k$ to the letter $\sigma R$ and $\chi^{\sigma}$ is a word connecting $-\sigma R$ to $\sigma R$. The following lemma allows us to bound the length of these connecting words. 

\begin{lemma}[Bound on the last letter of $v^{\textup{c}}$]\label{lemma_bound_on_last_point_if_close_to_mu_c}
If $v^{\textup{c}} \in \Omega_{\textup{c}}$, then
\begin{equation*}\label{eq_lemma_bound_on_last_point_if_close_to_mu_c}
|v^{\textup{c}}_{-1}| \leq R+2 \varepsilon n.
\end{equation*}
\end{lemma}

\begin{proof}
If $\alpha_0 \leq 3 \varepsilon$, then $v^{\textup{c}}=(0)$ and the result is clear. Otherwise, $v^{\textup{c}} \in \Omega_{t^0}^{(0)}(\mu_0,2^{-R}\varepsilon)$. Since $R>R_{\mu_0}(\varepsilon)$, we can apply Corollary~\ref{corollary_bound_on_last_point_if_close_to_mu_c_appendix} to obtain
\begin{equation*}
|v^{\textup{c}}_{-1}| \leq R+2 \varepsilon t^0 \leq R+2\varepsilon n. \qedhere
\end{equation*}
\end{proof} 

\begin{lemma}\label{lemma_existence_b_lower_bound}
For all $\sigma \in \{-,+\}$ and $ \big(v^{\textup{c}},v^{\textup{exc}},v^{\textup{mea}}\big) \in \Omega_{\textup{c}} \times \Omega_{\textup{exc}}^{-\sigma} \times \Omega_{\textup{mea}}^{\sigma}$, there exists a word $b \in \Omega_{\textup{fin}}$ such that
\begin{enumerate}
\item $v^{\textup{mea}}_{-1} \sim b_1$,
\item $|v^{\textup{c}}|+|\xi^{-\sigma}(v^{\textup{c}}_{-1})|+|v^{\textup{exc}}|+|\chi^{\sigma}|+|v^{\textup{mea}}|+|b|=n$,
\item For all $j \in \dbl 1,|b| \dbr$, $b_j \in A^{\sigma}$.
\end{enumerate}
\end{lemma}

\begin{proof}
Let us start by showing that for all $\sigma \in \{-,+\}$ and $ \big(v^{\textup{c}},v^{\textup{exc}},v^{\textup{mea}}\big) \in \Omega_{\textup{c}} \times \Omega_{\textup{exc}}^{-\sigma} \times \Omega_{\textup{mea}}^{\sigma}$, 
\begin{align}\label{eq_subwords_and_connecting_not_too_long}
|v^{\textup{c}}|+|\xi^{-\sigma}(v^{\textup{c}}_{-1})|+|v^{\textup{exc}}|+|\chi^{\sigma}|+|v^{\textup{mea}}|<n.
\end{align}
First, if we had $\alpha_{\sigma} \leq 3 \varepsilon$ for all $\sigma \in \{-,0,+\}$, then, using that $\varepsilon<\frac{1}{9}$, we would have
\begin{align*}
1=\sum_{\sigma \in \{-,0,+\}} \alpha_{\sigma} \leq 9 \varepsilon < 1,
\end{align*}
which is a contradiction. Let $\sigma^* \in \{-,0,+\}$ be such that $\alpha_{\sigma^*} \geq 3 \varepsilon$. The lengths of the subwords $v^{\textup{c}}$, $v^{\textup{exc}}$ and $v^{\textup{mea}}$ are given respectively by the times $t^0$, $t^{-\sigma}$ and $t^{\sigma}$ defined in Equation~(\ref{eq_def_times_for_typical_traj}). We will use the following bounds: 
\begin{align*}
t^{\sigma^*} \leq (\alpha_{\sigma^*}-3\varepsilon)n+3 \qquad \text{and} \qquad t^{\sigma} \leq \alpha_{\sigma}n+3 \quad \text{for } \sigma \neq \sigma^*.
\end{align*}
We can also bound the lengths of the connecting subwords $\xi^{\sigma}(v^{\textup{c}}_{-1})$ and $\chi^{\sigma}$. Since $v^{\textup{c}} \in \Omega_{\textup{c}}$, we have by Lemma~\ref{lemma_bound_on_last_point_if_close_to_mu_c} that $|v^{\textup{c}}_{-1}| \leq 2 \varepsilon n + R$. Thus, using Equation~(\ref{eq_def_xi}), we obtain for all $\sigma \in \{-,+\}$ that 
\begin{align*}
|\xi^{-\sigma}(v^{\textup{c}}_{-1})| & \leq |v^{\textup{c}}_{-1}|+R-1 \\
& \leq 2 \varepsilon n + 2R.
\end{align*}
We also directly get from Equation~(\ref{eq_def_chi}) that
\begin{align*}
|\chi^{\sigma}| \leq 2R.
\end{align*}
Thus, using that $\alpha_0+\alpha_{+}+\alpha_{-}=1$, we obtain
\begin{align*}
|v^{\textup{c}}|+|\xi^{-\sigma}(v^{\textup{c}}_{-1})|+|v^{\textup{exc}}|+|\chi^{\sigma}|+|v^{\textup{mea}}| & \leq t^0 + 2 \varepsilon n + 2R +t^{-\sigma} + 2R + t^{\sigma} \\
& \leq (\alpha_{\sigma^*}-3\varepsilon)n+\sum_{\sigma \neq \sigma^*}\alpha_{\sigma}n+2 \varepsilon n+ 4R+9 \\
& = n + 4R+9-\varepsilon n.
\end{align*}
Since $n>\frac{4R+9}{\varepsilon}$, we obtain Equation~(\ref{eq_subwords_and_connecting_not_too_long}). Therefore,
\begin{align*}
h \eqdef n-|v^{\textup{c}}|-|\xi^{-\sigma}(v^{\textup{c}}_{-1})|-|v^{\textup{exc}}|-|\chi^{\sigma}|-|v^{\textup{mea}}|>0.
\end{align*}
Let $b=\big(v^{\textup{mea}}_{-1}+\sigma 1 \, , \cdots , \, v^{\textup{mea}}_{-1}+\sigma h \big)$. One can verify that $b$ satisfies the three required properties.
\end{proof}

\begin{definition}\label{def_map_psi_sigma}
For $\sigma \in \{-,+\}$ we define the map $\psi^{\sigma}$ by setting
\begin{equation*}\label{eq_def_typical_trajectory_function}
\begin{array}{cccc}
\psi^{\sigma} & : \, \Omega_{\textup{c}} \times \Omega_{\textup{exc}}^{-\sigma} \times \Omega_{\textup{mea}}^{\sigma} & \longrightarrow & \mathscr{C}_{n}^{\sigma} \\
& \big(v^{\textup{c}},v^{\textup{exc}},v^{\textup{mea}}\big) & \longmapsto & v^{\textup{c}} \cdot \xi^{-\sigma}(v^{\textup{c}}_{-1}) \cdot v^{\textup{exc}} \cdot \chi^{\sigma} \cdot v^{\textup{mea}} \cdot b
\end{array},
\end{equation*}
where $b \in \Omega_{\textup{fin}}$ is an arbitrary word which satisfies the properties \textit{1.}-\textit{3.} of Lemma~\ref{lemma_existence_b_lower_bound}.
\end{definition}

\begin{remark}\label{remark_psi_well_defined}
Let us verify that $\psi^{\sigma}(v^{\textup{c}},v^{\textup{exc}},v^{\textup{mea}}) \in \mathscr{C}_n^{\sigma}$. First, the length of $b$ ensures that the concatenated word is of length $n$. Then, by construction of $\xi^{-\sigma}$, $\chi^{\sigma}$ and $b$ we also have $v^{\textup{c}}_{-1} \sim \xi^{-\sigma}(v^{\textup{c}}_{-1})_1$, $\xi^{-\sigma}(v^{\textup{c}}_{-1})_{-1} \sim v^{\textup{exc}}_1$, $v^{\textup{exc}}_{-1} \sim \chi^{\sigma}_1$, $\chi^{\sigma}_{-1} \sim v^{\textup{mea}}_1$ and $v^{\textup{mea}}_{-1} \sim b_1$. Thus, $\psi^{\sigma}(v^{\textup{c}},v^{\textup{exc}},v^{\textup{mea}}) \in \Omega_n$. We also have $v^{\textup{c}}_1=0$ and so, as required, the concatenated word starts at $0$. Finally, by property \textit{3.} of Lemma~\ref{lemma_existence_b_lower_bound}, $b_{-1} \in A^{\sigma}$. Thus, $\psi^{\sigma}( v^{\textup{c}},v^{\textup{exc}},v^{\textup{mea}})_{-1} \in A^{\sigma}$ and the range of $\psi^{\sigma}$ is indeed included in $\mathscr{C}_n^{\sigma}$.
\end{remark}

\noindent For all $\sigma \in \{-,+\}$, we define the \textit{typical trajectories} of the class $\sigma$, as the set
\begin{equation}\label{eq_def_typical_traj}
E_{\textup{typ}}^{\sigma} \eqdef \psi^{\sigma}\big( \Omega_{\textup{c}} \times \Omega_{\textup{exc}}^{-\sigma} \times \Omega_{\textup{mea}}^{\sigma}\big) \subseteq \mathscr{C}_{n}^{\sigma}.
\end{equation}
For instance, the trajectories in $E_{\textup{typ}}^{+}$ are the trajectories whose empirical measures first approximate $\mu_0$, then make an excursion below $-R$ and finally stay above $R$. Figure~\ref{figure_typical_traj_class_minus_one} illustrates a typical trajectory for the class $\mathscr{C}_n^-$.

\begin{figure}[ht]
    \centering
    \includegraphics[width=\textwidth]{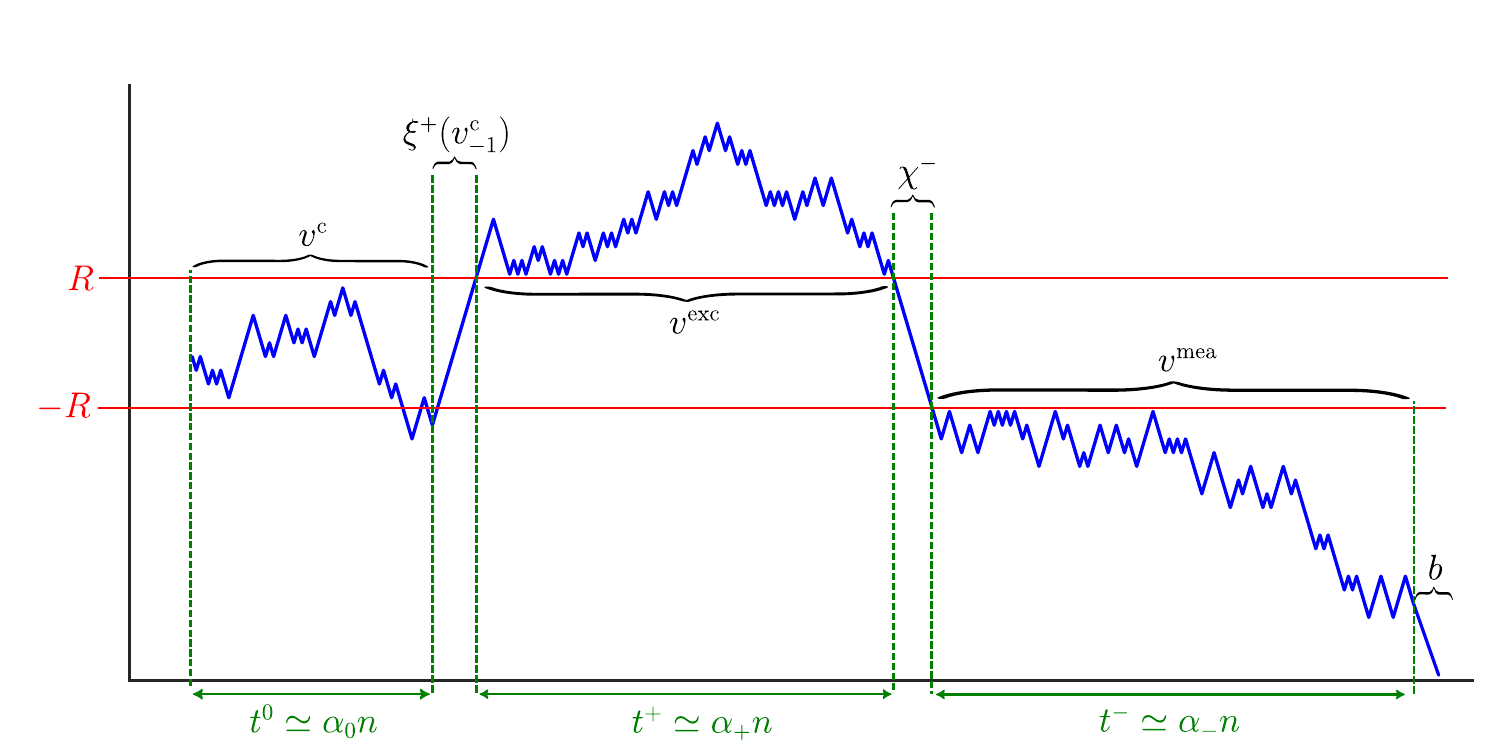}
    \caption{Example of a trajectory $w \in E_{\text{typ}}^{-}$.}
    \label{figure_typical_traj_class_minus_one}
\end{figure}

The role of the words $\xi^{-\sigma}$ and $\chi^{\sigma}$ is to connect $v^{\textup{c}}$, $v^{\textup{exc}}$ and $v^{\textup{mea}}$ together and the role of $b$ is to ensure that the length of the concatenated word $\psi^{\sigma}(v^{\textup{c}},v^{\textup{exc}},v^{\textup{mea}})$ is $n$. These words will not play a major role at the level of the empirical measure. Indeed, the following lemma gives a bound on their relative length.

\begin{lemma}[Connecting subwords]\label{lemma_bound_length_xi_chi_b}
For all $\sigma \in \{-,+\}$ and $\big(v^{\textup{c}},v^{\textup{exc}},v^{\textup{mea}}\big) \in \Omega_{\textup{c}} \times \Omega_{\textup{exc}}^{-\sigma} \times \Omega_{\textup{mea}}^{\sigma}$, 
\begin{equation}\label{eq_lemma_bound_length_xi_chi_b}
|\xi^{-\sigma}(v^{\textup{c}}_{-1})|+|\chi^{\sigma}|+|b| \leq 9 \varepsilon n.
\end{equation}
\end{lemma}

\begin{proof}
For each $\sigma \in \{-,0,+\}$, we have the lower bound $t^{\sigma} \geq (\alpha_{\sigma}-3\varepsilon)n$. Thus, using the fact that $\alpha_0+\alpha_{-\sigma}+\alpha_{\sigma}=1$ and the definition of $b$, we obtain
\begin{align*}
|\xi^{-\sigma}(v^{\textup{c}}_{-1})|+|\chi^{\sigma}|+|b| & = n-|v^{\textup{c}}|-|v^{\textup{exc}}|-|v^{\textup{mea}}| \\
& \leq n-(\alpha_0-3\varepsilon)n-(\alpha_{-}-3\varepsilon) n-(\alpha_{+}-3\varepsilon) n  \\
& = 9\varepsilon n. \qedhere
\end{align*}
\end{proof}

\begin{lemma}[Injectivity of $\psi^{\sigma}$]\label{lemma_psi_R_n_is_injective}
For all $\sigma \in \{-,+\}$, the map $\psi^{\sigma}$ is injective.
\end{lemma}

\begin{proof}
Suppose $(v^{\textup{c}},v^{\textup{exc}},v^{\textup{mea}}) \, , \, (u^{\textup{c}},u^{\textup{exc}},u^{\textup{mea}}) \in \Omega_{\textup{c}} \times \Omega_{\textup{exc}}^{-\sigma} \times \Omega_{\textup{mea}}^{\sigma}$ are such that
\begin{align*}
\psi^{\sigma}(v^{\textup{c}},v^{\textup{exc}},v^{\textup{mea}})=\psi^{\sigma}(u^{\textup{c}},u^{\textup{exc}},u^{\textup{mea}}).
\end{align*}
There exist $b,b' \in \Omega_{\textup{fin}}$ such that
\begin{align*}
\psi^{\sigma}(v^{\textup{c}},v^{\textup{exc}},v^{\textup{mea}})& =v^{\textup{c}} \cdot \xi^{-\sigma}(v^{\textup{c}}_{-1}) \cdot v^{\textup{exc}} \cdot \chi^{\sigma} \cdot v^{\textup{mea}} \cdot b \\
\text{and} \qquad \psi^{\sigma}(u^{\textup{c}},u^{\textup{exc}},u^{\textup{mea}}) & = u^{\textup{c}} \cdot \xi^{-\sigma}(u^{\textup{c}}_{-1}) \cdot u^{\textup{exc}} \cdot \chi^{\sigma} \cdot u^{\textup{mea}} \cdot b'.
\end{align*}
Since $|v^{\textup{c}}|=t^0=|u^{\textup{c}}|$ we deduce that $v^{\textup{c}}=u^{\textup{c}}$. Since $v^{\textup{c}}_{-1}=u^{\textup{c}}_{-1}$, we obtain that $\xi^{-\sigma}(v^{\textup{c}}_{-1})=\xi^{-\sigma}(u^{\textup{c}}_{-1})$ and in particular their lengths match. We can then deduce that $v^{\textup{exc}}=u^{\textup{exc}}$. Similarly, we can again show that $v^{\textup{mea}}=u^{\textup{mea}}$ and thus $\psi^{\sigma}$ is indeed injective.
\end{proof}

\subsection{Properties of the typical trajectories}\label{subsection_properties_of_typical_trajectories}

In the previous section, for each class $\mathscr{C}_n^{\sigma}$ and each measure $\mu \in \mathcal{P}(\overline{\ZZ})$, we have constructed  a corresponding set of typical trajectories $E_{\textup{typ}}^{\sigma}$. We now describe their properties. The following lemma shows that any typical trajectory lies in the appropriate class and that its empirical measure approximates the target measure $\mu$. 

\begin{lemma}[Typical trajectories are well constructed]\label{lemma_typical_events_included}
For each $\sigma \in \{-,+\}$,
\begin{equation*}\label{eq_lemma_typical_events_included}
E_{\textup{typ}}^{\sigma} \subseteq \mathscr{C}_{n}^{\sigma}(\mu,22 \varepsilon).
\end{equation*}
\end{lemma}

\begin{proof}
Consider $w \in E^{\sigma}_{\textup{typ}}$. We want to show that $\| \ell(w)-\mu\|<22 \varepsilon$. By definition of $E^{\sigma}_{\textup{typ}}$ and injectivity of $\psi^{\sigma}$, there exists a unique triple $(v^{\textup{c}},v^{\textup{exc}},v^{\textup{mea}}) \in \Omega_{\textup{c}} \times \Omega_{\textup{exc}}^{-\sigma} \times \Omega_{\textup{mea}}^{\sigma}$ such that $w=\psi^{\sigma}(v^{\textup{c}},v^{\textup{exc}},v^{\textup{mea}})$. By decomposing the empirical measure of $w$ according to the decomposition of $w$ into the subwords $w=v^{\textup{c}} \cdot \xi^{-\sigma} \cdot v^{\textup{exc}} \cdot \chi^{\sigma} \cdot v^{\textup{mea}} \cdot b$, we obtain
\begin{align}
\| \ell(w)-\mu \| \notag & = \left\| \frac{1}{n} \left( \sum_{j=1}^{t^0} \delta_{v^{\textup{c}}_j} + \sum_{j=1}^{|\xi^{-\sigma}|} \delta_{\xi^{-\sigma}_j} + \sum_{j=1}^{t^{-\sigma}} \delta_{v^{\textup{exc}}_j} + \sum_{j=1}^{|\chi^{\sigma}|} \delta_{\chi^{\sigma}_j} + \sum_{j=1}^{t^{\sigma}} \delta_{v^{\textup{mea}}_j} + \sum_{j=1}^{|b|} \delta_{b_j} \right) - \mu \right\| \notag \\
& \leq \left\| \frac{1}{n} \sum_{j=1}^{t^0} \delta_{v^{\textup{c}}_j} - \alpha_0 \mu_0 \right\| + \left\| \frac{1}{n} \sum_{j=1}^{t^{-\sigma}} \delta_{v^{\textup{exc}}_j} - \alpha_{-\sigma} \delta_{-\sigma \infty} \right\|+\left\| \frac{1}{n} \sum_{j=1}^{t^{\sigma}} \delta_{v^{\textup{mea}}_j} - \alpha_{\sigma}\delta_{\sigma \infty} \right\| \label{eq_lemma_typ_traj_good_emp_measure_main_term} \\
& \quad +\frac{|\xi^{-\sigma}|+|\chi^{\sigma}|+|b|}{n}, \notag 
\end{align}
where in the last line we have used Equation~(\ref{eq_lemma_properties_KR_norm_of_proba}). We now treat each term on the right-hand side of inequality (\ref{eq_lemma_typ_traj_good_emp_measure_main_term}) separately. Consider first the case where $\alpha_0>3 \varepsilon$. Since $\alpha_0 \neq 0$, we have
\begin{align*}
\alpha_0 \left\| \frac{1}{\alpha_0 n} \sum_{j=1}^{t^0} \delta_{v^{\textup{c}}_j} - \mu_0 \right\| & \leq \alpha_0 \left| \frac{t^0}{\alpha_0 n}-1 \right| + \alpha_0 \left\| \frac{1}{t^0}\sum_{j=1}^{t^0} \delta_{v^{\textup{c}}_j} - \mu_0 \right\|.
\end{align*}
Since $\alpha_0>3\varepsilon$ and $v^{\textup{c}} \in \Omega_{\textup{c}}$, we have $\|\ell(v^{\textup{c}})-\mu_0 \| < 2^{-R}\varepsilon \leq \varepsilon$. In addition, by definition of $t^0$ and the fact that $n>\frac{4R+9}{\varepsilon}>\frac{1}{3\varepsilon}$, we deduce that $|\alpha_0 n-t^0|\leq 3\varepsilon n$. Therefore,
\begin{equation*}
\left\| \frac{1}{n} \sum_{j=1}^{t^0} \delta_{v^{\textup{c}}_j} - \alpha_0 \mu_0 \right\| < 3 \varepsilon + \alpha_0 \varepsilon.
\end{equation*}
If $\alpha_0 \leq 3 \varepsilon$, then $v^c=(0)$ and we obtain by using Equation~(\ref{eq_lemma_properties_KR_norm_of_proba}) that
\begin{align*}
\left\| \frac{1}{n} \delta_{v^{\textup{c}}_1}-\alpha_0 \mu_0 \right\| & \leq \frac{1}{n} + \alpha_0 \leq \frac{1}{n} + 3 \varepsilon. 
\end{align*}
Thus, in both cases, we obtain
\begin{equation}\label{eq_lemma_typ_traj_good_emp_measure_bound_first_term}
\left\| \frac{1}{n} \sum_{j=1}^{t^0} \delta_{v^{\textup{c}}_j} - \alpha_0 \mu_0 \right\| < 4 \varepsilon + \frac{1}{n}.
\end{equation}
Let us now consider the second term of (\ref{eq_lemma_typ_traj_good_emp_measure_main_term}). Since $v^{\textup{exc}} \in \Omega_{\textup{exc}}^{-\sigma}$, we have $v^{\textup{exc}}_j \in A^{-\sigma}$ for all $j \in \dbl 1,t^{-\sigma} \dbr$. Thus, by applying (\ref{eq_lemma_properties_KR_norm_distance_dirac}) and then (\ref{eq_lemma_proximity_to_infty}) with the fact that $R \geq R_{\overline{\ZZ}}(\varepsilon)$, we deduce that,
\begin{equation}\label{eq_excursion_points_close_to_infinity}
\big\| \delta_{-\sigma \infty} - \delta_{v^{\textup{exc}}_j} \big\| \leq d_{\overline{\ZZ}}\left(-\sigma \infty \, , \, v^{\textup{exc}}_j \right) < \varepsilon.
\end{equation} 
If $\alpha_{-\sigma}>3\varepsilon$, then $\vert t^{-\sigma}-\alpha_{-\sigma}n \vert \leq 3\varepsilon n$. Thus, we obtain with the help of Equation~(\ref{eq_excursion_points_close_to_infinity}) that
\begin{align*}
\alpha_{-\sigma} \left\| \frac{1}{\alpha_{-\sigma} n} \sum_{j=1}^{t^{-\sigma}} \delta_{v^{\textup{exc}}_j} - \delta_{-\sigma \infty} \right\| & \leq \alpha_{- \sigma} \left \vert \frac{t^{-\sigma}}{\alpha_{-\sigma}n} -1  \right \vert + \frac{\alpha_{-\sigma}}{t^{-\sigma}} \sum_{j=1}^{t^{-\sigma}} \left \| \delta_{v^{\textup{exc}}_j}-\delta_{-\sigma \infty} \right \| \\
& < 3 \varepsilon + \alpha_{-\sigma} \varepsilon. 
\end{align*}
If $\alpha_{-\sigma}\leq 3 \varepsilon$, then $t^{-\sigma}=1$ and
\begin{align*}
\left\| \frac{1}{n} \delta_{v^{\textup{exc}}_1}-\alpha_{-\sigma}\delta_{-\sigma\infty} \right\| \leq \frac{1}{n}+3\varepsilon.
\end{align*}
In any case, we have
\begin{equation}\label{eq_lemma_typ_traj_good_emp_measure_bound_second_term}
\left\| \frac{1}{n} \sum_{j=1}^{t^{-\sigma}} \delta_{v^{\textup{exc}}_j} - \alpha_{-\sigma} \delta_{-\sigma \infty} \right\| < 4\varepsilon + \frac{1}{n}.
\end{equation}

The third term is dealt with in a similar way. Indeed, since $R \geq R_{\overline{\ZZ}}(\varepsilon)$ and $v^{\textup{mea}}_j \in A^{\sigma}$, we obtain that $\|\delta_{\sigma \infty}-\delta_{v^{\textup{mea}}_j} \|<\varepsilon$ for all $j \in \dbl 1,t^{\sigma} \dbr$. With the same arguments, one finds that
\begin{equation}\label{eq_lemma_typ_traj_good_emp_measure_bound_third_term}
\left\| \frac{1}{n} \sum_{j=1}^{t^{\sigma}} \delta_{v^{\textup{mea}}_j} - \alpha_{\sigma} \delta_{\sigma \infty} \right\| < 4\varepsilon + \frac{1}{n}.
\end{equation}
Finally, a bound on the fourth term of (\ref{eq_lemma_typ_traj_good_emp_measure_main_term}) is given by Equation~(\ref{eq_lemma_bound_length_xi_chi_b}) of Lemma~\ref{lemma_bound_length_xi_chi_b}. Plugging this with (\ref{eq_lemma_typ_traj_good_emp_measure_bound_first_term}), (\ref{eq_lemma_typ_traj_good_emp_measure_bound_second_term}),  and (\ref{eq_lemma_typ_traj_good_emp_measure_bound_third_term}) into (\ref{eq_lemma_typ_traj_good_emp_measure_main_term}), we obtain
\begin{align*}
\| \ell(w)-\mu \| & < 21 \varepsilon + \frac{3}{n} < 22\varepsilon,
\end{align*}
where the last inequality comes from the fact that $n>\frac{4R+9}{\varepsilon}>\frac{3}{\varepsilon}$. Thus, we have shown that $\ell(w) \in B(\mu,22\varepsilon)$. In addition, by definition of $\psi^{\sigma}$, we have $w \in \mathscr{C}_{n}^{\sigma}$. Therefore, we obtain the desired inclusion and the lemma is proved.
\end{proof}

\noindent The following lemma relates the probability we are interested in to the product of the three terms whose exponential rate of decay have been computed in Section~\ref{section_regional_estimates}. We introduce the quantity
\begin{equation}\label{eq_def_p_star_R}
p_{*} \eqdef \inf \big\{ p(k),1-p(k) \, : \, k \in \ZZ \big\} >0 \, ,
\end{equation}
where the strict positivity of $p_*$ comes from Hypothesis~\ref{hypothesis_on_transition_function_p}.

\begin{lemma}[Lower bound in three blocks]\label{lemma_lower_bound_decomposition_three_terms}
For each $\sigma \in \{-,+\}$,
\begin{equation*}\label{eq_lemma_lower_bound_decomposition_three_terms}
\PP\big( \mathscr{C}_{n}^{\sigma}(\mu,22 \varepsilon) \big) \geq p_{*}^{10\varepsilon n} \; \PP \left(  \Omega_{\textup{c}} \right) \, \PP_{-\sigma R}\left(  \Omega_{\textup{exc}}^{-\sigma} \right) \, \PP_{\sigma R}\left( \Omega_{\textup{mea}}^{\sigma} \right).
\end{equation*}
\end{lemma}

\begin{proof}
Let us first establish an analogous lower bound at the level of individual trajectories. We recall that for all $w \in \Omega_n$, we defined $\mathfrak{p}(w)$ in (\ref{eq_def_probability_of_word}). As mentioned in (\ref{eq_link_p_w_with_probability}), $\mathfrak{p}(w)$ is the product of all the transitions $P(w_i,w_{i+1})$ constituting the trajectory. Let us first show that if $(v^{\textup{c}},v^{\textup{exc}} ,v^{\textup{mea}}) \in \Omega_{\textup{c}} \times \Omega_{\textup{exc}}^{-\sigma} \times \Omega_{\textup{mea}}^{\sigma}$, then
\begin{equation}\label{eq_probability_individual_traj_lower_bound}
\mathfrak{p}\big( \psi^{\sigma}(v^{\textup{c}},v^{\textup{exc}} ,v^{\textup{mea}}) \big) \geq p_{*}^{10\varepsilon n} \, \mathfrak{p}(v^{\textup{c}}) \, \mathfrak{p}(v^{\textup{exc}}) \, \mathfrak{p}(v^{\textup{mea}}).
\end{equation}
We decompose $\mathfrak{p}\big( \psi^{\sigma}(v^{\textup{c}},v^{\textup{exc}} ,v^{\textup{mea}}) \big)$ according to the transitions forming the trajectory, grouping together those that constitute the subwords $v^{\textup{c}}$, $v^{\textup{exc}}$ and $v^{\textup{mea}}$. In the case where $\xi^{-\sigma}$ is not empty, we obtain
\begin{align*}
\mathfrak{p}\big( \psi^{\sigma}(v^{\textup{c}},v^{\textup{exc}} ,v^{\textup{mea}}) \big) & = \mathfrak{p}(v^{\textup{c}}) \, P(v^{\textup{c}}_{-1},\xi^{-\sigma}_{1}) \, \mathfrak{p}(\xi^{-\sigma}) \, P(\xi^{-\sigma}_{-1},v^{\textup{exc}}_1) \, \mathfrak{p}(v^{\textup{exc}}) \, P(v^{\textup{exc}}_{-1},\chi^{\sigma}_1) \, \mathfrak{p}(\chi^{\sigma}) \, P(\chi^{\sigma}_{-1},v^{\textup{mea}}_1) \\
& \quad \times \mathfrak{p}(v^{\textup{mea}}) \, P(v^{\textup{mea}}_{-1},b_1) \, \mathfrak{p}(b). 
\end{align*}
There are $|\xi^{-\sigma}|+|\chi^{\sigma}|+|b|+2$ transitions involving one of the letters $\xi^{-\sigma}_j$, $\chi^{\sigma}_j$, or $b_j$. Each of these transitions is bounded below by $p_*$, and therefore we obtain
\begin{equation}\label{eq_ineq_for_individual_traj_techn}
\mathfrak{p}\big( \psi^{\sigma}(v^{\textup{c}},v^{\textup{exc}} ,v^{\textup{mea}}) \big) \geq p_{*}^{|\xi^{-\sigma}|+|\chi^{\sigma}|+|b|+2} \, \mathfrak{p}(v^{\textup{c}}) \, \mathfrak{p}(v^{\textup{exc}}) \, \mathfrak{p}(v^{\textup{mea}}).
\end{equation}
If $\xi^{-\sigma}$ is empty, the same reasoning applies, yielding
\begin{align*}
& \mathfrak{p}\big( \psi^{\sigma}(v^{\textup{c}},v^{\textup{exc}} ,v^{\textup{mea}}) \big) \\
& \qquad \qquad =\mathfrak{p}(v^{\textup{c}}) \, P(v^{\textup{c}}_{-1},v^{\textup{exc}}_1) \, \mathfrak{p}(v^{\textup{exc}}) \, P(v^{\textup{exc}}_{-1},\chi^{\sigma}_1) \, \mathfrak{p}(\chi^{\sigma}) \, P(\chi^{\sigma}_{-1},v^{\textup{mea}}_1) \, \mathfrak{p}(v^{\textup{mea}}) \, P(v^{\textup{mea}}_{-1},b_1) \, \mathfrak{p}(b)
\end{align*}
and thus Equation~(\ref{eq_ineq_for_individual_traj_techn}) continues to hold. Using (\ref{eq_lemma_bound_length_xi_chi_b}) and the fact that $n>\frac{4R+9}{\varepsilon}>\frac{2}{\varepsilon}$ we obtain $|\xi^{-\sigma}|+|\chi^{\sigma}|+|b|+2 \leq 10\varepsilon n$. Substituting this inequality into \eqref{eq_ineq_for_individual_traj_techn} yields \eqref{eq_probability_individual_traj_lower_bound}.

We can now turn towards the claim of the lemma. Using the inclusion obtained in Lemma~\ref{lemma_typical_events_included}, we have for each $\sigma \in \{-,+\}$,
\begin{align*}
\PP\big( \mathscr{C}_{n}^{\sigma}(\mu,22 \varepsilon) \big) \geq \PP\big(E_{\textup{typ}}^{\sigma}\big).
\end{align*}
Decomposing the event on the right hand side of the above inequality in terms of the individual trajectories constituting it, we obtain for all $\sigma \in \{-,+\}$,
\begin{align*}
\PP\big( \mathscr{C}_{n}^{\sigma}(\mu,22 \varepsilon) \big) \geq \sum_{w \in E_{\textup{typ}}^{\sigma}} \mathfrak{p}(w).
\end{align*}
Using the injectivity of $\psi^{\sigma}$ obtained in Lemma~\ref{lemma_psi_R_n_is_injective} and that $E_{\textup{typ}}^{\sigma}=\psi^{\sigma}\left( \Omega_{\textup{c}} \times \Omega_{\textup{exc}}^{-\sigma} \times \Omega_{\textup{mea}}^{\sigma} \right)$, we obtain
\begin{align*}
\sum_{w \in E_{\textup{typ}}^{\sigma}} \mathfrak{p}(w)=\sum_{(v^{\textup{c}},v^{\textup{exc}},v^{\textup{mea}}) \in \Omega_{\textup{c}} \times \Omega_{\textup{exc}}^{-\sigma} \times \Omega_{\textup{mea}}^{\sigma}} \mathfrak{p}\big(\psi^{\sigma}(v^{\textup{c}},v^{\textup{exc}},v^{\textup{mea}})\big).
\end{align*}
The lower bound on $\mathfrak{p}\big( \psi^{\sigma}(v^{\textup{c}},v^{\textup{exc}} ,v^{\textup{mea}}) \big)$ obtained in (\ref{eq_probability_individual_traj_lower_bound}), then gives us
\begin{align*}
\PP\big(\mathscr{C}_{n}^{\sigma}(\mu,22 \varepsilon) \big) & \geq p_{*}^{10 \varepsilon n} \left( \sum_{v^{\textup{c}} \in \Omega_{\textup{c}}} \mathfrak{p}(v^{\textup{c}}) \right)\left( \sum_{v^{\textup{exc}} \in \Omega_{\textup{exc}}^{-\sigma}} \mathfrak{p}(v^{\textup{exc}})\right)\left( \sum_{v^{\textup{mea}} \in \Omega_{\textup{mea}}^{\sigma}} \mathfrak{p}(v^{\textup{mea}})\right).
\end{align*}
Then, with the help of Equation~(\ref{eq_def_probability_of_word}), 
\begin{align*}
\sum_{v^{\textup{c}} \in \Omega_{\textup{c}}} \mathfrak{p}(v^{\textup{c}}) = \PP \left( \Omega_{\textup{c}} \right) \quad , \quad \sum_{v^{\textup{exc}} \in \Omega_{\textup{exc}}^{-\sigma}} \mathfrak{p}(v^{\textup{exc}}) = \PP_{-\sigma R}\left( \Omega^{-\sigma}_{\textup{exc}} \right)  \quad \text{and} \quad \sum_{v^{\textup{mea}} \in \Omega_{\textup{mea}}^{\sigma}} \mathfrak{p}(v^{\textup{mea}}) = \PP_{\sigma R}\left(\Omega^{\sigma}_{\textup{mea}} \right).
\end{align*}
We thereby obtain the desired inequality and the lemma is proved.
\end{proof}

\subsection{Proof of the lower bound}\label{subsection_proof_of_the_lower_bound}

With the help of the results of Section~\ref{section_regional_estimates} and the decomposition given in Lemma~\ref{lemma_lower_bound_decomposition_three_terms}, we are ready to show the lower bound. Let us first reformulate the exponential rates
obtained in Propositions~\ref{proposition_exponential_rate_decay_I_DV} and \ref{proposition_exponential_rate_of_decay_exc_and_end} in a form better suited to prove Proposition~\ref{prop_lower_bound}.

\begin{lemma}[Lower bound at rate $\gamma$]\label{lemma_exponential_decay_with_gamma_n_lower}
Fix $\varepsilon>0$ and let $\gamma>0$ be a constant, which may depend on $\varepsilon$. Then, for all $R \geq R_{\overline{\ZZ}}(\varepsilon)$,
\begin{align}
-\gamma I_{\textup{Cr}}^{p_{\sigma}}(0) - \gamma g(\varepsilon) & \leq \liminf \limits_{n \to \infty} \frac{1}{n} \log \PP_{\sigma R}\left( \Omega_{2 \left \lceil \frac{\gamma n}{2} \right \rceil+1,\textup{exc}}^{\sigma} \right), \label{eq_lemma_exponential_rate_of_decay_exc_lower_gamma} \\
-\gamma \inf_{\sigma x \in [0,1]}I_{\textup{Cr}}^{p_{\sigma}}(x)  - \gamma g(\varepsilon) & \leq \liminf \limits_{n \to \infty} \frac{1}{n} \log \PP_{\sigma R}\left( \Omega_{2\left \lceil \frac{\gamma n}{2} \right \rceil+1,\textup{mea}}^{\sigma} \right), \label{eq_lemma_exponential_rate_of_decay_final_lower_gamma}
\end{align}
where $g$ is as in Proposition~\ref{proposition_exponential_rate_of_decay_exc_and_end}. In addition, for all $\mu_0 \in \mathcal{P}(\ZZ)$,
\begin{align}
-\gamma I_{\textup{DV}}(\mu_0)& \leq \lim_{R \to \infty} \liminf \limits_{n \to \infty} \frac{1}{n} \log \PP \left( \ell_{2 \left \lceil \frac{\gamma n}{2} \right \rceil+1} \in B(\mu_0,2^{-R}\varepsilon) \right). \label{eq_lemma_I_DV_with_gamma_n_lower}
\end{align}
\end{lemma}

\begin{proof}
Let us start with Equation~(\ref{eq_lemma_I_DV_with_gamma_n_lower}). Using the fact that $\lim_{n \to \infty} \frac{2 \left \lceil \frac{\gamma n}{2} \right \rceil+1}{n}=\gamma$ and then the fact that for all $k \in \NN$, $
\left\{2 \left \lceil \frac{\gamma n}{2} \right \rceil+1 \, : \, n \geq k \right\} \subseteq \left\{ n \in \NN \, : \, n \geq 2 \left \lceil \frac{\gamma k}{2} \right \rceil+1 \right\}$, we obtain for all $\varepsilon>0$,
\begin{align*}
\liminf_{n \to \infty} \frac{1}{n} \log \PP \left( \ell_{2 \left \lceil \frac{\gamma n}{2} \right \rceil+1} \in B(\mu_0,2^{-R}\varepsilon) \right)&=\gamma \liminf_{n \to \infty} \frac{1}{2 \left \lceil \frac{\gamma n}{2} \right \rceil+1} \log \PP \left( \ell_{2 \left \lceil \frac{\gamma n}{2} \right \rceil+1} \in B(\mu_0,2^{-R}\varepsilon) \right) \\
& \geq \gamma\liminf_{n \to \infty} \frac{1}{n} \log \PP \left( \ell_{n} \in B(\mu_0,2^{-R}\varepsilon) \right).
\end{align*}
We obtain Equation~(\ref{eq_lemma_I_DV_with_gamma_n_lower}) by taking $R$ to $+\infty$ and using Proposition~\ref{proposition_exponential_rate_decay_I_DV}.
Equations~(\ref{eq_lemma_exponential_rate_of_decay_exc_lower_gamma}) and (\ref{eq_lemma_exponential_rate_of_decay_final_lower_gamma}) are shown in the same way but this time using Equations~(\ref{eq_proposition_exponential_rate_of_decay_exc_lower}) and (\ref{eq_proposition_exponential_rate_of_decay_final_lower}) of Proposition~\ref{proposition_exponential_rate_of_decay_exc_and_end}.
\end{proof}

\begin{proof}[Proof of Proposition~\ref{prop_lower_bound}]
We recall that our aim is to show that
\begin{align*}
-I(\mu) \leq \lim \limits_{\varepsilon \to 0^+} \liminf \limits_{n \to \infty} \frac{1}{n} \log \PP\big( \ell_n \in B(\mu,\varepsilon) \big).
\end{align*}
Consider $0<\varepsilon < \varepsilon^*$, $R > \max\{R_{\overline{\ZZ}}(\varepsilon),R_{\mu_0}(\varepsilon)\}$ and $n>\frac{4R+9}{\varepsilon}$. First, since $\mathscr{C}_{n}^{\sigma}(\mu,22\varepsilon) \subseteq \Omega_n^{(0)}(\mu,22 \varepsilon)$, we deduce that for each $\sigma \in \{-,+\}$,
\begin{align*}
\PP\left(\mathscr{C}_{n}^{\sigma}(\mu,22 \varepsilon) \right) & \leq \PP\left(\Omega_n^{(0)}(\mu,22 \varepsilon) \right) \\
& = \PP\big(\ell_n \in B(\mu,22 \varepsilon) \big).
\end{align*}
Therefore, using Lemma~\ref{lemma_lower_bound_decomposition_three_terms}, we obtain that for each $\sigma \in \{-,+\}$,
\begin{align*}
\PP\big( \ell_n \in B(\mu,22 \varepsilon)\big) & \geq p_{*}^{10 \varepsilon n} \, \PP\big(\Omega_{\textup{c}} \big) \PP_{-\sigma R} \left( \Omega_{\textup{exc}}^{-\sigma} \right)\PP_{\sigma R} \left( \Omega_{\textup{mea}}^{\sigma} \right).
\end{align*}
Taking the logarithm, dividing by $n$, and then taking $\liminf_{n \to \infty}$, we obtain for each $\sigma \in \{-,+\}$,
\begin{align}
& \liminf \limits_{n \to \infty} \frac{1}{n} \log \PP\big( \ell_n \in B(\mu,22\varepsilon) \big) \label{eq_proof_lower_bound_liminf_1_n_log_three_blocks} \\
& \qquad \qquad \geq 10 \varepsilon \log(p_{*}) + \liminf \limits_{n \to \infty} \frac{1}{n} \log \PP\big(\Omega_{\textup{c}} \big) + \liminf \limits_{n \to \infty} \frac{1}{n} \log \PP_{-\sigma R} \left( \Omega_{\textup{exc}}^{-\sigma} \right)  + \liminf \limits_{n \to \infty} \frac{1}{n} \log\PP_{\sigma R} \left( \Omega_{\textup{mea}}^{\sigma} \right). \notag
\end{align}

Assume first that $\alpha_{\sigma}>0$ for every $\sigma \in \{-,0,+\}$ and choose $\varepsilon< \frac{1}{3} \min \big\{ \alpha_{\sigma} \, : \, \sigma \in \{-,0,+\} \big\}$. Then, $t^{\sigma}=2 \left \lceil \frac{(\alpha_{\sigma}-3\varepsilon)n}{2} \right \rceil +1$ for every $\sigma \in \{-,0,+\}$. Since $R>R_{\overline{\ZZ}}(\varepsilon)$, Lemma~\ref{lemma_exponential_decay_with_gamma_n_lower}, applied with $\gamma=\alpha_{\pm \sigma}-3\varepsilon$, gives
\begin{align*}
\liminf \limits_{n \to \infty} \frac{1}{n} \log \PP_{-\sigma R} \left( \Omega_{\textup{exc}}^{-\sigma} \right) & \geq -(\alpha_{-\sigma}-3\varepsilon) I_{\textup{Cr}}^{p_{-\sigma}}(0) - (\alpha_{-\sigma}-3\varepsilon)g(\varepsilon), \\
\liminf \limits_{n \to \infty} \frac{1}{n} \log\PP_{\sigma R} \left( \Omega_{\textup{mea}}^{\sigma} \right) & \geq -(\alpha_{\sigma}-3\varepsilon) \inf_{\sigma x \in [0,1]} I_{\textup{Cr}}^{p_{\sigma}}(x) - (\alpha_{\sigma}-3\varepsilon)g(\varepsilon).
\end{align*}
Combining these estimates with
\eqref{eq_proof_lower_bound_liminf_1_n_log_three_blocks}, and using the central-block estimate
\eqref{eq_lemma_I_DV_with_gamma_n_lower}, we obtain, after taking $R$ to $\infty$, 
\begin{align*}
& \liminf \limits_{n \to \infty} \frac{1}{n} \log \PP\big( \ell_n \in B(\mu,22\varepsilon) \big)\\
& \qquad \qquad \geq 10 \varepsilon \log(p_{*}) -(\alpha_0-3\varepsilon)I_{\textup{DV}}(\mu_0) -(\alpha_{-\sigma}-3\varepsilon)I_{\textup{Cr}}^{p_{-\sigma}}(0) - (\alpha_{\sigma}-3\varepsilon) \inf_{\sigma x \in [0,1]}I_{\textup{Cr}}^{p_{\sigma}}(x) - g(\varepsilon).
\end{align*}
Finally, by taking $\varepsilon$ to 0, we deduce that for each $\sigma \in \{-,+\}$,
\begin{align}\label{eq_proof_lemma_lower_bound_almost_end}
& \lim \limits_{\varepsilon \to 0^+}\liminf \limits_{n \to \infty} \frac{1}{n} \log \PP\big( \ell_n \in B(\mu,\varepsilon) \big) \geq -\alpha_0 I_{\textup{DV}}(\mu_0) -\alpha_{-\sigma} I_{\textup{Cr}}^{p_{-\sigma}}(0) - \alpha_{\sigma} \inf_{\sigma x \in [0,1]}I_{\textup{Cr}}^{p_{\sigma}}(x),
\end{align}
where the fact that $\lim_{\varepsilon \to 0^+} g(\varepsilon)=0$ follows from Proposition~\ref{proposition_exponential_rate_of_decay_exc_and_end}. 

Let us now show that \eqref{eq_proof_lemma_lower_bound_almost_end} remains valid when at least one of the coefficients
$\alpha_{-},\alpha_0,\alpha_{+}$ is zero. If $\alpha_0=0$, then $t^0=1$ and $\Omega_{\textup{c}}=\{(0)\}$, so
\begin{align*}
\liminf_{n\to\infty}\frac1n\log \PP(\Omega_{\textup{c}})=0.
\end{align*}
If $\alpha_{-\sigma}=0$, then $t^{-\sigma}=1$ and
$\Omega_{\textup{exc}}^{-\sigma}=\{(-\sigma R)\}$. Thus,
\begin{align*}
\liminf_{n\to\infty}\frac1n\log
\PP_{-\sigma R}\left(\Omega_{\textup{exc}}^{-\sigma}\right)=0.
\end{align*}
The same argument applies to the meander term when $\alpha_{\sigma}=0$. Thus, in \eqref{eq_proof_lower_bound_liminf_1_n_log_three_blocks}, the terms corresponding to strictly positive coefficients are handled as before, whereas those associated with vanishing coefficients are equal to zero. It follows, with the convention $0\cdot\infty=0$, that \eqref{eq_proof_lemma_lower_bound_almost_end} remains valid even when at least one of the coefficients $\alpha_-$, $\alpha_0$, $\alpha_+$ is zero.

Taking the maximum over $\sigma \in \{-,+\}$ in \eqref{eq_proof_lemma_lower_bound_almost_end} and invoking the definition of $I$ from \eqref{eq_def_rate_fuction} yields
\begin{align*}
\lim \limits_{\varepsilon \to 0^+}\liminf \limits_{n \to \infty} \frac{1}{n} \log \PP\big( \ell_n \in B(\mu,\varepsilon) \big) \geq - I(\mu).
\end{align*}
Since this holds for all $\mu \in \mathcal{P}(\overline{\ZZ})$, this proves Proposition~\ref{prop_lower_bound}.
\end{proof}

\section{The upper bound}\label{section_upper_bound}

The aim of this section is to establish inequality (\ref{eq_prop_upper_bound_RL}) of Proposition~\ref{proposition_RL_functions}, which we restate below.

\begin{proposition}\label{prop_upper_bound}
For all $\mu \in \mathcal{P}(\overline{\ZZ})$, 
\begin{equation*}\label{eq_prop_upper_bound}
\lim \limits_{\varepsilon \to 0^+} \limsup \limits_{n \to \infty} \frac{1}{n} \log \PP \left( \ell_n \in B(\mu,\varepsilon) \right) \leq -I(\mu) ,
\end{equation*}
where the rate function $I$ is defined in (\ref{eq_def_rate_fuction}).
\end{proposition}

\subsection{Decomposition and stitching of trajectories}\label{subsection_decomp_and_stitch_traj_upper_bound}

In the rest of this section, we fix in the following order
\begin{enumerate}
\item a probability measure $\mu=\sum_{\sigma \in \{-,0,+\}} \alpha_{\sigma} \mu_{\sigma} \in \mathcal{P}(\overline{\ZZ})$ with $\mu_0 \in \mathcal{P}(\ZZ)$,
\item $\varepsilon>0$ sufficiently small so that $\varepsilon < \frac{1}{2} \min_{\sigma \in \{-,0,+\}} \left\{ \alpha_{\sigma} \, : \, \alpha_{\sigma}>0 \right\}$ and $\varepsilon< \min \left\{\frac{1}{40},\varepsilon^* \right\}$, where $\varepsilon^*$ is defined after Equation~(\ref{eq_for_lemma_G_varepsilon_to_zero}),
\item $R > \max\{R_{\mu_0}(\varepsilon),R_{\overline{\ZZ}}(\varepsilon)\}$ and $n > \frac{3}{\varepsilon}$.
\end{enumerate}

Fix a trajectory $w \in \Omega_n^{(0)}$. The idea is to decompose $w$ into successive pieces according to the region visited, and then regroup together the pieces corresponding to the same region. In this way, the original trajectory is reorganised into three trajectories, each confined to a single region. In the following definitions, we will not write the dependency on $w$ for the sake of readability, except when it is necessary. 

We define by induction the sequence $(t_i,\sigma_i)_{i \in \NN} \in \big((\NN \cup \{+\infty\}) \times \{-,0,+\}\big)^{\NN}$ by first setting $t_1 = 1$ and $\sigma_1=0$. Suppose the sequence is constructed up to the $i$-th step and that $w_{t_i} \in A^{\sigma_i}$. Then, we set
\begin{equation*}\label{eq_def_t_and_sigma}
t_{i+1} \eqdef \inf \big\{ k \in \dbl t_{i}+1,n \dbr  \, : \, w_k \notin A^{\sigma_i} \big\},
\end{equation*}
with the convention that the $\inf$ over an empty set is $+\infty$. If $t_{i+1} \leq n$, we set $\sigma_{i+1}$ to be the unique element in $\{-,0,+\}$ such that $w_{t_{i+1}} \in A^{\sigma_{i+1}}$, otherwise $t_{i+1}=+\infty$ and we assign an arbitrary value to $\sigma_{i+1}$. Thus, the sequence $(t_i)_{i \in \NN}$ records the times at which the trajectory moves from one region to another, while $(\sigma_i)_{i \in \NN}$ records the successive regions visited. Note that the sequence $(t_i)_{i \in \NN}$ is strictly increasing up to the index at which it first attains the value $+\infty$. We then let
\begin{equation}\label{eq_def_L}
L \eqdef \# \left\{ i \in \NN \, : \, t_i \leq n \right\}.
\end{equation}
For all $\sigma \in \{-,0,+\}$, we also let
\begin{equation}\label{eq_def_J_sigma}
J^{\sigma} \eqdef \{ i \in \dbl 1,L \dbr \, : \, \sigma_i=\sigma \}.
\end{equation}
Hence, $L$ is the number of successive visits of the trajectory to the regions $A^-$, $A^0$ and $A^+$, and for each $\sigma \in \{-,0,+\}$, the set $J^{\sigma}$ records which of these visits take place in the region $A^{\sigma}$. We enumerate $J^{\sigma}$ in the following way. There exists a sequence $1 \leq i^{\sigma}_1<i^{\sigma}_2< \cdots < i^{\sigma}_{\# J^{\sigma}}\leq L$ such that
\begin{align*}
J^{\sigma}=\big\{i_1^{\sigma} \, , \, i_2^{\sigma} \, , \, \dots \, , \,  i^{\sigma}_{\#J^{\sigma}}\big\}.
\end{align*}
For all $\sigma \in \{-,0,+\}$ and all $h \in \dbl 1,\#J^{\sigma} \dbr$, we define the subwords
\begin{equation}\label{eq_def_u_h_sigma}
u^{h,\sigma} \eqdef w_{[t_{i_h^{\sigma}} \, : \, (t_{i_h^{\sigma}+1}-1) \wedge n]}.
\end{equation}
These definitions are illustrated in Figure~\ref{figure_example_random_walk_for_def}. 
\begin{figure}[ht]
    \centering
    \includegraphics[width=\textwidth]{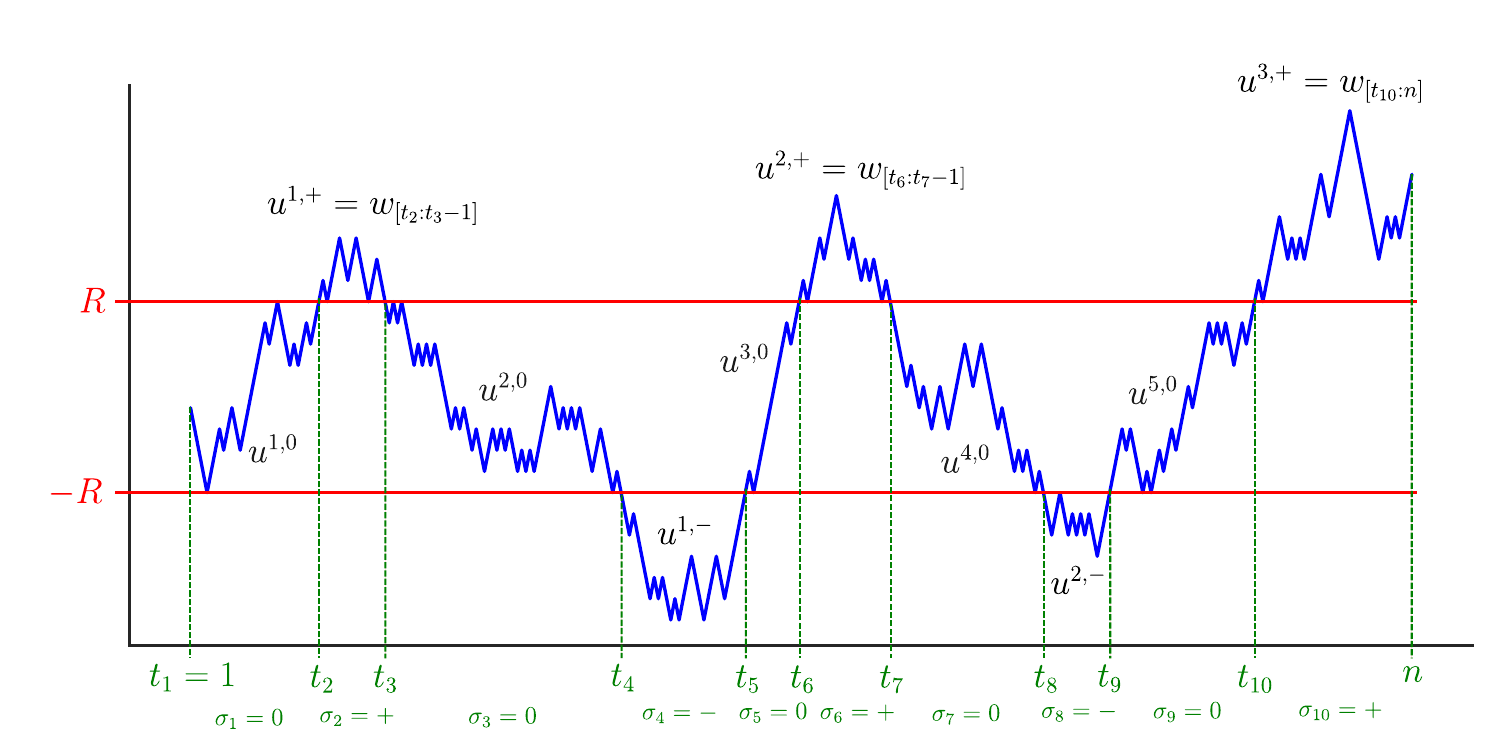}
    \caption{Example of a random walk with the associated sequence $(t_i,\sigma_i)_{i=1}^{10}$ and subwords $u^{h,\sigma}$.}
    \label{figure_example_random_walk_for_def}
\end{figure}
In this example, we have $L=10$, $J^{0}=\{1,3,5,7,9\} $, $J^{-}=\{4,8\}$ and $J^{+}=\{2,6,10\}$. 

The next lemma bounds the number of subwords associated with each $\sigma \in \{-,0,+\}$. We recall that for all $\sigma \in \{-,0,+\}$ and all $k \in \ZZ$, we introduced in \eqref{eq_def_C_sigma_and_C_k} the occupation times
\begin{equation*}
N^{\sigma}(w) \eqdef \sum_{j=1}^n \mathbbm{1}_{\{w_j \in A^{\sigma}\}} \qquad \text{and} \qquad N_k(w) \eqdef \sum_{j=1}^n \mathbbm{1}_{\{w_j =k\}}.
\end{equation*}
For each $\sigma \in \{-,0,+\}$ and each $j \in \dbl 1,n \dbr$ such that $w_j \in A^{\sigma}$, there exists by definition of the stopping times $(t_i)_{i \in \NN}$ and of the index set $J^{\sigma}$, a unique $h_j \in \dbl 1,\#J^{\sigma} \dbr$ such that $t_{i_{h_j}^{\sigma}} \leq j < t_{i_{h_j}^{\sigma}+1}$. Conversely, if $t_{i^{\sigma}_h} \leq j < t_{i^{\sigma}_h+1}$ for some $h \in \dbl 1,\#J^{\sigma} \dbr$, then $w_j \in A^{\sigma}$. Hence, for each $\sigma$, the map
\begin{equation}\label{eq_bijection_index_set}
\begin{array}{cccc}
\mathcal{I}^{\sigma} : & \{j \in \dbl 1,n \dbr \, : \, w_j \in A^{\sigma}\} & \longrightarrow & \{(h,l) \in \NN^2 \, : \, h \in \dbl 1,\#J^{\sigma} \dbr \, , \, l \in \dbl 1,|u^{h,\sigma}| \dbr \} \\
& j & \longmapsto & (h_j,j-t_{i^{\sigma}_{h_j}}+1)
\end{array},
\end{equation}
is a bijection. In particular,
\begin{equation}\label{eq_N_sigma_with_subwords}
N^{\sigma}(w)=\sum_{h=1}^{\#J^{\sigma}}|u^{h,\sigma}|.
\end{equation}

\begin{lemma}[Bound on the number of subwords]\label{lemma_bound_on_J_sigma}
Suppose $w \in \Omega_n^{(0)}(\mu,2^{-R}\varepsilon)$. Then, for every $\sigma \in \{-,0,+\}$,
\begin{equation*}\label{eq_lemma_bound_on_J_sigma}
\#J^{\sigma} < 6 \varepsilon n+1.
\end{equation*}
\end{lemma}

\begin{proof}
Consider $\sigma \in \{-,+\}$ and $h \in \dbl 1, \#J^{\sigma} \dbr$. Since we are considering random walks with steps of size 1, a trajectory must pass through $\sigma R$ in order to enter $A^{\sigma}$. Consequently, the first letter of the subword $u^{h,\sigma}$ is $\sigma R$, that is, $u^{h,\sigma}_1=\sigma R$. It follows that each visit to $A^{\sigma}$ contributes at least one visit to $\sigma R$, and therefore $\#J^{\sigma} \leq N_{\sigma R}(w)$. Since $R>R_{\mu_0}(\varepsilon)$ and $w \in \Omega_n^{(0)}(\mu,2^{-R} \varepsilon)$, we may apply inequality (\ref{eq_lemma_bound_on_last_point_if_close_to_mu_c_appendix_2}) of Lemma~\ref{lemma_bound_on_last_point_if_close_to_mu_c_appendix} to bound $N_{\sigma R}(w)$ from above. We thereby obtain
\begin{align*}
\# J^{\sigma} \leq N_{\sigma R}(w) < 3 \varepsilon n.
\end{align*}
If $\sigma=0$ and $h \in \dbl 2,\#J^{0} \dbr$, we similarly obtain that $u^{h,0}_1 \in \{-R+1,R-1\}$. By using again (\ref{eq_lemma_bound_on_last_point_if_close_to_mu_c_appendix_2}) of Lemma~\ref{lemma_bound_on_last_point_if_close_to_mu_c_appendix} to bound $N_{R-1}(w)$ and $N_{-R+1}(w)$ from above, we obtain
\begin{align*}
\# J^{0}-1 \leq N_{R-1}(w) + N_{-R+1}(w) < 6 \varepsilon n.
\end{align*}
We thereby obtain the desired inequality for each $\sigma \in \{-,0,+\}$.
\end{proof}

\begin{lemma}[Existence of a connecting letter]\label{lemma_existence_connecting_one_letter_word}
Consider $w \in \Omega_n^{(0)}$ and $\sigma \in \{-,0,+\}$. For all $h \in \dbl 1,\#J^{\sigma}-1 \dbr$, there exists a one-letter word $\xi^{h,\sigma}=(k)$, with $R \leq |k| \leq R+1$ such that
\begin{equation*}\label{eq_lemma_existence_connecting_one_letter_word}
u^{h,\sigma}\cdot\xi^{h,\sigma} \cdot u^{h+1,\sigma} \in \Omega_{\textup{fin}}.
\end{equation*}
\end{lemma}

\begin{proof}
If $\sigma \in \{-,+\}$, set $\xi^{h,\sigma}=(\sigma (R+1)) \in \Omega_1$. This is a one-letter word with the required property. If $\sigma=0$, there exists $\epsilon_h \in \{-,+\}$ such that $u_{-1}^{h,0}=\epsilon_h (R-1)=u^{h+1,0}_1$. In this case, define $\xi^{h,0}=(\epsilon_h R) \in \Omega_1$, which again satisfies the required property.
\end{proof}

Our aim is to define maps that, for each $\sigma \in \{-,0,+\}$,  concatenate the subwords $u^{h,\sigma}$ into a single word. We set the length of the concatenated word to be
\begin{align*}
\overline{\alpha}_{\sigma,n} \eqdef 2 \left \lceil \frac{\alpha_{\sigma}+8 \varepsilon}{2} n \right \rceil +1.
\end{align*}

\begin{lemma}\label{lemma_overline_alpha_sigma_long_enough}
Fix $\sigma \in \{-,0,+\}$ and $w \in \Omega_n^{(0)}(\mu,2^{-R}\varepsilon)$. If $\#J^{\sigma} \neq 0$, then
\begin{equation}\label{eq_lemma_overline_alpha_sigma_long_enough}
\sum_{h=1}^{\#J^{\sigma}}|u^{h,\sigma}|+\#J^{\sigma}-1 < \overline{\alpha}_{\sigma,n}.
\end{equation}
\end{lemma}

\begin{proof}
As mentioned in \eqref{eq_N_sigma_with_subwords}, the total time spent in the region $A^{\sigma}$ is given by
\begin{align*}
N^{\sigma}(w)=\sum_{h=1}^{\#J^{\sigma}}|u^{h,\sigma}|.
\end{align*}
Since $R>R_{\mu_0}(\varepsilon)$ and $\ell(w) \in B(\mu,2^{-R}\varepsilon)$, we obtain for each $\sigma \in \{-,0,+\}$, by (\ref{eq_lemma_bound_on_last_point_if_close_to_mu_c_appendix_1}) of Lemma~\ref{lemma_bound_on_last_point_if_close_to_mu_c_appendix}, that $|N^{\sigma}(w)-\alpha_{\sigma} n| < 2 \varepsilon n$. In addition, using Lemma~\ref{lemma_bound_on_J_sigma} to bound $\#J^{\sigma}-1$ from above, we obtain
\begin{align*}
\sum_{h=1}^{\#J^{\sigma}}|u^{h,\sigma}|+\#J^{\sigma} -1 & = N^{\sigma}(w)+\#J^{\sigma}-1 \\
& < \alpha_{\sigma}n+8\varepsilon n.
\end{align*}
The upper bound given in Equation~(\ref{eq_lemma_overline_alpha_sigma_long_enough}) then follows by definition of $\overline{\alpha}_{\sigma,n}$.
\end{proof}

\begin{definition}\label{def_b_sigma}
For $\sigma \in \{-,0,+\}$ and $w \in \Omega_n^{(0)}(\mu,2^{-R}\varepsilon)$, we define the word $b^{\sigma} \in \Omega_{\textup{fin}}$ as follows. 
\begin{itemize}
\item[-] If $\#J^{\sigma}=0$, which can only happen if $\sigma \neq 0$, then we set for $j \in \dbl 1,\overline{\alpha}_{\sigma,n} \dbr$,
\begin{align*}
b^{\sigma}_j = \sigma(R+\mathbbm{1}_{j \in 2 \NN}).
\end{align*}
\item[-] If $\#J^{\sigma}>0$, let $m \eqdef \overline{\alpha}_{\sigma,n} - \left(\sum_{h=1}^{\#J^{\sigma}}|u^{h,\sigma}|+\#J^{\sigma}-1\right)$, which is positive by Lemma~\ref{lemma_overline_alpha_sigma_long_enough}. Then, for all $j \in \dbl 1, m \dbr$, set
\begin{align*}
b^{\sigma}_j = u^{\#J^{\sigma},\sigma}_{-1}+ \delta \mathbbm{1}_{j \in 2\NN-1},
\end{align*}
where $\delta \in \{-1,1\}$ is chosen so that $b^{\sigma}_j \in A^{\sigma}$ for all $j \in \dbl 1,m \dbr$.
\end{itemize}
\end{definition}

\noindent The existence of a $\delta \in \{-1,1\}$ which satisfies the condition of the second item of Definition~\ref{def_b_sigma} follows from the fact that for each $\sigma \in \{-,0,+\}$ and $k \in A^{\sigma}$, we have $k-1 \in A^{\sigma}$ or $k+1 \in A^{\sigma}$.

\begin{lemma}\label{lemma_bound_b_sigma}
For $\sigma \in \{-,0,+\}$ and $w \in \Omega_n^{(0)}(\mu,2^{-R}\varepsilon)$, we have
\begin{equation}\label{eq_lemma_bound_b_sigma}
|b^{\sigma}| < 11 \varepsilon n.
\end{equation}
\end{lemma}
\begin{proof}
If $\#J^{\sigma}=0$, then by definition of $b^{\sigma}$, we have
\begin{align*}
|b^{\sigma}|&=\overline{\alpha}_{\sigma,n} \leq (\alpha_{\sigma}+8\varepsilon)n+3.
\end{align*}
Since $\#J^{\sigma}=0$, we have $N^{\sigma}(w)=0$. Moreover, because $R>R_{\mu_0}(\varepsilon)$ and $w \in \Omega_n^{(0)}(\mu,2^{-R}\varepsilon)$, Equation~(\ref{eq_lemma_bound_on_last_point_if_close_to_mu_c_appendix_1}) implies $\alpha_{\sigma}<2 \varepsilon$. Together with the assumption $n > \frac{3}{\varepsilon}$, this gives
\begin{align*}
|b^{\sigma}| & < 10 \varepsilon n+3 < 11 \varepsilon n.
\end{align*}
If $\#J^{\sigma}>0$, then by definition of $b^{\sigma}$,
\begin{align*}
|b^{\sigma}| & \leq \overline{\alpha}_{\sigma,n}-\sum_{h=1}^{\#J^{\sigma}}|u^{h,\sigma}| \leq (\alpha_{\sigma}+8\varepsilon) n +3-N^{\sigma}(w).
\end{align*}
Using again Equation~(\ref{eq_lemma_bound_on_last_point_if_close_to_mu_c_appendix_1}), we deduce 
\begin{align*}
|b^{\sigma}| < (\alpha_{\sigma}+8\varepsilon) n +3 - \alpha_{\sigma}n+2 \varepsilon n.
\end{align*}
Since $n>\frac{3}{\varepsilon}$, we conclude that $|b^{\sigma}| < 11 \varepsilon n$, which establishes \eqref{eq_lemma_bound_b_sigma}.
\end{proof}

\begin{definition}\label{def_phi}
For each $\sigma \in \{-,0,+\}$, we define the maps $\phi^{\sigma}$ by setting,
\begin{equation*}\label{eq_def_phi}
\begin{array}{cccc}
\phi^{\sigma}:&\Omega_n^{(0)}(\mu,2^{-R}\varepsilon) & \longrightarrow & \Omega_{\overline{\alpha}_{\sigma,n}} \\
& w & \longmapsto &
\begin{cases}
u^{1,\sigma} \cdot \xi^{1,\sigma} \cdot u^{2,\sigma} \cdot \xi^{2,\sigma} \, \cdots \, u^{\#J^{\sigma},\sigma} \cdot b^{\sigma}, & \text{if } \#J^{\sigma}>0, \\
b^{\sigma}, & \text{if } \#J^{\sigma}=0,
\end{cases}
\end{array}
\end{equation*}
where for all $h \in \dbl 1,\#J^{\sigma}-1 \dbr$, $\xi^{h,\sigma}$ is the one-letter word connecting $u^{h,\sigma}_{-1}$ to $u^{h+1,\sigma}_1$, which exists by Lemma~\ref{lemma_existence_connecting_one_letter_word} and $b^{\sigma}$ is the word introduced in Definition~\ref{def_b_sigma}.
\end{definition}

\begin{remark}\label{remark_phi_sigma_well_defined}
The choice of $b^{\sigma}$ in Definition \ref{def_b_sigma} ensures that, for all $w \in \Omega_n^{(0)}(\mu,2^{-R}\varepsilon)$, one has $|\phi^{\sigma}(w)|=\overline{\alpha}_{\sigma,n}$. This is clear when $\#J^{\sigma}=0$. Otherwise, since $\xi^{h,\sigma}$ is a one-letter word, we have
\begin{align*}
|\phi^{\sigma}(w)|=\sum_{h=1}^{\#J^{\sigma}}|u^{h,\sigma}|+\sum_{h=1}^{\#J^{\sigma}-1}|\xi^{h,\sigma}|+|b^{\sigma}| & = \sum_{h=1}^{\#J^{\sigma}}|u^{h,\sigma}|+\#J^{\sigma}-1+|b^{\sigma}|.
\end{align*}
The result then follows from the choice of the length of $b^{\sigma}$. 
\end{remark}

\noindent The role of the $\xi^{h,\sigma}$ is to connect the $u^{h,\sigma}$ together, while $b^{\sigma}$ ensures that concatenated word $\phi^{\sigma}(w)$ has length $\overline{\alpha}_{\sigma,n}$. These additional words play no significant role at the level of the empirical measure. Indeed, Lemmas~\ref{lemma_bound_on_J_sigma} and~\ref{lemma_bound_b_sigma} show that both $\#J^{\sigma}-1=\sum_{h=1}^{\#J^{\sigma}-1}|\xi^{h,\sigma}|$ and $|b^{\sigma}|$ are negligible compared to $n$. The maps $\phi^{-}$, $\phi^{0}$ and $\phi^{+}$ are depicted in Figure~\ref{figure_random_walk_decomposed}. 

\begin{figure}[ht]
    \centering
    \includegraphics[width=\textwidth]{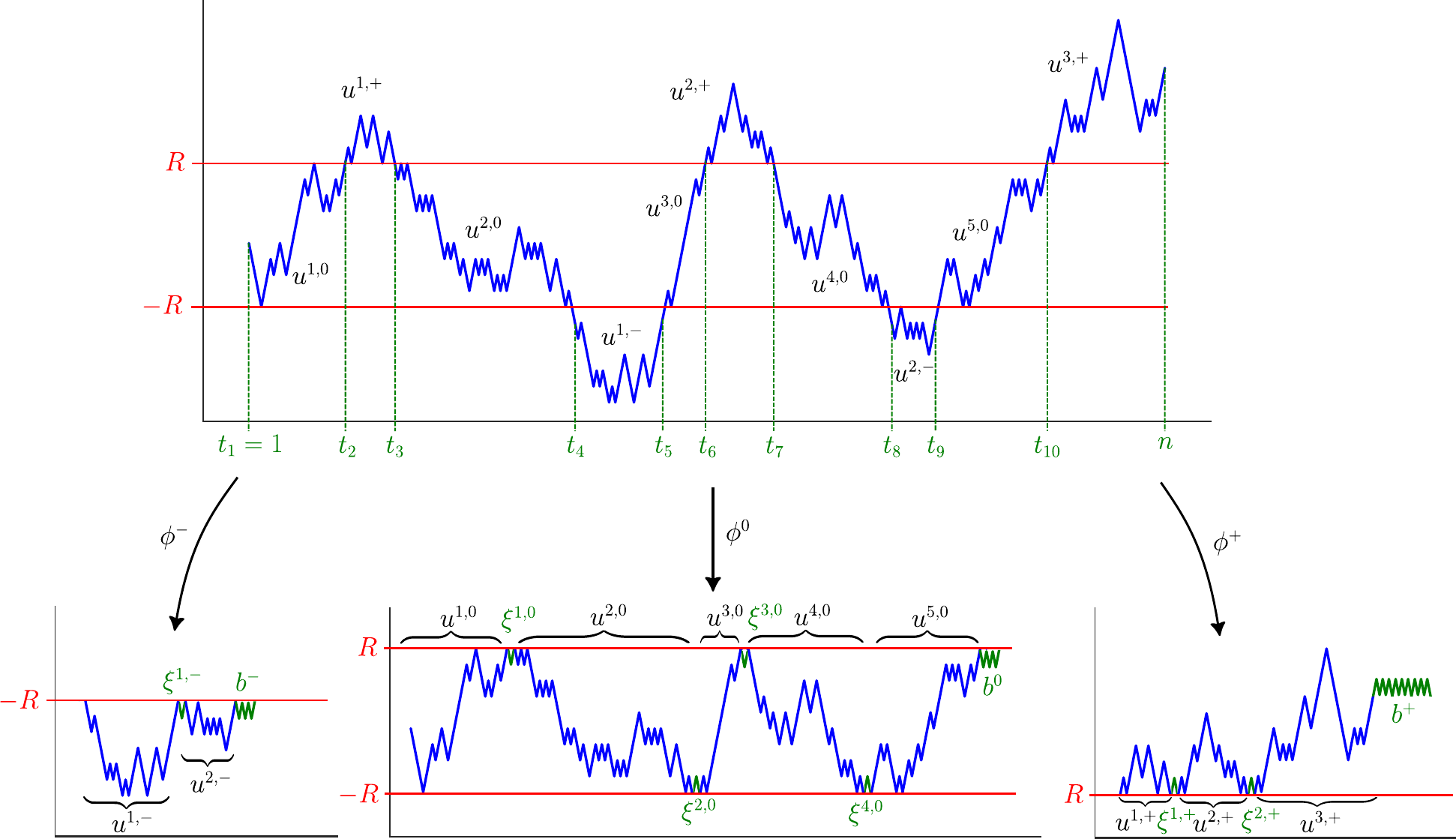}
    \caption{Decomposition of a trajectory $w \in \mathscr{C}_n^{+}$ with the maps $\phi^{\sigma}$.}
    \label{figure_random_walk_decomposed}
\end{figure}

\subsection{Properties of the stitched trajectories}

\noindent The maps $\phi^{\sigma}$ separate a trajectory $w$ into three smaller trajectories $\phi^{-}(w)$, $\phi^{0}(w)$ and $\phi^{+}(w)$. Let us present properties that the trajectories $\phi^{\sigma}(w)$ inherit from $w$. Until now we have indexed the regions by $\sigma \in \{-,0,+\}$. In what follows, it will be convenient to introduce a second index $\rho \in \{-,0,+\}$, where $\rho$ denotes the class of the trajectory, while $\sigma$ continues to label the regions themselves. The following lemma, characterises for $w \in \mathscr{C}_n^{\rho}$ each of the three trajectories $\phi^{-}(w)$, $\phi^{0}(w)$ and $\phi^{+}(w)$, depending on the value of $\rho \in \{-,0,+\}$.

\begin{lemma}[Properties of stitched trajectories]\label{lemma_image_og_phi_sigma_right_set}
Consider $\rho \in \{-,0,+\}$ and $w \in \mathscr{C}_{n}^{\rho}(\mu,2^{-R}\varepsilon)$.
\begin{enumerate}
\item If $\alpha_0 > 0$, then there exists a constant $C_1>0$, which only depends on $\alpha_0$, such that
\begin{equation*}\label{eq_lemma_image_og_phi_sigma_right_set_phi_0}
\phi^0(w) \in \Omega^{(0)}_{\overline{\alpha}_{0,n}}(\mu_0,C_1 \varepsilon).
\end{equation*}

\item If $\sigma \in \{-,+\}$ and $\sigma =\rho$, 
\begin{equation*}\label{eq_lemma_image_og_phi_sigma_right_set_widehat_sigma_neq_0}
\phi^{\sigma}(w) \in 
\Omega_{\overline{\alpha}_{\sigma,n},\textup{mea}}^{\sigma}.
\end{equation*}

\item If $\sigma \in \{-,+\}$ and $\sigma \neq \rho$, 
\begin{equation*}\label{eq_lemma_image_og_phi_sigma_right_set_widehat_sigma_eq_0}
\phi^{\sigma}(w) \in \Omega_{\overline{\alpha}_{\sigma,n},\textup{exc}}^{\sigma}.
\end{equation*}

\end{enumerate}
\end{lemma}

\begin{proof}
Since the first claim does not depend on $\rho \in \{-,0,+\}$, fix $w \in \Omega_{n}^{(0)}(\mu,2^{-R}\varepsilon)$. Moreover, when $\sigma=0$, one necessarily has $\#J^{\sigma}>0$, so it suffices to treat this case. Consider the restricted empirical measure
\begin{equation*}
\ell^{0}(w) \eqdef \frac{1}{N^{0}(w)} \sum_{\substack{ 1 \leq j \leq n \, : \\ w_j \in A^0}} \delta_{w_j} \in \mathcal{P}(\overline{\ZZ}),
\end{equation*}
defined in \eqref{eq_def_restricted_empirical_measure}. We recall that $N^0(w)$, defined in (\ref{eq_def_C_sigma_and_C_k}), counts the number of letters in $w$ that belong to $A^0$. By inequality (\ref{eq_lemma_bound_on_last_point_if_close_to_mu_c_appendix_1}) of Lemma~\ref{lemma_bound_on_last_point_if_close_to_mu_c_appendix} and the assumption $\varepsilon<\frac{\alpha_0}{2}$, we have $N^{0}(w) > 0$, so $\ell^{0}(w)$ is well defined. With the bijection $\mathcal{I}^{0}$ given in \eqref{eq_bijection_index_set}, we can rewrite
\begin{align*}
\ell^{0}(w) = \frac{1}{N^{0}(w)} \sum_{h=1}^{\#J^{0}} \sum_{l=1}^{|u^{h,0}|} \delta_{u^{h,0}_l}.
\end{align*}

\noindent Decomposing $\phi^{0}(w)$ into its constituent subwords and using \eqref{eq_lemma_properties_KR_norm_of_proba}, we obtain 
\begin{align*}
\left\| \ell\big(\phi^{0}(w)\big) - \mu_{0} \right\| 
& = \left\| \frac{1}{|\phi^{0}(w)|} \left( \sum_{h=1}^{\#J^{0}} \sum_{j=1}^{|u^{h,0}|} \delta_{u^{h,0}_j} + \sum_{h=1}^{\#J^{0}-1} \delta_{\xi^{h,0}_1} + \sum_{j=1}^{|b^{0}|} \delta_{b^{0}_j} \right) - \mu_{0} \right\| \\
& \leq \left| \frac{N^{0}(w)}{|\phi^{0}(w)|}-1 \right| + \left\| \frac{1}{N^{0}(w)} \sum_{h=1}^{\#J^{0}} \sum_{j=1}^{|u^{h,0}|} \delta_{u^{h,0}_j} -\mu_{0} \right\| + \frac{\sum_{h=1}^{\#J^{0}-1}|\xi^{h,0}| + |b^{0}|}{|\phi^{0}(w)|}.
\end{align*}
Note that 
\begin{align*}
|\phi^{0}(w)|-N^{0}(w) & = \sum_{h=1}^{\#J^{0}-1}|\xi^{h,0}| + |b^{0}|.
\end{align*}
Therefore,
\begin{align*}
\left\| \ell\big(\phi^{0}(w)\big) - \mu_{0} \right\| 
\leq 2 \frac{\sum_{h=1}^{\#J^{0}-1}|\xi^{h,0}| + |b^{0}|}{|\phi^{0}(w)|} + \left\| \ell^{0}(w)-\mu_{0} \right\|.
\end{align*}
We can now use the quantitative bounds proved previously. Lemma~\ref{lemma_bound_on_J_sigma} gives a bound on $\#J^0-1=\sum_{h=1}^{\#J^0-1}|\xi^{h,0}|$, Equation~(\ref{eq_lemma_bound_b_sigma}) bounds $|b^0|$, and by construction
\begin{align*}
|\phi^{0}(w)|=\overline{\alpha}_{0,n} \geq \alpha_{0} n.
\end{align*}
Combining these bounds yields
\begin{align*}
\left\| \ell\big(\phi^{0}(w)\big) - \mu_{0} \right\| \leq \frac{34}{\alpha_{0}}\varepsilon + \left\| \ell^{0}(w)-\mu_{0} \right\|.
\end{align*}
Finally, apply Lemma~\ref{lemma_restriction_of_measure_good_approx} to bound $\left\| \ell^{0}(w)-\mu_{0} \right\|$ by $\frac{6 \varepsilon}{\alpha_0}$. Altogether, we obtain
\begin{align*}
\left\| \ell\big(\phi^{0}(w)\big) - \mu_{0} \right\| \leq \frac{40}{\alpha_{0}} \varepsilon,
\end{align*}
and the first item of the lemma follows with $C_1=\frac{40}{\alpha_{0}}$.

We now turn to the second item, and fix $\sigma=\rho \in \{-,+\}$. To prove that $\phi^{\sigma}(w)$ is a meander, it suffices to show that $\phi^{\sigma}(w)_1=\sigma R$ and $\phi^{\sigma}(w)_j \in A^{\sigma}$ for every $j \in \dbl 1, |\phi^{\sigma}(w)| \dbr$. If $\#J^{\sigma}=0$, then by definitions of $b^{\sigma}$ and $\phi^{\sigma}$, we have for $j \in \dbl 1, \overline{\alpha}_{\sigma,n} \dbr$,
\begin{align*}
\phi^{\sigma}(w)_j&=b^{\sigma}_j=\sigma(R+\mathbbm{1}_{j \in 2 \NN}).
\end{align*}
In particular, $\phi^{\sigma}(w)_1=\sigma R$ and every letter belongs to $A^{\sigma}$. Suppose $\#J^{\sigma}>0$. Then,
\begin{align*}
\phi^{\sigma}(w)_1=u^{1,\sigma}_1=w_{t_{i^{\sigma}_1}}.
\end{align*}
By definition of the stopping times $t_{i^{\sigma}_1}$, we have $w_{t_{i^{\sigma}_1}-1} \notin A^{\sigma}$ and $w_{t_{i^{\sigma}_1}} \in A^{\sigma}$. Hence, $\phi^{\sigma}(w)_1=\sigma R$. Furthermore, for every $h \in \dbl 1,\#J^{\sigma}\dbr$ and $j \in \dbl 1,|u^{h,\sigma}| \dbr$, 
\begin{align*}
u^{h,\sigma}_j \in A^{\sigma} \, , \, \forall j \in \dbl 1, |u^{h,\sigma}| \dbr \qquad \text{and} \qquad \xi^{h,\sigma}_1=\sigma(R+1) \in A^{\sigma}.
\end{align*}
In addition, by choice of $\delta$ in the construction of $b^{\sigma}$ in Definition~\ref{def_b_sigma}, we again have $b^{\sigma}_j \in A^{\sigma}$ for all $j \in \dbl 1,|b^{\sigma}| \dbr$. We deduce that for all $j \in \dbl 1,|\phi^{\sigma}(w)| \dbr$,
\begin{align*}
\phi^{\sigma}(w)_j \in A^{\sigma}.
\end{align*}
Thus, the two conditions defining $\Omega_{\overline{\alpha}_{\sigma,n},\textup{mea}}^{\sigma}$ are satisfied and so $\phi^{\sigma}(w) \in \Omega_{\overline{\alpha}_{\sigma,n},\textup{mea}}^{\sigma}$.

Finally, we prove Item~\textit{3}. Fix $\sigma \in \{-,+\}$ with $\sigma \neq \rho$. From the argument in the proof of Item~\textit{2.}, which does not rely on $\sigma=\rho$, we already know that $\phi^{\sigma}(w)_1=\sigma R$ and $\phi^{\sigma}(w)_j \in A^{\sigma}$ for every $j$. To prove that $\phi^{\sigma}(w)$ is an excursion, it remains to show that the last letter satisfies $\phi^{\sigma}(w)_{-1}=\sigma R$. 

If $\#J^{\sigma}=0$, then $\overline{\alpha}_{\sigma,n}$ is odd by construction, and hence
\begin{align*}
\phi^{\sigma}(w)_{-1}=\sigma(R+\mathbbm{1}_{\overline{\alpha}_{\sigma,n} \in 2\NN})=\sigma R,
\end{align*}
so the claim holds. 

If $\#J^{\sigma}>0$, by definitions of $b^{\sigma}$ and $\phi^{\sigma}$, we have
\begin{equation}\label{eq_lemma_last_letter_phi_sigma}
\phi^{\sigma}(w)_{-1}=b^{\sigma}_{-1}=u^{\#J^{\sigma},\sigma}_{-1}+\delta \mathbbm{1}_{|b^{\sigma}| \in 2\NN-1}.
\end{equation}
For each $h \leq \#J^{\sigma}$, the stopping times satisfy
\begin{align*}
w_{t_{i^{\sigma}_h}-1} \notin A^{\sigma} \qquad \text{and} \qquad u^{h,\sigma}_1=w_{t_{i^{\sigma}_h}} \in A^{\sigma}.
\end{align*}
Since $\sigma \neq 0$, we have $t_{i^{\sigma}_h}-1 \geq 1$, so these statements are also well defined for $h=1$. Therefore, for all $h \in \dbl 1,\#J^{\sigma} \dbr$, we have $u^{h,\sigma}_1=\sigma R$. 

Define $t_{*,h} \eqdef (t_{i^{\sigma}_{h}+1}-1) \wedge n$. If $t_{*,h}=n$ for some $h \in \dbl 1,\#J^{\sigma} \dbr$, then the trajectory $w$ would end in $A^{\sigma}$, contradicting $w \in \mathscr{C}^{\rho}_n$ with $\rho \neq \sigma$. Hence, $t_{*,h}+1 \leq  n$ for every $h$, which gives
\begin{align*}
u^{h,\sigma}_{-1}=w_{t_{*,h}} \in A^{\sigma} \qquad \text{and} \qquad w_{t_{*,h}+1} \notin A^{\sigma}.
\end{align*}
Thus, for all $h \in \dbl 1,\#J^{\sigma} \dbr$, we have $u^{h,\sigma}_{-1}=\sigma R$. In particular,
\begin{align*}
u^{\#J^{\sigma},\sigma}_{-1}=\sigma R.
\end{align*}
By Equation~(\ref{eq_lemma_last_letter_phi_sigma}), it remains to show that $|b^{\sigma}|$ is even. Since $u^{h,\sigma}_{1}=\sigma R=u^{h,\sigma}_{-1}$, parity of the simple random walk implies $|u^{h,\sigma}| \in 2\NN-1$ for every $h$. Consequently, 
\begin{align*}
\sum_{h=1}^{\#J^{\sigma}}|u^{h,\sigma}|+\#J^{\sigma}-1&=\sum_{h=1}^{\#J^{\sigma}}(|u^{h,\sigma}|+1)-1 \in 2\NN-1.
\end{align*}
Since $\overline{\alpha}_{\sigma,n}$ is odd by construction and 
\begin{align*}
|b^{\sigma}|&=\overline{\alpha}_{\sigma,n}-\left(\sum_{h=1}^{\#J^{\sigma}}|u^{h,\sigma}|+\#J^{\sigma}-1\right),
\end{align*}
we deduce that $|b^{\sigma}|$ is even. Therefore, $\phi^{\sigma}(w)_{-1}=\sigma R$, showing that all three conditions defining $\Omega_{\overline{\alpha}_{\sigma,n},\textup{exc}}^{\sigma}$ are satisfied. Hence, $\phi^{\sigma}(w) \in \Omega_{\overline{\alpha}_{\sigma,n},\textup{exc}}^{\sigma}$.
\end{proof}

The last two assertions of Lemma~\ref{lemma_image_og_phi_sigma_right_set} are illustrated in Figure~\ref{figure_random_walk_decomposed}. If $w \in \mathscr{C}^{+}_{n}$, then $\phi^{+}(w)$ starts at $R$ and remains in $A^+$ throughout its evolution. In particular, it never goes below $R$. Hence, $\phi^{+}(w) \in \Omega_{q,\textup{mea}}^+$ for some $q \in \NN$. Concerning $\phi^{-}(w)$, it starts at $-R$, stays below $-R$ during the whole trajectory and then finishes at $-R$. Thus, $\phi^{-}(w)  \in \Omega_{m,\textup{exc}}^{-}$ for some $m \in \NN$.

Lemma~\ref{lemma_image_og_phi_sigma_right_set} links the images of the maps $\phi^{\sigma}$ to sets for which we have the exponential rate of decay of their probabilities. The next lemma gives a combinatorial bound on the preimages of the maps $\phi^{\sigma}$. 

\begin{lemma}[Cardinality of the preimage of $\phi^{\sigma}$]\label{lemma_upper_bound_combinatorial_on_phi_sigma}
There exists a function $\widetilde{f}:\RR_+^* \to \RR$ such that $\lim_{\varepsilon \to 0^+}\widetilde{f}(\varepsilon)=0$ and for all $\rho \in \{-,0,+\}$ and all $ \overline{w}^{-}\in \phi^{-}\big( \mathscr{C}_n^{\rho}(\mu,2^{-R}\varepsilon)\big)$, $\overline{w}^{0}\in \phi^{0}\big( \mathscr{C}_n^{\rho}(\mu,2^{-R}\varepsilon)\big)$ and $ \overline{w}^{+}\in \phi^{+}\big( \mathscr{C}_n^{\rho}(\mu,2^{-R}\varepsilon)\big)$, 
\begin{equation*}\label{eq_lemma_upper_bound_combinatorial_on_phi_sigma}
\# \left\{ w \in \mathscr{C}_n^{\rho}(\mu,2^{-R}\varepsilon) \, : \, \phi^{\sigma}(w)=\overline{w}^{\sigma} \, , \, \forall \sigma \in \{-,0,+\} \right\} \leq e^{\widetilde{f}(\varepsilon) n+o(n)}.
\end{equation*}
\end{lemma}

\begin{proof}
The general idea of the proof is to show that a trajectory $w \in \mathscr{C}_n^{\rho}(\mu,2^{-R}\varepsilon)$ such that $\phi^{\sigma}(w)=\overline{w}^{\sigma}$ for all $\sigma \in \{-,0,+\}$, is fully characterised by the associated sequence $(t_i,\sigma_i)_{i=1}^L$. Consider
\begin{align*}
\begin{array}{cccc}
\Phi : & \left\{ w \in \mathscr{C}_n^{\rho}(\mu,2^{-R}\varepsilon) \, : \, \phi^{\sigma}(w)=\overline{w}^{\sigma} \, , \, \forall \sigma \in \{-,0,+\} \right\} & \longrightarrow & \bigcup \limits_{k=1}^{\lceil 18 \varepsilon n \rceil + 3} \big( \mathcal{P}_k(n) \times \{-,0,+\}^k \big) \\
& w & \longmapsto & \{ t_i(w)\}_{i=1}^{L(w)} \times \big(\sigma_i(w)\big)_{i=1}^{L(w)}
\end{array},
\end{align*}
where $\mathcal{P}_k(n)$ denotes the subsets of $\dbl 1,n \dbr$ of size $k$. Note that in this proof, we will write the dependence in $w$ for the quantities $t_i$, $\sigma_i$ and $L$ since we will have to compare such quantities across different trajectories. First of all, using the bound obtained in Lemma~\ref{lemma_bound_on_J_sigma}, we obtain for all $ w \in \mathscr{C}_n^{\rho}(\mu,2^{-R}\varepsilon)$,
\begin{align*}
L(w) & = \sum_{\sigma \in \{-,0,+\}} \#J^{\sigma}(w) \leq 18 \varepsilon n +3.
\end{align*}
Thus, the map $\Phi$ is well defined. Let us show that $\Phi$ is injective. Suppose that $w, \widehat{w} \in \mathscr{C}_n^{\rho}(\mu,2^{-R}\varepsilon)$ satisfy
\begin{enumerate}
\item $\phi^{\sigma}(w)=\overline{w}^{\sigma}=\phi^{\sigma} (\widehat{w})$ for all $\sigma \in \{-,0,+\}$, and
\item $\Phi(w)=\Phi(\widehat{w})$.
\end{enumerate} 
By Item~2, we have $L(w)=L(\widehat{w})$. We call this common quantity $L$. Item~2 further implies that $t_i(w)=t_i(\widehat{w})$ and $\sigma_i(w)=\sigma_i(\widehat{w})$ for all $i \in \dbl 1,L \dbr$. Therefore, $J^{\sigma}(w)=J^{\sigma}(\widehat{w})$ for all $\sigma \in \{-,0,+\}$. 
Fix $\sigma \in \{-,0,+\}$ and $h \in \dbl 1,\#J^{\sigma} \dbr$. Let $u^{h,\sigma}$ and $\widehat{u}^{h,\sigma}$ be the subwords defined in (\ref{eq_def_u_h_sigma}) for $w$ and $\widehat{w}$ respectively. Since $t_i(w)=t_i(\widehat{w})$ and $\sigma_i(w)=\sigma_i(\widehat{w})$, we have $t_{i^{\sigma}_h}(w)=t_{\widehat{i}^{\sigma}_h}(\widehat{w})$. Therefore, $|u^{h,\sigma}|=|\widehat{u}^{h,\sigma}|$. By construction of $\phi^{\sigma}(w)$ and $\phi^{\sigma}(\widehat{w})$ and using Item~1, we obtain 
\begin{align*}
u^{h,\sigma}&=\phi^{\sigma}(w)_{\left[h+\sum_{k=1}^{h-1}|u^{k,\sigma}| \, : \, h-1+\sum_{k=1}^{h}|u^{k,\sigma}| \right]} \\
& = \phi^{\sigma}(\widehat{w})_{\left[h+\sum_{k=1}^{h-1}|\widehat{u}^{k,\sigma}| \, : \, h-1+\sum_{k=1}^{h}|\widehat{u}^{k,\sigma}| \right]} = \widehat{u}^{h,\sigma}.
\end{align*}
Therefore, by the definition of the subwords $u^{h,\sigma}$ in (\ref{eq_def_u_h_sigma}), it follows that 
\begin{align*}
w_{\left[t_{i^{\sigma}_h}:(t_{i^{\sigma}_h+1}-1) \wedge n \right]}=\widehat{w}_{\left[t_{i^{\sigma}_h}:(t_{i^{\sigma}_h+1}-1) \wedge n \right]}.
\end{align*}
This holds for all $\sigma \in \{-,0,+\}$ and $h \in \dbl 1,\#J^{\sigma} \dbr$ and since the intervals $\left[t_{i^{\sigma}_h}:(t_{i^{\sigma}_h+1}-1) \wedge n \right]$ cover $\dbl 1,n \dbr$, we deduce that $w=\widehat{w}$. Therefore $\Phi$ is injective. From this injectivity we deduce that
\begin{align*}
\# \left\{ w \in \mathscr{C}_n^{\rho}(\mu,2^{-R}\varepsilon) \, : \, \phi^{\sigma}(w)=\overline{w}^{\sigma} \, , \, \forall \sigma \in \{-,0,+\} \right\} & \leq \# \bigcup \limits_{k=1}^{\lceil 18 \varepsilon n \rceil + 3} \big( \mathcal{P}_k(n) \times \{-,0,+\}^k \big) \\
& = \sum_{k=1}^{\lceil 18 \varepsilon n \rceil + 3} \binom{n}{k} 3^k.
\end{align*}
Since $\varepsilon<\frac{1}{40}$, we have $\lceil 18 \varepsilon n \rceil +3 \leq 20 \varepsilon n < \frac{n}{2}$. Thus, using the unimodality of the binomial coefficient, we obtain
\begin{align*}
\sum_{k=1}^{\lceil 18 \varepsilon n \rceil + 3} \binom{n}{k} 3^k & \leq 20 \varepsilon n \times 3^{20 \varepsilon n} \binom{n}{\lceil 20 \varepsilon n \rceil}.
\end{align*} 
Using the standard Stirling-type upper bound $\binom{n}{k} < \left( \frac{e n}{k} \right)^k$, we deduce that	
\begin{align*}
\# \left\{ w \in \mathscr{C}_n^{\rho}(\mu,2^{-R}\varepsilon) \, : \, \phi^{\sigma}(w)=\overline{w}^{\sigma} \, , \, \forall \sigma \in \{-,0,+\} \right\} \leq e n \times \left( \frac{3 e}{20 \varepsilon} \right)^{20 \varepsilon n}.
\end{align*}
The result follows by setting $\widetilde{f}(\varepsilon)=20 \varepsilon \log \left( \frac{3 e}{20} \right) - 20 \varepsilon \log \varepsilon$ for all $\varepsilon>0$.
\end{proof}
\begin{remark}\label{remark_interpretation_counting_traj_lemma}
The maps $\phi^{\sigma}$ induce a three-part decomposition of each trajectory: a central piece, an excursion, and a meander above $R$ (when $\sigma=+$) or below $-R$ (when $\sigma=-$). This mirrors the three-block decomposition introduced for typical trajectories in Section~\ref{section_lower_bound}. Lemma~\ref{lemma_upper_bound_combinatorial_on_phi_sigma} allows us to compare the set $\mathscr{C}_n^{\rho}(\mu,2^{-R}\varepsilon)$ with its subset of associated typical trajectories. It shows that enlarging our class from the typical trajectories to all trajectories in $\mathscr{C}_n^{\rho}(\mu,2^{-R}\varepsilon)$ introduces only the multiplicative factor $e^{\widetilde{f}(\varepsilon)n+o(n)}$ and therefore does not affect the exponential rate of decay.
\end{remark}

The next lemma gives an upper bound in terms of three quantities whose exponential rates of decay have been established in Section~\ref{section_regional_estimates}. 
\begin{lemma}[Upper bound for each class]\label{lemma_upper_bound_three_terms}
Let $C_1>0$ be as in Item~1.~of Lemma~\ref{lemma_image_og_phi_sigma_right_set}. Then there exists $f:\RR_+^* \to \RR$ with $\lim_{\varepsilon \to 0^+}f(\varepsilon)=0$ such that, for every $\rho \in \{-,+\}$, 
\begin{align}
& \PP\left( \mathscr{C}_n^{\rho}(\mu,2^{-R}\varepsilon) \right)  \notag \\
& \qquad \leq 
\begin{cases}
e^{f(\varepsilon) n + o(n)} \PP\Big( \Omega_{\overline{\alpha}_{0,n}}^{(0)}(\mu_0,C_1\varepsilon) \Big) \PP_{-\rho R} \left( \Omega_{\overline{\alpha}_{-\rho,n},\textup{exc}}^{-\rho} \right)  \PP_{\rho R} \left( \Omega_{\overline{\alpha}_{\rho,n},\textup{mea}}^{\rho} \right), & \text{if } \alpha_0>0, \\
e^{f(\varepsilon) n + o(n)} \PP_{-\sigma R} \left( \Omega_{\overline{\alpha}_{-\rho,n},\textup{exc}}^{-\rho} \right)  \PP_{\rho R} \left( \Omega_{\overline{\alpha}_{\rho,n},\textup{mea}}^{\rho} \right), & \text{if } \alpha_0=0,
\end{cases} \label{eq_lemma_upper_bound_three_terms_sigma_1}
\end{align}
and for $\rho=0$, 
\begin{align}
& \PP\left( \mathscr{C}_n^{0}(\mu,2^{-R}\varepsilon) \right) \notag \\
& \qquad  \leq
\begin{cases}
e^{f(\varepsilon) n + o(n)} \PP\Big( \Omega_{\overline{\alpha}_{0,n}}^{(0)}(\mu_0,C_1\varepsilon) \Big) \PP_{-R} \left( \Omega_{\overline{\alpha}_{-,n},\textup{exc}}^{-} \right)  \PP_{R} \left( \Omega_{\overline{\alpha}_{+,n},\textup{exc}}^{+} \right), & \text{if } \alpha_0>0, \\
e^{f(\varepsilon) n + o(n)} \PP_{-R} \left( \Omega_{\overline{\alpha}_{-,n},\textup{exc}}^{-} \right)  \PP_{R} \left( \Omega_{\overline{\alpha}_{+,n},\textup{exc}}^{+} \right), & \text{if } \alpha_0=0.
\end{cases} \label{eq_lemma_upper_bound_three_terms_0}
\end{align}
\end{lemma}

\begin{proof}
We begin by establishing an analogous inequality at the level of individual trajectories. More precisely, we show that there exists a constant $C>0$ such that for all $w \in \Omega_n^{(0)}(\mu,2^{-R}\varepsilon)$,
\begin{equation}\label{eq_lemma_comparaison_probabilities_individual_traj}
\mathfrak{p}(w) \leq C^{\varepsilon n} \prod_{\sigma \in \{-,0,+\}} \mathfrak{p}\big( \phi^{\sigma}(w) \big).
\end{equation}
By using the decomposition of $w$ into the subwords $u^{h,\sigma}$, we obtain
\begin{align*}
\mathfrak{p}(w) & \leq \prod_{\substack{\sigma \in \{-,0,+\}\, : \\ \#J^{\sigma} >0}} \prod_{h=1}^{\#J^{\sigma}} \mathfrak{p}(u^{h,\sigma}).
\end{align*} 
The uniform lower bound on the transition probabilities in (\ref{eq_def_p_star_R}) implies that, for each $\sigma \in \{-,0,+\}$ such that $\#J^{\sigma}>0$,
\begin{align*}
&\prod_{h=1}^{\#J^{\sigma}} \mathfrak{p}(u^{h,\sigma}) \\
& \qquad  \leq \frac{1}{p_*^{2(\#J^{\sigma}-1) + |b^{\sigma}|}} \left( \prod_{h=1}^{\#J^{\sigma}-1} \mathfrak{p}(u^{h,\sigma})P(u^{h,\sigma}_{-1},\xi^{h,\sigma}_1)P(\xi^{h,\sigma}_1,u^{h+1,\sigma}_{1}) \right) \mathfrak{p}(u^{\#J^{\sigma},\sigma})P(u^{\#J^{\sigma},\sigma}_{-1},b^{\sigma}_1)\mathfrak{p}(b) \\
& \qquad = \frac{1}{p_*^{2(\#J^{\sigma}-1) + |b^{\sigma}|}} \mathfrak{p}\big( \phi^{\sigma}(w) \big).
\end{align*}
By Lemma~\ref{lemma_bound_on_J_sigma} we have $\#J^{\sigma}-1 < 6 \varepsilon n$ and by Equation~(\ref{eq_lemma_bound_b_sigma}) we have $|b^{\sigma}| < 11 \varepsilon n$. Thus,
\begin{align*}
\frac{1}{p_*^{2(\#J^{\sigma}-1) + |b^{\sigma}|}} \leq \frac{1}{p_*^{23 \varepsilon n}} , \quad \sigma \in \{-,0,+\}.
\end{align*}
Therefore, for all $w \in \Omega_n^{(0)}(\mu,2^{-R}\varepsilon)$, 
\begin{align*}
\mathfrak{p}(w) & \leq \prod_{\substack{\sigma \in \{-,0,+\} \, : \\ \#J^{\sigma}>0}} \frac{1}{p_*^{\,23 \varepsilon n}} \, \mathfrak{p}\big( \phi^{\sigma}(w) \big).
\end{align*}
If $\#J^{\sigma}=0$, then $\phi^{\sigma}(w)=b^{\sigma}$. Since $|b^{\sigma}| \leq 23 \varepsilon n$, we deduce that $1 \leq \frac{1}{p_*^{23 \varepsilon n}} \mathfrak{p}(\phi^{\sigma}(w))$. Inequality~(\ref{eq_lemma_comparaison_probabilities_individual_traj}) then follows by setting $C=\big(\frac{1}{p_*}\big)^{69}$.

Let us now show Equation~(\ref{eq_lemma_upper_bound_three_terms_sigma_1}). Consider $\rho \in \{-,+\}$ and fix $w \in \mathscr{C}^{\rho}_n(\mu,2^{-R}\varepsilon)$. Suppose first that $\alpha_0>0$. By the three points of Lemma~\ref{lemma_image_og_phi_sigma_right_set}, we obtain
\begin{align*}
\phi^{0}(w) \in \Omega_{\overline{\alpha}_{0,n}}^{(0)}(\mu_0,C_1 \varepsilon) \quad , \quad \phi^{\rho}(w) \in \Omega^{\rho}_{\overline{\alpha}_{\rho,n},\textup{mea}} \quad \text{and} \quad \phi^{-\rho}(w) \in \Omega^{-\rho}_{\overline{\alpha}_{-\rho,n},\textup{exc}}.
\end{align*}
Thus, we deduce the following partition
\begin{align*}
\mathscr{C}_n^{\rho}(\mu,2^{-R}\varepsilon) = \bigcup_{\substack{\overline{w}^{0} \in \Omega^{(0)}_{\overline{\alpha}_{0,n}}(\mu_0,C_1 \varepsilon) \, , \\
\overline{w}^{-\rho} \in \Omega^{-\rho}_{\overline{\alpha}_{-\rho,n},\textup{exc}} \, , \\
\overline{w}^{\rho} \in \Omega^{\rho}_{\overline{\alpha}_{\rho,n},\textup{mea}}}} \Big\{ w \in \mathscr{C}_n^{\rho}(\mu,2^{-R}\varepsilon) \, : \, \phi^{\sigma}(w)=\overline{w}^{\sigma} \, , \, \forall \sigma \in \{-,0,+\} \Big\}.
\end{align*}
By decomposing $\PP\left(\mathscr{C}_n^{\rho}(\mu,2^{-R}\varepsilon)\right)$ in terms of its individual trajectories, we get
\begin{align}\label{eq_decomposition_of_Theta_sigma_with_phi}
\PP\left(\mathscr{C}_n^{\rho}(\mu,2^{-R}\varepsilon)\right) & = \sum_{w \in \mathscr{C}_n^{\rho}(\mu,2^{-R}\varepsilon)} \mathfrak{p}(w) \\
& = \sum_{\substack{\overline{w}^{0} \in \Omega^{(0)}_{\overline{\alpha}_{0,n}}(\mu_0,C_1 \varepsilon) \, , \\
\overline{w}^{-\rho} \in \Omega^{-\rho}_{\overline{\alpha}_{-\rho,n},\textup{exc}} \, , \\
\overline{w}^{\rho} \in \Omega^{\rho}_{\overline{\alpha}_{\rho,n},\textup{mea}}}} \sum_{\substack{w \in \mathscr{C}_n^{\rho}(\mu,2^{-R}\varepsilon) \, : \\ \phi^{\sigma}(w)=\overline{w}^{\sigma} \, , \, \forall \sigma \in \{-,0,+\}}} \mathfrak{p}(w). \notag
\end{align}
Using the upper bound (\ref{eq_lemma_comparaison_probabilities_individual_traj}), we obtain
\begin{align*}
& \sum_{\substack{w \in \mathscr{C}_n^{\rho}(\mu,2^{-R}\varepsilon) \, : \\ \phi^{\sigma}(w)=\overline{w}^{\sigma} \, , \, \forall \sigma \in \{-,0,+\}}} \mathfrak{p}(w) \\
& \qquad \qquad \qquad \leq \sum_{\substack{w \in \mathscr{C}_n^{\rho}(\mu,2^{-R}\varepsilon) \, : \\ \phi^{\sigma}(w)=\overline{w}^{\sigma} \, , \, \forall \sigma \in \{-,0,+\}}} C^{\varepsilon n} \prod_{\sigma \in \{-,0,+\}} \mathfrak{p}\big( \phi^{\sigma}(w) \big) \\
& \qquad \qquad \qquad = C^{\varepsilon n}\#\left\{ w \in \mathscr{C}_n^{\rho}(\mu,2^{-R}\varepsilon)\, : \, \phi^{\sigma}(w)=\overline{w}^{\sigma}  \, , \, \forall \sigma \in \{-,0,+\} \right\}\prod_{\sigma \in \{-,0,+\}} \mathfrak{p}\big( \overline{w}^{\sigma} \big).
\end{align*}
Then, using the combinatorial bound obtained in Lemma~\ref{lemma_upper_bound_combinatorial_on_phi_sigma}, we deduce that
\begin{align*}
\sum_{\substack{w \in \mathscr{C}_n^{\rho}(\mu,2^{-R}\varepsilon) \, : \\ \phi^{\sigma}(w)=\overline{w}^{\sigma} \, , \, \forall \sigma \in \{-,0,+\}}} \mathfrak{p}(w) \leq e^{f(\varepsilon) n + o(n)} \prod_{\sigma \in \{-,0,+\}} \mathfrak{p}\big( \overline{w}^{\sigma} \big),
\end{align*}
where $f(\varepsilon)=\widetilde{f}(\varepsilon)+\varepsilon \log C$. Plugging this back into (\ref{eq_decomposition_of_Theta_sigma_with_phi}) yields
\begin{align*}
& \PP\left(\mathscr{C}_n^{\rho}(\mu,2^{-R}\varepsilon)\right)\\
& \qquad \qquad \leq e^{f(\varepsilon) n + o(n)}\Bigg( \sum_{\overline{w}^{0} \in \Omega^{(0)}_{\overline{\alpha}_{0,n}}(\mu_0,C_1 \varepsilon)} \mathfrak{p}\big( \overline{w}^0 \big)\Bigg) \Bigg(\sum_{\overline{w}^{-\rho} \in \Omega^{-\rho}_{\overline{\alpha}_{-\rho,n},\textup{exc}}} \mathfrak{p}\big( \overline{w}^{-\rho} \big)\Bigg) \Bigg( \sum_{\overline{w}^{\rho} \in \Omega^{\rho}_{\overline{\alpha}_{\rho,n},\textup{mea}}} \mathfrak{p}(\overline{w}^{\rho})\Bigg) \\
& \qquad \qquad = e^{f(\varepsilon) n + o(n)} \PP \left(\Omega^{(0)}_{\overline{\alpha}_{0,n}}(\mu_0,C_1\varepsilon) \right)\PP_{-\rho R} \left(\Omega^{-\rho}_{\overline{\alpha}_{-\rho,n},\textup{exc}} \right)\PP_{\rho R} \left(\Omega^{\rho}_{\overline{\alpha}_{\rho,n},\textup{mea}} \right).
\end{align*}
This shows (\ref{eq_lemma_upper_bound_three_terms_sigma_1}) in the case where $\alpha_0>0$. Suppose now $\alpha_0=0$. By using points \textit{2.} and \textit{3.} of Lemma~\ref{lemma_image_og_phi_sigma_right_set} and the fact that $\phi^{0}(w) \in \Omega^{(0)}_{\overline{\alpha}_{0,n}}$, we obtain in a similar manner that
\begin{align*}
\PP\left(\mathscr{C}_n^{\rho}(\mu,2^{-R}\varepsilon)\right) \leq e^{f(\varepsilon) n + o(n)} \PP \left(\Omega^{(0)}_{\overline{\alpha}_{0,n}} \right)\PP_{-\rho R} \left(\Omega^{-\rho}_{\overline{\alpha}_{-\rho,n},\textup{exc}} \right)\PP_{\rho R} \left(\Omega^{\rho}_{\overline{\alpha}_{\rho,n},\textup{mea}} \right).
\end{align*}
Since $\Omega^{(0)}_{\overline{\alpha}_{0,n}}$ consists of all trajectories of length $\overline{\alpha}_{0,n}$ started from the origin, $\PP \big(\Omega^{(0)}_{\overline{\alpha}_{0,n}} \big)=1$. This shows (\ref{eq_lemma_upper_bound_three_terms_sigma_1}) in the case where $\alpha_0=0$.

Let us now turn to Equation~(\ref{eq_lemma_upper_bound_three_terms_0}). Fix $w \in \mathscr{C}^0_n(\mu,2^{-R}\varepsilon)$. Notice that in this case, point \textit{2.} of Lemma~\ref{lemma_image_og_phi_sigma_right_set} never applies and we have for all $\sigma \in \{-,+\}$,
\begin{align*}
\phi^{\sigma}(w) \in \Omega_{\overline{\alpha}_{\sigma,n},\textup{exc}}^{\sigma}.
\end{align*}
The rest of the proof follows the same steps as that of (\ref{eq_lemma_upper_bound_three_terms_sigma_1}) and we thereby obtain \eqref{eq_lemma_upper_bound_three_terms_0}.
\end{proof}

\subsection{Proof of the upper bound}\label{subsection_proof_of_the_upper_bound}

\noindent With the help of the results of Section~\ref{section_regional_estimates} and the decomposition given in Lemma~\ref{lemma_upper_bound_three_terms}, we are ready to show the upper bound. Let us first reformulate the exponential rates obtained in Propositions~\ref{proposition_exponential_rate_decay_I_DV} and \ref{proposition_exponential_rate_of_decay_exc_and_end} in a form better suited to prove Proposition~\ref{prop_upper_bound}.

\begin{lemma}[Upper bound at rate $\gamma$]\label{lemma_exponential_decay_with_gamma_n_upper}
Let $\gamma:\RR_+^* \to \RR_+^*$ be a function such that $\lim_{\varepsilon \to 0^+}\gamma(\varepsilon)=\overline{\gamma}$. Then, for all $\mu_0 \in \mathcal{P}(\ZZ)$,
\begin{align}
\limsup_{\varepsilon \to 0^+} \limsup \limits_{n \to \infty} \frac{1}{n} \log \PP \left( \ell_{2 \left \lceil \frac{\gamma(\varepsilon) n}{2} \right \rceil+1} \in B(\mu_0,\varepsilon) \right) & \leq -\overline{\gamma} I_{\textup{DV}}(\mu_0). \label{eq_lemma_I_DV_with_gamma_n_upper}
\end{align}
In addition, for all $\varepsilon>0$, all $R \geq R_{\overline{\ZZ}}(\varepsilon)$ and all $\rho \in \{-,+\}$,
\begin{align}
\limsup \limits_{n \to \infty} \frac{1}{n} \log \PP_{\rho R}\left(\Omega_{2 \left \lceil \frac{\gamma(\varepsilon) n}{2} \right \rceil+1,\textup{exc}}^{\rho} \right) & \leq -\gamma(\varepsilon) I_{\textup{Cr}}^{p_{\rho}}(0) + \gamma(\varepsilon) g(\varepsilon), \label{eq_lemma_exponential_rate_of_decay_exc_gamma_n_upper}\\
\limsup \limits_{n \to \infty} \frac{1}{n} \log \PP_{\rho R}\left( \Omega_{2 \left \lceil \frac{\gamma(\varepsilon) n}{2} \right \rceil + 1,\textup{mea}}^{\rho} \right) & \leq - \gamma(\varepsilon) \inf_{\rho x \in [0,1]} I_{\textup{Cr}}^{p_{\rho}}(x) + \gamma(\varepsilon) g(\varepsilon), \label{eq_lemma_exponential_rate_of_decay_final_gamma_n_upper}
\end{align}
where $g$ is as in Proposition~\ref{proposition_exponential_rate_of_decay_exc_and_end}. 
\end{lemma}

\begin{proof}
Observe that for all $\varepsilon>0$, $\lim_{n \to \infty} \frac{2 \left \lceil \frac{\gamma(\varepsilon) n}{2} \right \rceil+1}{n}=\gamma(\varepsilon)$. In addition, for all $k \in \NN$,
\begin{align*}
\left\{2 \left\lceil \frac{\gamma(\varepsilon) n}{2} \right \rceil+1 \, : \, n \geq k \right\} \subseteq \left\{ n \in \NN \, : \, n \geq 2 \left \lceil \frac{\gamma(\varepsilon) k}{2} \right \rceil+1 \right\}.
\end{align*}
The result then follows by the same argument as in the proof of Equation~(\ref{eq_lemma_I_DV_with_gamma_n_lower}) of Lemma~\ref{lemma_exponential_decay_with_gamma_n_lower}. Equations~(\ref{eq_lemma_exponential_rate_of_decay_exc_gamma_n_upper}) and (\ref{eq_lemma_exponential_rate_of_decay_final_gamma_n_upper}) are shown in the same way.
\end{proof}

\noindent We are now ready to give a proof of the upper bound.

\begin{proof}[Proof of Proposition~\ref{prop_upper_bound}] Our aim is to show that for all $\mu \in \mathcal{P}(\overline{\ZZ})$,
\begin{align*}
\lim \limits_{\varepsilon \to 0^+} \limsup \limits_{n \to \infty} \frac{1}{n} \log \PP \left( \ell_n \in B(\mu,\varepsilon) \right) \leq -I(\mu).
\end{align*}
Recall that $\{A^{\rho}\}_{\rho \in \{-,0,+\}}$ partitions $\overline{\ZZ}$. Hence, we obtain for all $\varepsilon>0$, $R \in \NN$ and $n\in \NN$, the following partition at the level of trajectories
\begin{align*}
\Omega_n^{(0)}(\mu,2^{-R}\varepsilon)=\bigcup_{\rho \in \{-,0,+\}} \mathscr{C}_n^{\rho}(\mu,2^{-R}\varepsilon).
\end{align*}
We thereby deduce that
\begin{align*}
\PP\left( \ell_n \in B(\mu,2^{-R}\varepsilon)\right) \leq 3 \max_{\rho \in \{-,0,+\}} \Big\{ \PP\left(\mathscr{C}_n^{\rho}(\mu,2^{-R}\varepsilon)\right) \Big\}.
\end{align*}
Therefore, 
\begin{equation}\label{eq_upper_bound_max_reduction_to_sigma}
\limsup \limits_{n \to \infty} \frac{1}{n} \log \PP\left( \ell_n \in B(\mu,2^{-R}\varepsilon)\right) \leq \max_{\rho \in \{-,0,+\}} \left\{ \limsup_{n \to \infty} \frac{1}{n} \log \PP\left(\mathscr{C}_n^{\rho}(\mu,2^{-R}\varepsilon)\right) \right\}.
\end{equation}
Consider now $\varepsilon>0$ sufficiently small and $R \in \NN$ and $n \in \NN$ sufficiently large so that the constraints given at the beginning of Section~\ref{subsection_decomp_and_stitch_traj_upper_bound} are satisfied. Lemma~\ref{lemma_upper_bound_three_terms} provides, for each $\rho \in \{-,0,+\}$, an upper bound on $\PP\left( \mathscr{C}_n^{\rho}(\mu,2^{-R}\varepsilon)\right)$ in the form of a product of three terms. Let us give exponential rates of decay for each term. Using that $\overline{\alpha}_{\rho,n}=2 \left \lceil \frac{\alpha_{\rho}+8 \varepsilon}{2} n \right \rceil + 1$, we obtain thanks to Lemma~\ref{lemma_exponential_decay_with_gamma_n_upper} with $\gamma(\varepsilon)=\alpha_{\rho}+8\varepsilon$, that 
\begin{align}
\limsup_{\varepsilon \to 0^+} \limsup \limits_{n \to \infty} \frac{1}{n} \log \PP \left( \Omega_{\overline{\alpha}_{0,n}}^{(0)}(\mu_0,C_1\varepsilon) \right) & \leq -\alpha_0  I_{\textup{DV}}(\mu_0), \label{eq_upper_bound_proof_technical_1} \\
\limsup \limits_{n \to \infty} \frac{1}{n} \log \PP_{-\rho R} \left( \Omega_{\overline{\alpha}_{-\rho,n},\textup{exc}}^{-\rho} \right) & \leq -(\alpha_{-\rho}+8\varepsilon) I_{\textup{Cr}}^{p_{-\rho}}(0) + (\alpha_{-\rho}+8\varepsilon) g(\varepsilon), \label{eq_upper_bound_proof_technical_2} \\
\limsup \limits_{n \to \infty} \frac{1}{n} \log \PP_{\rho R} \left( \Omega_{\overline{\alpha}_{\rho,n},\textup{mea}}^{\rho} \right) & \leq -(\alpha_{\rho}+8\varepsilon) \inf_{\rho x \in [0,1]} I_{\textup{Cr}}^{p_{\rho}}(x) + (\alpha_{\rho}+8\varepsilon)g(\varepsilon). \label{eq_upper_bound_proof_technical_3}
\end{align}
Given (\ref{eq_upper_bound_max_reduction_to_sigma}), let us compute each of the term within the $\max$ to obtain the upper bound. We start with the case where $\rho \in \{-,+\}$ and $\alpha_0>0$. Using Equation~(\ref{eq_lemma_upper_bound_three_terms_sigma_1}) of Lemma~\ref{lemma_upper_bound_three_terms}, we obtain
\begin{align*}
\limsup \limits_{n \to \infty} \frac{1}{n} \log \PP\left( \mathscr{C}_n^{\rho}(\mu,2^{-R}\varepsilon)\right) & \leq f(\varepsilon)  + \limsup \limits_{n \to \infty} \frac{1}{n} \log \PP\left( \Omega_{\overline{\alpha}_{0,n}}^{(0)}(\mu_0,C_1\varepsilon) \right) \\
& \;\;\, + \limsup \limits_{n \to \infty} \frac{1}{n} \log \PP_{-\rho R} \left( \Omega_{\overline{\alpha}_{-\rho,n},\textup{exc}}^{-\rho} \right) + \limsup \limits_{n \to \infty} \frac{1}{n} \log \PP_{\rho R} \left( \Omega_{\overline{\alpha}_{\rho,n},\textup{mea}}^{\rho} \right).
\end{align*}
Using the monotonicity of $\eta \mapsto \limsup_{n \to \infty} \frac{1}{n} \log \PP (\mathscr{C}_n^{\rho}(\mu,\eta))$ and Equations~(\ref{eq_upper_bound_proof_technical_2}) and (\ref{eq_upper_bound_proof_technical_3}) presented above, we obtain
\begin{align*}
\lim_{\eta \to 0^+} \limsup_{n \to \infty} \frac{1}{n} \log \PP\left( \mathscr{C}_n^{\rho}(\mu,\eta) \right)& = \lim_{R \to \infty} \limsup_{n \to \infty} \frac{1}{n} \log \PP\left( \mathscr{C}_n^{\rho}(\mu,2^{-R}\varepsilon) \right) \\
& \leq f(\varepsilon) + \limsup_{n \to \infty} \frac{1}{n} \log \PP \left( \Omega_{\overline{\alpha}_{0,n}}^{(0)}(\mu_0,C_1 \varepsilon) \right) \\
& \quad - (\alpha_{-\rho}+8 \varepsilon)I_{\textup{Cr}}^{p_{-\rho}}(0)-(\alpha_{\rho}+8\varepsilon) \inf_{\rho x \in [0,1]} I_{\textup{Cr}}^{p_{\rho}}(x) + (1+16 \varepsilon)g(\varepsilon).
\end{align*}
Taking $\limsup_{\varepsilon \to 0^+}$ on the right hand side, using the fact that $\lim_{\varepsilon \to 0^+}f(\varepsilon)=0$, $\lim_{\varepsilon \to 0^+}g(\varepsilon)=0$ and Equation~(\ref{eq_upper_bound_proof_technical_1}), we deduce that for each $\rho \in \{-,+\}$,
\begin{align*}
\lim \limits_{\varepsilon \to 0^+} \limsup \limits_{n \to \infty} \frac{1}{n} \log  \PP\left(\mathscr{C}_n^{\rho}(\mu,\varepsilon)\right) & \leq -\alpha_0 I_{\textup{DV}}(\mu_0) -\alpha_{-\rho} I_{\textup{Cr}}^{p_{-\rho}}(0)- \alpha_{\rho} \inf_{\sigma x \in [0,1]} I_{\textup{Cr}}^{p_{\rho}}(x) \\
& \leq -I(\mu).
\end{align*}
Let us now consider the case where $\rho \in \{-,+\}$ and $\alpha_0=0$. We obtain thanks to (\ref{eq_lemma_upper_bound_three_terms_sigma_1}), that
\begin{align*}
& \limsup \limits_{n \to \infty} \frac{1}{n} \log \PP\left( \mathscr{C}_n^{\rho}(\mu,2^{-R}\varepsilon)\right) \\
& \qquad \qquad \qquad \qquad \leq f(\varepsilon) + \limsup \limits_{n \to \infty} \frac{1}{n} \log \PP_{-\rho R} \left( \Omega_{\overline{\alpha}_{-\rho,n},\textup{exc}}^{-\rho} \right) + \limsup \limits_{n \to \infty} \frac{1}{n} \log \PP_{\rho R} \left( \Omega_{\overline{\alpha}_{\rho,n},\textup{mea}}^{\rho} \right).
\end{align*}
Similarly as before, we obtain when taking $R$ to infinity and then $\varepsilon$ to 0,
\begin{align*}
\lim \limits_{\varepsilon \to 0^+} \limsup \limits_{n \to \infty} \frac{1}{n} \log  \PP\left( \mathscr{C}_n^{\rho}(\mu,\varepsilon)\right) & \leq -\alpha_{-\rho}I_{\textup{Cr}}^{p_{-\rho}}(0) - \alpha_{\rho} \inf_{\rho x \in [0,1]} I_{\textup{Cr}}^{p_{\rho}}(x) \\
&  \leq -I(\mu).
\end{align*}
The case $\sigma =0$ is handled similarly, except that (\ref{eq_lemma_upper_bound_three_terms_0}) is used instead of (\ref{eq_lemma_upper_bound_three_terms_sigma_1}). To conclude, we have for every $\rho \in \{-,0,+\}$,
\begin{align*}
\lim \limits_{\varepsilon \to 0^+} \limsup \limits_{n \to \infty} \frac{1}{n} \log  \PP\left(\mathscr{C}_n^{\rho}(\mu,\varepsilon)\right) \leq -I(\mu).
\end{align*}
Finally, taking  $\varepsilon$ to 0 in (\ref{eq_upper_bound_max_reduction_to_sigma}), we deduce that
\begin{align*}
\lim \limits_{\varepsilon \to 0^+} \limsup \limits_{n \to \infty} \frac{1}{n} \log \PP\left( \ell_n \in B(\mu,\varepsilon)\right) \leq -I(\mu).
\end{align*}
Thus, we have shown the upper bound and Proposition~\ref{prop_upper_bound} is proved.
\end{proof}

\subsection*{Acknowledgements}
I would like to thank my supervisors, Tristan Benoist, Noé Cuneo, and Clément Pellegrini, for their guidance and support, and for their careful reading of this manuscript. I also thank Léo Daures, Pierre Petit and Ofer Zeitouni for fruitful discussions.

\begin{appendices}

\addtocontents{toc}{\setcounter{tocdepth}{1}}

\newpage

\section{Appendix}\label{appendix_A}

\subsection{Proof of Proposition~\ref{proposition_no_LDP_for_bounded_observables}}\label{subsection_appendix_proof_prop_no_LDP_bdd_obser}

\begin{proof}[Proof of Proposition~\ref{proposition_no_LDP_for_bounded_observables}]
Consider a simple random walk $(S_n)_{n \in \NN}$ with constant transition probability function $p \equiv \overline{p}>\frac{1}{2}$, which corresponds to a drift towards $+\infty$. The case $\overline{p}<\frac{1}{2}$ can be treated analogously, up to a change of sign. We construct a bounded observable $f:\ZZ \to \RR$ such that the sequence $(\ell_n(f))_{n \in \NN}$ does not satisfy a LDP. Set $c_1=1$ and define inductively
\begin{align*}
d_n \eqdef (n+1)c_n, \quad n \in \NN \qquad \text{and} \qquad c_n \eqdef d_{n-1}^2, \quad  n \in \NN \backslash \{1\}.
\end{align*}
Note that for all $n \in \NN$, we have $c_n < d_n < c_{n+1}$. Thus, the intervals $\dbl c_n,d_n \dbr$ do not overlap. In addition, we have $\lim_{n \to \infty} c_n=+\infty$ and $\lim_{n \to \infty} d_n = +\infty$. Define the bounded observable $f:\ZZ \to \RR$ by setting for all $x \in \ZZ$,
\begin{align*}
f(x) \eqdef \begin{cases}
1, & \text{if } x \in \bigcup_{n \in \NN} \dbl c_n,d_n \dbr, \\
0, & \text{otherwise.}
\end{cases}
\end{align*}
The function $f$ oscillates between 0 and 1, the blocks on which $f=1$ and $f=0$ become progressively longer. We will show that the sequence $(\ell_n(f))_{n \in \NN}$ does not satisfy a LDP by proving the strict inequality
\begin{align*}
\inf_{\varepsilon>0} \liminf_{n \to \infty} \frac{1}{n} \log \PP \left( |\ell_n(f)-1|<\varepsilon \right) < \inf_{\varepsilon>0} \limsup_{n \to \infty} \frac{1}{n} \log \PP \left( |\ell_n(f)-1|<\varepsilon \right).
\end{align*}
Indeed, by \cite[Theorem~4.1.18]{DZ_LDP}, this will imply that no LDP holds.

Fix $\varepsilon>0$ and choose $n>\frac{1}{\varepsilon}$. Consider the event 
\begin{align*}
A_n \eqdef \big\{S_k=k-1 \, , \, \forall k \in \dbl 1, c_n+1 \dbr \big\} \cap \big\{ S_k \geq c_n \, , \, \forall k \in \dbl c_n+1, d_n \dbr \big\}.
\end{align*}
Since the increments are bounded by 1, we have on $A_n$ that $S_k \in \dbl c_n , d_n \dbr$ for every $k \in \dbl c_n+1, d_n \dbr$. Hence $f(S_k)=1$ for those $k$, while $f(\cdot) \geq 0$ elsewhere, so
\begin{align*}
\ell_{d_n}(f)& =\frac{1}{d_n} \sum_{k=1}^{c_n} f(k-1) + \frac{1}{d_n} \sum_{k=c_n+1}^{d_n} f(S_k) \geq \frac{d_n-c_n}{d_n}.
\end{align*}
Since $d_{n}=(n+1)c_n$ and $n>\frac{1}{\varepsilon}$, we deduce that on $A_n$,
\begin{align*}
\ell_{d_n}(f)& \geq 1- \frac{1}{n+1} >1-\varepsilon.
\end{align*}
By inclusion of events, we obtain
\begin{align*}
\PP \left( \vert \ell_{d_n}(f) - 1 \vert < \varepsilon \right) & \geq \PP \big( A_n \big) \\
& = \overline{p}^{c_n}\PP\left( S_k \geq 0 \, , \, \forall k \in \dbl 1,d_n-c_n \dbr \right),
\end{align*}
where the equality follows from the Markov property and spatial homogeneity. Since $\overline{p}>\frac{1}{2}$ we obtain by \cite[p.~347, eq.~(2.8)]{Feller_intro_proba} that there exists a constant $C>0$ such that
\begin{align*}
\PP\left( S_k \geq 0 \, , \, \forall k \in \dbl 1,d_n-c_n \dbr \right) \geq \PP\left( S_k \geq 0 \, , \, \forall k \in \NN \right) \geq C.
\end{align*}
We deduce that
\begin{align*}
\limsup_{n \to \infty} \frac{1}{n} \log \PP\big(\vert\ell_n(f)-1 \vert< \varepsilon \big) & \geq \limsup_{n \to \infty} \frac{1}{d_n} \log \PP\big(\vert\ell_{d_n}(f)-1 \vert< \varepsilon \big) \\
& \geq \limsup_{n \to \infty} \frac{c_n}{d_n} \log \overline{p}.
\end{align*}
Since $d_n=(n+1)c_n$, it follows that $\lim_{n \to \infty} \frac{c_n}{d_n}=0$. Together with the trivial upper bound $\PP\big(\vert\ell_n(f)-1 \vert< \varepsilon \big) \leq 1$, we deduce that for all $\varepsilon>0$,
\begin{equation}\label{eq_limsup_eq_zero_in_counter_example}
\limsup_{n \to \infty} \frac{1}{n} \log \PP\big(\vert\ell_n(f)-1 \vert< \varepsilon \big) =0.
\end{equation}

Next, consider the event
\begin{align*}
B_n \eqdef \big\{S_{c_n}>d_{n-1}+\varepsilon c_n \big\}.
\end{align*}
If $B_n$ occurs, then because individual steps have size at most 1, for each $k \in \dbl \lceil (1-\varepsilon)c_n \rceil , c_n \dbr$ we have $S_k \in \dbl d_{n-1}+1,c_n-1 \dbr$. By definition of $f$, this implies that for all such $k$, we have $f(S_k)=0$. Thus, on the event $B_n$, we have
\begin{align*}
\ell_{c_n}(f) & =\frac{1}{c_n} \sum_{k=1}^{\lceil (1-\varepsilon)c_n \rceil-1}f(S_k) \leq \frac{\lceil (1-\varepsilon)c_n \rceil-1}{c_n} \leq 1-\varepsilon.
\end{align*}
Hence $\{\vert \ell_{c_n}(f)-1 \vert < \varepsilon \} \subseteq B_n^c$ so
\begin{align*}
\PP\big( \vert \ell_{c_n}(f)-1 \vert<\varepsilon \big) & \leq \PP\big( S_{c_n} \leq d_{n-1}+\varepsilon c_n \big).
\end{align*}
Since $\lim_{n \to \infty} d_n=+\infty$, for $n$ large enough we have $\frac{1}{d_{n-1}} \leq \varepsilon$. By definition $c_n=d_{n-1}^2$, so that $d_{n-1} \leq \varepsilon c_n$, and therefore $d_{n-1}+\varepsilon c_n \leq 2\varepsilon c_n$. Thus,
\begin{align*}
\PP\big( \vert \ell_{c_n}(f)-1 \vert<\varepsilon \big) & \leq \PP \left( \frac{S_{c_n}}{c_n} \leq 2 \varepsilon  \right).
\end{align*}
Taking $\liminf_{n \to \infty}$ and then using a strengthening of Cramér's theorem, see \cite[Corollary 2.2.19]{DZ_LDP}, we obtain
\begin{align*}
\liminf_{n \to \infty} \frac{1}{n} \log \PP\big( \vert \ell_{n}(f)-1 \vert<\varepsilon \big) & \leq \liminf_{n \to \infty} \frac{1}{c_n} \log \PP \left( \frac{S_{c_n}}{c_n} \leq 2 \varepsilon  \right) \\
& = \lim_{n \to \infty} \frac{1}{n} \log \PP \left( \frac{S_{n}}{n} \leq 2 \varepsilon  \right) \\
& \leq -\inf_{x \leq 2 \varepsilon} \Lambda^*(x),
\end{align*}
where $\Lambda^*$ is the Fenchel--Legendre transform of $\Lambda$, the cumulant generating function of the step distribution $(1-\overline{p}) \delta_{-1}+\overline{p} \delta_1$. Thus, 
\begin{align*}
\inf_{\varepsilon>0} \liminf_{n \to \infty}\frac{1}{n} \log \PP\big( \vert \ell_{n}(f)-1 \vert<\varepsilon \big) \leq -\inf_{x \leq 0} \Lambda^*(x).
\end{align*}
We know that the rate function $\Lambda^*$ is strictly convex and since $\overline{p}>\frac{1}{2}$, it is equal to zero only at $2\overline{p}-1>0$. Therefore, $-\inf_{x \leq 0} \Lambda^*(x)=-\Lambda^*(0)<0$. Using Equation~(\ref{eq_limsup_eq_zero_in_counter_example}), we deduce that
\begin{align*}
\inf_{\varepsilon>0} \liminf_{n \to \infty}\frac{1}{n} \log \PP\big( \vert \ell_{n}(f)-1 \vert<\varepsilon \big)<0=\inf_{\varepsilon>0} \limsup_{n \to \infty}\frac{1}{n} \log \PP\big( \vert \ell_{n}(f)-1 \vert<\varepsilon \big).
\end{align*}
Hence, the sequence $(\ell_n(f))_{n \in \NN}$ does not satisfy a LDP.
\end{proof}

\subsection{Restricting the starting point of the random walk to the origin}\label{subsection_appendix_restriction_starting_point_origin}

In this section, we give the proof of Lemma~\ref{lemma_restriction_to_0}, which states that we can restrict our attention to the random walk with the origin as starting point and which we recall here.

\begin{lemma}\label{lemma_restriction_to_0_appendix}
Suppose Equations (\ref{eq_prop_lower_bound_RL}) and (\ref{eq_prop_upper_bound_RL}) are satisfied for $m=0$, then the result holds for all $m \in \ZZ$.
\end{lemma}

\begin{proof}
Fix $\mu \in \mathcal{P}(\overline{\ZZ})$, $m \in \ZZ \backslash\{0\}$ and $\varepsilon>0$ and let $n > \frac{4|m|}{\varepsilon}$. Denote the empirical measure restricted to the segment $\dbl |m|+1,n \dbr$ by $\ell_{[|m|+1:n]} \eqdef \frac{1}{n-|m|}\sum_{j=|m|+1}^n \delta_{S_j}$. Let $s=1$ if $m>0$ and $s=-1$ if $m<0$. Suppose that 
\begin{align*}
S_j=m-s(j-1) , \quad j \in \dbl 1,|m|+1 \dbr \qquad \text{and} \qquad \ell_{[|m|+1:n]} \in B\left(\mu,\frac{\varepsilon}{2}\right).
\end{align*}
Using Equation~(\ref{eq_lemma_properties_KR_norm_of_proba}), one can verify that $\ell_n \in B(\mu,\varepsilon)$. Therefore,
\begin{align*}
\PP_m \big( \ell_n \in B(\mu,\varepsilon) \big) \geq \PP_m\left(S_1=m,S_2=m-s,\dots,S_{|m|+1}=0,\ell_{[|m|+1:n]} \in B\left(\mu,\frac{\varepsilon}{2}\right) \right).
\end{align*}
Consider $p_{*} \eqdef \inf \{p(k),1-p(k)\, : \, k \in \ZZ\}>0$, where the strict positivity follows from Hypothesis~\ref{hypothesis_on_transition_function_p}. Using Markov's property and bounding the first transitions in the inequality above from below by $p_*$, we obtain
\begin{equation}\label{eq_comparaison_m_and_0}
\PP_m \big( \ell_n \in B(\mu,\varepsilon) \big) \geq p_*^{|m|} \PP_0 \left(  \ell_{n-|m|} \in B\left(\mu,\frac{\varepsilon}{2} \right) \right).
\end{equation}
Taking the logarithm, dividing by $n$ and then taking $\liminf_{n \to \infty}$ on both sides, we obtain
\begin{align*}
\liminf_{n \to \infty} \frac{1}{n} \log \PP_m \big( \ell_n \in B(\mu,\varepsilon) \big) \geq \liminf_{n \to \infty} \frac{1}{n} \log \PP_0 \left(  \ell_{n} \in B\left(\mu,\frac{\varepsilon}{2} \right) \right).
\end{align*}
If the lower bound in Proposition~\ref{proposition_RL_functions} holds for $m=0$, then taking $\varepsilon$ to 0 in the above inequality yields
\begin{align*}
\lim_{\varepsilon \to 0^+} \liminf_{n \to \infty} \frac{1}{n} \log \PP_m \big( \ell_n \in B(\mu,\varepsilon) \big) \geq -I(\mu).
\end{align*}
Regarding the upper bound, we obtain by exchanging the roles of $m$ and $0$ in (\ref{eq_comparaison_m_and_0}) that
\begin{align*}
\left( \frac{1}{p_*} \right)^{|m|} \PP_0 \big( \ell_n \in B(\mu,\varepsilon) \big) \geq \PP_m \left(  \ell_{n-|m|} \in B\left(\mu,\frac{\varepsilon}{2} \right) \right).
\end{align*} 
Taking the logarithm, dividing by $n$ and taking $\limsup_{n \to \infty}$ on both sides, we obtain
\begin{align*}
\limsup_{n \to \infty} \frac{1}{n} \log \PP_0 \big( \ell_n \in B(\mu,\varepsilon) \big) \geq \limsup_{n \to \infty} \frac{1}{n} \log \PP_m \left(  \ell_{n} \in B\left(\mu,\frac{\varepsilon}{2} \right) \right).
\end{align*}
Finally, taking $\varepsilon$ to 0 in the above inequality yields
\begin{align*}
-I(\mu) \geq \lim_{\varepsilon \to 0^+} \limsup_{n \to \infty} \frac{1}{n} \log \PP_m \big( \ell_n \in B(\mu,\varepsilon) \big).
\end{align*}
Thus, if Proposition~\ref{proposition_RL_functions} holds for $m=0$, then it also holds for an arbitrary $m \in \ZZ$.
\end{proof}

\newpage

\section{List of Notations}\label{section_appendix_list_of_notations}

\renewcommand{\arraystretch}{1.3}

\begin{tabular}{llll}
\toprule
\textbf{Symbol} & \textbf{Description} & \textbf{} & \textbf{First use} \\
\midrule

\notation{$p_+,p_-,p_{\sigma}$}{Limiting probabilities}{$p_{\pm}=\lim_{k \to \pm \infty}p(k)$}{p.\pageref{eq_def_p_pm}}

\notation{$\ell_n$}{Empirical measure}{$\ell_n=\frac{1}{n}\sum \limits_{j=1}^n \delta_{S_j}$}{p.\pageref{eq_def_empirical_measure}}

\notation{$I$}{LDP rate function}{}{(\ref{eq_def_rate_fuction}),(\ref{eq_rate_function_other_form})}

\notation{$I_{\textup{DV}}$}{DV rate function}{$I_{\textup{DV}}(\mu)=\sup \limits_{u_k \geq 1} \displaystyle \sum \limits_{k \in \ZZ} \mu(k) \log \left( \frac{u_k}{(\Pi u)_k} \right)$}{(\ref{Donsker_Varadhan_our_case})}

\notation{$I_{\textup{Cr}}^{p_{\sigma}}$}{Cramér rate function}{$\frac{1-x}{2} \log \left( \frac{1-x}{2(1-p_{\sigma})} \right) + \frac{1+x}{2} \log \left( \frac{1+x}{2p_{\sigma}} \right)$}{\eqref{eq_I_Cr_Fenchel_Legendre},(\ref{eq_def_I_Cr_p})}

\notation{$d_{\overline{\ZZ}}$}{Metric on $\overline{\ZZ}$}{}{p.\pageref{eq_def_distance_Z_bar}}

\notation{$\| \mu \|$}{Kantorovich-Rubinstein norm}{}{(\ref{eq_def_BL_norm})}

\notation{$\Omega_n, \Omega_n^{(0)}$}{Trajectories of length $n$}{}{\eqref{def_Omega_n}}

\notation{$\Omega_{\textup{fin}}$}{Traj. of finite length}{}{p.\pageref{def_Omega_n}}

\notation{$\mathfrak{p}(w)$}{Probability of word $w$}{$\mathfrak{p}(w)=\PP_{w_1} \left(S_j=w_j \, ,\, \forall j \in \dbl 1,n \dbr \right)$}{(\ref{eq_def_probability_of_word})}

\notation{$A^{\sigma}=A_{R}^{\sigma}$}{Regions of $\overline{\ZZ}$}{}{p.\pageref{eq_def_A_sigma_R}}

\notation{$R_{\overline{\ZZ}}(\varepsilon)$}{}{$R_{\overline{\ZZ}}(\varepsilon)=\left \lceil \log_2 \left( \frac{1}{\varepsilon} \right) \right \rceil+1$}{(\ref{eq_def_R_Z_varepsilon})}

\notation{$\Omega_n^{(0)}(\mu,\varepsilon)$}{Traj. close to $\mu$}{$\ell_n(w)\in B(\mu,\varepsilon)$}{(\ref{eq_def_trajectories_close_to_mu})}

\notation{$\Omega_{n,\textup{exc}}^{\sigma}$}{Excursions in $A^{\sigma}$ of length $n$}{$w_1=\sigma R=w_n$ and $\sigma w_j \geq R$}{(\ref{eq_def_trajectories_excursions_above_R})}

\notation{$\Omega_{n,\textup{mea}}^{\sigma}$}{Meanders in $A^{\sigma}$ of length $n$}{$w_1=\sigma R$ and $\sigma w_j \geq R$}{(\ref{eq_def_trajectories_end_above_R})}

\notation{$\mathscr{C}_n^{\sigma}$}{Traj. in the class $\sigma$}{$w_n \in A^{\sigma}$}{\eqref{eq_def_class}}

\notation{$\mathscr{C}_n^{\sigma}(\mu,\varepsilon)$}{Traj. in the class $\sigma$ close to $\mu$}{$w_n \in A^{\sigma}$ and $\ell_n(w) \in B(\mu,\varepsilon)$}{\eqref{eq_def_class}}

\notation{$N^{\sigma},N_k$}{Occupation times}{$N^{\sigma}=\#\{j \in \dbl1,n \dbr : w_j \in A^{\sigma}\}$}{(\ref{eq_def_C_sigma_and_C_k})}

\notation{$R_{\mu_0}(\varepsilon)$}{}{}{(\ref{eq_def_R_mu_0_varepsilon})}

\notation{$\ell^0(w)$}{Restricted empirical measure}{$\ell^0(w)=\frac{1}{N^{0}(w)} \sum_{j : w_j \in A^0} \delta_{w_j}$}{\eqref{eq_def_restricted_empirical_measure}}

\notation{$G(\varepsilon)$}{}{}{(\ref{eq_for_lemma_G_varepsilon_to_zero})}

\notation{$\varepsilon^*$}{}{$G(\varepsilon)<1 \; \text{for all } \varepsilon \in (0,\varepsilon^*)$}{p.\pageref{eq_for_lemma_G_varepsilon_to_zero}}

\notation{$t^{\sigma}$}{Times for constr. of typ. traj.}{$t^{\sigma} \simeq \alpha_{\sigma} n$}{(\ref{eq_def_times_for_typical_traj})}

\notation{$\xi^{\sigma}(k),\chi^{\sigma}$}{Connecting words}{}{(\ref{eq_def_xi}),(\ref{eq_def_chi})}

\notation{$b$}{Ending word}{}{p.\pageref{lemma_existence_b_lower_bound}}

\notation{$\psi^{\sigma}$}{Stitching map}{$\psi^{\sigma}(v^{\textup{c}},v^{\textup{exc}},v^{\textup{mea}})=v^{\textup{c}}\xi^{\sigma}v^{\textup{exc}}\chi^{\sigma}v^{\textup{mea}}b$}{p.\pageref{def_map_psi_sigma}}

\notation{$E_{\textup{typ}}^{\sigma}$}{Typical trajectories}{}{\eqref{eq_def_typical_traj}}

\notation{$p_*$}{Lower bound on transitions}{$p_*=\inf_{k \in \ZZ}\{p(k),1-p(k)\}$}{(\ref{eq_def_p_star_R})}

\notation{$t_i,\sigma_i$}{Stopping times and regions}{$t_{i+1}=\inf\{k \in \dbl t_i+1,n\dbr : w_k \notin A^{\sigma_i}\}$}{p.\pageref{eq_def_t_and_sigma}}

\notation{$L$}{Nb. of cuts in the trajectory}{$L=\#\{i \in \NN : t_i \leq n\}$}{(\ref{eq_def_L})}

\notation{$J^{\sigma}$}{Index of subwords in $A^{\sigma}$}{$J^{\sigma}=\{i \in \dbl 1, L \dbr : \sigma_i=\sigma\}$}{(\ref{eq_def_J_sigma})}

\notation{$u^{h,\sigma}$}{Subwords in $A^{\sigma}$}{$u^{h,\sigma}=w_{[t_{i_h^{\sigma}} \, : \, (t_{i_h^{\sigma}+1}-1) \wedge n]}$}{(\ref{eq_def_u_h_sigma})}

\notation{$\xi^{h,\sigma}$}{Connecting words}{}{p.\pageref{lemma_existence_connecting_one_letter_word}}

\notation{$\overline{\alpha}_{\sigma,n}$}{}{$\overline{\alpha}_{\sigma,n}=2 \left \lceil \frac{\alpha_{\sigma}+8\varepsilon}{2}n \right \rceil +1$}{p.\pageref{lemma_overline_alpha_sigma_long_enough}}

\notation{$b^{\sigma}$}{Ending word}{}{p.\pageref{def_b_sigma}}

\notation{$\phi^{\sigma}$}{Map that stitches subwords}{$\phi^{\sigma}(w)=u^{1,\sigma}\xi^{1,\sigma}u^{2,\sigma} \cdots u^{\#J^{\sigma},\sigma}b^{\sigma}$}{p.\pageref{def_phi}}

\bottomrule
\end{tabular}

\end{appendices}
\newpage


\begin{thebibliography}{FdLF02}

\bibitem[BD96]{Bryc_Dembo}
W.~Bryc and A.~Dembo.
\newblock Large deviations and strong mixing.
\newblock {\em Annales de l'I.H.P. Probabilités et statistiques}, 32:549--569,
  1996.

\bibitem[BJV91]{Baxter_Jain_Varadhan}
J.~R. Baxter, N.~C. Jain, and S.~R.~S Varadhan.
\newblock Some familiar examples for which the large deviation principle does
  not hold.
\newblock {\em Communications on Pure and Applied Mathematics}, 44:911--923,
  1991.

\bibitem[Bog18]{Bogachev}
V.~V. Bogachev.
\newblock {\em Weak Convergence of Measures}.
\newblock American Mathematical Society, 2018.

\bibitem[Dau25]{Daures}
L.~Daures.
\newblock Large deviations for possibly reducible {M}arkov chains on discrete
  state spaces, 2025.
\newblock arXiv:2507.11166.

\bibitem[DE97]{Dupuis_Ellis}
P.~Dupuis and R.~S. Ellis.
\newblock {\em A Weak Convergence Approach to the Theory of Large Deviations}.
\newblock Wiley, 1997.

\bibitem[Din93]{Dinwoodie}
I.~H. Dinwoodie.
\newblock Identifying a {Large Deviation Rate Function}.
\newblock {\em The Annals of Probability}, 21(1):216--231, 1993.

\bibitem[DV75]{Donsker_varadhan}
M.~D. Donsker and S.~R.~S. Varadhan.
\newblock Asymptotic evaluation of certain {Markov} process expectations for
  large time, {I}.
\newblock {\em Communications on Pure and Applied Mathematics}, 28(1):1--47,
  1975.

\bibitem[DV76]{Donsker_varadhan_3}
M.~D. Donsker and S.~R.~S. Varadhan.
\newblock Asymptotic evaluation of certain {Markov} process expectations for
  large time, {III}.
\newblock {\em Communications on Pure and Applied Mathematics}, 29(4):389--461,
  1976.

\bibitem[DZ10]{DZ_LDP}
A.~Dembo and O.~Zeitouni.
\newblock {\em Large Deviations Techniques and Applications, Second Edition}.
\newblock Springer, 2010.

\bibitem[FdLF02]{Fayolle}
G.~Fayolle and A.~de~La~Fortelle.
\newblock {Entropy and the principle of large deviations for discrete-time
  Markov chains}.
\newblock {\em Problemy Peredachi Informatsii}, 38(4):121--135, 2002.

\bibitem[Fel91]{Feller_intro_proba}
W.~Feller.
\newblock {\em An Introduction to Probability Theory and Its Applications, vol.
  1, Third Edition}.
\newblock Wiley, 1991.

\bibitem[JW05]{Jian_Wu}
Y.~W. Jian and L.~M. Wu.
\newblock Large deviations for empirical measures of not necessarily
  irreducible countable markov chains with arbitrary initial measures.
\newblock {\em Acta Mathematica Sinica}, 21:1377–1390, 2005.

\bibitem[Kem74]{Kemperman_RW}
J.~H.~B. Kemperman.
\newblock The oscillating random walk.
\newblock {\em Stochastic Processes and their Applications}, 2(1):1--29, 1974.

\bibitem[Mun00]{Munkres}
J.~R. Munkres.
\newblock {\em Topology, Second Edition}.
\newblock Prentice Hall, 2000.

\bibitem[RAS15]{Rassoul}
F.~Rassoul-Agha and T.~Seppäläinen.
\newblock {\em A Course on Large Deviations with an Introduction to Gibbs
  Measures}.
\newblock American Mathematical Society, 2015.

\bibitem[Var18]{Varadhan_role_topology}
S.~R.~S. Varadhan.
\newblock The role of topology in large deviations.
\newblock {\em Expositiones Mathematicae}, 36(3):362--368, 2018.

\bibitem[VW09]{Vatutin}
V.~A. Vatutin and V.~Wachtel.
\newblock {Local probabilities for random walks conditioned to stay positive}.
\newblock {\em Probability Theory and Related Fields}, 143:177--217, 2009.

\bibitem[Zei01]{Zeitouni_RWRE}
O.~Zeitouni.
\newblock {Random Walks in Random Environment}.
\newblock In J.~Picard, editor, {\em Lectures on Probability Theory and
  Statistics}, pages 190--312. Springer, 2001.

\end{thebibliography}

\end{document}